\documentclass[%
numbook,%
twocolumn,%
]{svjour3}
\smartqed
\usepackage{amsmath, amssymb}
\usepackage{mathtools}
\mathtoolsset{showonlyrefs}
\usepackage{fontenc}
\usepackage[utf8]{inputenc}
\usepackage[english]{babel}
\usepackage{graphicx}
\usepackage{subcaption,xargs}
\usepackage{algorithm, booktabs}
\usepackage[noend]{algpseudocode}
\usepackage{enumitem,listings}
\usepackage{hyperref}
\usepackage{bbm}
\usepackage{environ,ifthen,xcolor}
\spnewtheorem{remark*}{Remark}{\bfseries}{\itseries}
\usepackage{xcolor}
\usepackage[backgroundcolor=blue!50!green!25!white,linecolor=blue!50!green]{todonotes}
\usepackage{listings,tikz,pgfplots,pgfkeys}
\usetikzlibrary{external}
\usetikzlibrary{spy,calc,external,arrows}
\usetikzlibrary{decorations.markings,positioning}
\spnewtheorem*{xproof}{}{\itshape}{\rmfamily}

\renewenvironment{proof}[1][\proofname]
 {\renewcommand\xproofname{#1}\xproof}
 {\endxproof}

\newcommand{\R}{\mathbb{R}}
\newcommand{\N}{\mathbb{N}}

\DeclareMathOperator{\dist}{dist}
\DeclareMathOperator{\pdist}{\mathbf{dist}}
\DeclareMathOperator{\grid}{\Gamma}
\DeclareMathOperator{\neighbor}{\mathcal N}

\newcommand{\tT}{\mathrm{T}}

\newcommand{\dx}{\,\mathrm{d}}
\newcommand{\TGV}{\operatorname{TGV}}
\newcommand{\TV}{\operatorname{TV}}

\newcommand{\tv}{\operatorname{tv}}
\newcommand{\IC}{\operatorname{IC}}
\newcommand{\zb}[1]{\mathbf{#1}}
\newcommand{\SPD}[1]{\mathcal P(#1)}
\newcommand{\Sym}[1]{\operatorname{Sym}}
\makeatletter
\newcommandx{\abs}[2][1=\@empty]{#1\lvert #2 #1\rvert}
\newcommandx{\norm}[3][1=\@empty,3=\@empty]{#1\lVert #2 #1\rVert_{#3}}
\makeatletter

\DeclareMathOperator{\distM}{dist}

\DeclareMathOperator{\grad}{grad}
\newcommand{\geodesic}[1]{\gamma_{\overset{\frown}{#1}}}
\newcommand{\MM}{\ensuremath{\mathcal{M}}}
\newcommand{\geoPV}[1]{\ensuremath{\gamma_{{#1}}}}
\newcommand{\RR}{\ensuremath{\mathbb{R}}}
\newcommand{\dotgeoPV}[1]{\ensuremath{\dot\gamma_{{#1}}}}
\newcommand{\geo}[1]{\ensuremath{\gamma_{\overset{\frown}{#1}}}} 
\graphicspath{{./images/}{./}}%
\hypersetup{pdfauthor={Ronny Bergmann, Jan Henrik Fitschen, Johannes Parsch, Gabriele Steidl},
	pdftitle={Priors with Coupled First and Second Order Differences for  Manifold-Valued Image Processing},
	pdfsubject={preprint},%
	pdfcreator = {pdflatex and TextMate},%
	pdfkeywords={},
	plainpages=false, pdfstartview=FitH, pdfview=FitH, pdfpagemode=UseOutlines,%
	bookmarksnumbered=true,bookmarksopen=false,bookmarksopenlevel=0,%
	colorlinks=true,linkcolor=black,citecolor=black,urlcolor=black%
}
\title{Priors with Coupled First and Second Order Differences for  Manifold-Valued Image Processing}
\titlerunning{Coupled First and Second Order Differences}
\author{Ronny Bergmann \and Jan Henrik Fitschen \and Johannes Persch \and Gabriele Steidl}
\authorrunning{R.~Bergmann, J.~H.~Fitschen, J.~Persch, G.~Steidl}
\institute{%
Ronny Bergmann \and Jan Henrik Fitschen \and Johannes Persch \and Gabriele Steidl \at Departement of Mathematics,
Technische Universität Kaiserslautern,
Paul-Ehrlich-Str.~31, 67663 Kaiserslautern, Germany
\\
\email{bergmann@mathematik.uni-kl.de},\\
\email{fitschen@mathematik.uni-kl.de},\\
\email{persch@mathematik.uni-kl.de},\\
\email{steidl@mathematik.uni-kl.de}.
\and
Gabriele Steidl \at Fraunhofer ITWM, Kaiserslautern, Germany
\\\email{steidl@mathematik.uni-kl.de.}%
}
\date{\today}
\begin{document}
\maketitle
\begin{abstract}
We generalize discrete variational models 
involving the infimal convolution (IC) of first and second order differences
and the total generalized variation (TGV)
to manifold-valued images. We propose both 
extrinsic and intrinsic approaches.
The extrinsic models are based on embedding the manifold into
an Euclidean space of higher dimension with ma\-ni\-fold constraints. 
An alternating direction methods of multipliers
can be employed for finding the minimizers. 
However, the components
within the extrinsic IC or TGV decompositions
live in the embedding space  which makes their
interpretation difficult.
Therefore we investigate two intrinsic appro\-aches:
for Lie groups, we employ the group action within the models;
for more general manifolds 
our IC model is based on recently developed absolute second order differences on manifolds,
while our TGV approach uses an approximation of the parallel transport by the pole ladder.
For computing the minimizers of the intrinsic mo\-dels we
apply gradient descent algorithms.
Numerical examples demonstrate  that our approaches work well for certain
manifolds.
\end{abstract}

\keywords{Infimal convolution, total generalized variation, higher order differences, manifold-valued images,  optimization on manifolds}
\subclass{\\49M15, 49M25,49Q20,68U10,56Y99}
%
%
\section{Introduction}

Variational models of the form
  \begin{equation}\label{general_model}
    \mathcal E(u) = \mathcal E_{\mathrm{data}} (u;f) + \alpha \mathcal E_{\mathrm{prior}}(u),
    \qquad \alpha > 0,
  \end{equation}
where $f$ is the given data set, \(\mathcal E_{\mathrm{data}} \) the data fitting
  term and \(\mathcal E_{\mathrm{prior}}\) the prior also known as regularization term,
  were applied for various tasks in image processing.
    In this paper, we restrict our attention to least squares data fitting terms.
    
  Starting with methods having first order derivatives in their prior like the
  total variation (TV)~\cite{ROF92},
	higher order derivatives were
  incorporated into the prior to cope with the staircasing effect
  caused by the TV regularization and to better adapt to specific applications.
  Besides additive coupling of higher order
  derivatives, see, e.g.,~\cite{PS13}, their infimal convolution
  (IC)~\cite{CL97} or the total generalized variation (TGV)~\cite{BKP10} were
  proposed in the literature.
  In many applications such as image denoising 
   IC~\cite{SS08,SST11}
   or TGV~\cite{BH14,BKP10,BV11} show better
  results than just the additive coupling.
  For discrete TGV versions we refer to \cite{SS08,SST11}.
  A preconditioned Douglas--Rachford algorithm can be used to efficiently compute the minimizer of the $\TGV$ penalized problem~\cite{BH15}. 
  An extension of $\TGV$ to vector-valued images with applications in color image restoration was given in~\cite{Bre14}.
 In~\cite{BBEFSS18,BEFSS15}, IC, resp. TGV, of motion  fields were
 successfully applied to strain analysis, in particular for the early detection of cracks
 in materials during tensile tests. 
  
 With the emerging possibilities to capture different modalities of data, image
  processing methods are transferred to the case where the measurements (pixels)
  take values on Riemannian manifolds. Examples are Interferometric Synthetic
  Aperture Radar (InSAR)~\cite{BH98,BRF00}
  with values on the circle \(\mathbb S^1\), directional data on the
  2-sphere~\(\mathbb S^2\),
  electron backscatter diffraction (EBSD)~\cite{BHJPSW10,GA10}
  with data on quotient manifolds of $\operatorname{SO}(3)$
  or diffusion tensor magnetic resonance imaging (DT-MRI)~\cite{FJ07},
  where each measurement is a symmetric positive definite \(3\times 3\) matrix.
	These are rather simple manifolds for which explicit expressions
	of their geodesic distance and exponential map are known.

  Recently, the discrete TV model has been generalized to
  Riemannian manifolds in an intrinsic way \cite{LSKC13,SC11,WDS14}.
	Note that finding a global minimizer of the optimization problem is 
	NP hard already for the case of the circle \(\mathcal M = \mathbb S^1\)
	\cite{CS13,Fi17}.
  In~\cite{BBSW16,BW15,BW16}, the model was extended to include
  second order differences, where the coupling of the first and second order differences
	was only realized in an additive manner. The approach is based on
	a proper generalization of absolute values of second order differences
	to the manifold-valued setting.
	For the special case of DT-MRI, i.e., symmetric positive definite matrices of
  size \(3\times 3\), another approach using tensor calculus resulting in the
  Frobenius norm instead of a distance on the Riemannian manifold
  was investigated in~\cite{SSPB07} and extended to a
  TGV approach in~\cite{VBK13}. 
  Numerical analysis for $\mathbb S^2$-valued functions was also established in \cite{Alouges97}.
  
  In this paper, we generalize discrete variational models with
  least squares data term and IC, resp. TGV prior to the manifold-valued setting.
  We derive extrinsic and intrinsic approaches.
  The extrinsic models which generalize the first order model in \cite{RTKB14,RWTKB14} 
  have the drawback that the decomposition components of IC and TGV 
  live in the higher dimensional embedding space which makes their interpretation
  difficult. Therefore we propose two intrinsic approaches.
  For Lie groups as  the circle $\mathbb S^1$ 
  or the special orthogonal group~\(\operatorname{SO}(3)\),
  we incorporate the group operation within the IC and TGV models which lead to
  decompositions within the manifold.
  For more general Riemannian manifolds, our so-called Midpoint IC approach relies
  on the generalization of the absolute value of the second order
  difference by the distance of its center point from a geodesic joining the
  two other points~\cite{BBSW16}. 
  Our TGV model is based on the approximation of the parallel transport by the pole ladder \cite{LP14}.
	Note that the pole ladder mimics the parallel transport exactly for symmetric Riemannian manifolds
	all our numerical examples belong to.
	It leads to a decomposition within the tangent bundle of the manifold.
  We acknow\-ledge, that in parallel to our work an axiomatic TGV model 
	for manifold-valued images was developed by Bredies et al.~\cite{BHSW17}
  which was only available for the revised version of this manuscript.
  The first version of our paper contained an extrinsic TGV approach, 
  a TGV approach for Lie groups as well as a remark on an intrinsic approach by the Schild's ladder.
  This remark was extended in the final version, where we replaced the Schild's ladder by the Pole ladder, 
  since the later one  is an exact scheme for parallel transport in symmetric spaces.
  However, our approach is different from those in ~\cite{BHSW17}. 
  As suggested in our original remark, we work on the tangent bundle, while they work on the manifold itself.
  In contrast to our isotropic model, the authors in~\cite{BHSW17} propose an anisotropic one 
  using parallel transport or its approximation by Schild's ladder.
  Moreover,Moreover they focus on a cyclic proximal point algorithm, while we derive a gradient descent method.
  For more details see Remark~\ref{differences}.
  
  In the extrinsic case we choose an alternating
  direction method of multipliers (ADMM) for finding a (local) minimizer of the functionals.
  For the intrinsic models we smooth the functionals so that a gradient descent
  algorithm can be applied. 
  
  Various numerical examples show the denoising
  potential for images with values on the
  \begin{itemize}
 \item spheres $\mathbb S^d$, $d=1,2$, which includes the important case of  cyclic (phase) data;
 \item special orthogonal group $\mathrm{SO}(3)$;
 \item symmetric positive definite $r\times r$ matrices $\SPD{r}$.
\end{itemize}
The explicit expressions required for our computations are given in the Appendix~\ref{sec:app}.
The first two kind of manifolds are compact, while $\SPD{r}$ is an open convex cone in $\mathbb R^{r,r}$. 
  
  We developed the extrinsic IC model and the Midpoint IC approach for manifold-valued images 
	in the SSVM conference paper~\cite{BFPS17}
  and were invited to submit a full journal paper to JMIV.
  The current paper extends the SSVM paper significantly by models for Lie groups as well as all
  extrinsic and intrinsic TGV approaches.
  
  The outline of the paper is as follows: In Section~\ref{sec:euclid},
	we recall the discrete variational models for denoising gray-values images which we
  want to generalize.
	In Section~\ref{sec:ext_mani}, we propose the extrinsic models
	and comment how the ADMM algorithm can be adapted to these models.
	Unfortunately, dealing with manifolds requires to install  certain preliminaries.
	This is briefly done in Section~\ref{sec:mani} and 
	maybe skipped if the reader is familiar with the notation on manifolds.
	We propose an intrinsic Midpoint IC model and a TGV model based on the pole ladder
	in Section \ref{sec:int_mani}. 
	In Section \ref{sec:int_lie}, 
	we follow another idea driven by the group operation to set up intrinsic models for Lie groups.
	Section \ref{sec:ini_alg},
 shows how minimizers of the (smoothed)
  intrinsic models
	can be computed via a gradient descent algorithms and provides the necessary gradients.
	Numerical examples are presented in Section~\ref{sec:Num}.
	The paper finishes with conclusions in
  Section~\ref{sec:Concl}.
  Various technical details for the computation on manifolds are postponed to the appendix.

\section{Models for Real-Valued Images}  \label{sec:euclid}
In this section, we briefly reconsider models with priors containing first and second order
differences for gray-value images, where the focus is on the coupling of first and second order terms.
To keep the technicalities simple,
we just rely on gray-value images, but the approach can be simply generalized
to images with values in an Euclidean space as, e.g.,
RGB images.

Let
\begin{equation}
\grid\coloneqq\{1,\ldots,N_1\}\times\{1,\ldots,N_2\}
\end{equation}
denote the pixel grid of an image of size~\(N_1 \times N_2\)
and \(N\coloneqq N_1 N_2\).
We address grid points by $i = (i_1,i_2)$.
Let \(u\colon\grid\to\mathbb R\) be a gray-value image.
As data fitting term we focus on 
\begin{align} \label{data_disr}
\mathcal E_{\mathrm{data}} (u;f) \coloneqq  \frac{1}{2}\lVert f-u\rVert_2^2,
\end{align}
where the images are considered columnwise reshaped into  vectors.

To set up the different priors we need first and second order differences.
By $D_x u$ we denote
the forward differences in $x$-direction with Neumann (mirror) boundary conditions
\begin{align*}
(D_x u)_{i}\coloneqq
\begin{cases}
u_{i+(1,0)} - u_i & \mathrm{if} \;  i+(1,0) \in \grid, \\
0 & \mathrm{otherwise},
\end{cases}
\end{align*}
and analogously in $y$-direction.
Then
\begin{align}\label{not:1stg}
	\nabla \coloneqq \begin{pmatrix} D_x \\ D_y \end{pmatrix}
\end{align}
serves as discrete gradient and $\nabla u:\Gamma \rightarrow \mathbb R^2$.
For mappings ~$\xi \colon\grid \rightarrow \R^s$ we introduce the mixed norm
\begin{equation} \label{mixed_norm}
\lVert \xi \rVert_{2,1}
\coloneqq
\sum_{i \in \grid} | \xi_{i} |, \quad 
| \xi_{i} | \coloneqq \left( \xi_{i,1}^2 + \ldots \xi_{i,s}^2 \right)^\frac{1}{2}.
\end{equation}
We define the discrete TV regularizer by
\begin{equation}\begin{split}\label{intro:discrTV}
\TV(u) : = \lVert\nabla u\rVert_{2,1} =
\sum_{i \in\grid}
\Big(\sum_{j \in\neighbor(i)} \lvert u_{j}-u_{i}\rvert^2 \Big)^\frac{1}{2},
\end{split}
\end{equation}
where \(\neighbor(i)\coloneqq \{ i+(0,1), i+(1,0) \}\cap\grid\)
denotes the forward neighbors of pixel $i\in \Gamma$.
The backward difference $\widetilde D_x u$ in $x$-direction
is given by
\begin{align}
(\widetilde D_x u)_{i} \coloneqq
\begin{cases}
u_{i} - u_{i-(1,0)} & \mathrm{if} \  i \pm (1,0) \in \grid, \\
0 & \mathrm{otherwise},
\end{cases}
\end{align}
and similarly in  $y$-direction.
The choice of zero at the right boundary of the backward difference becomes clear in \eqref{TGV_discr}.
We will apply backward differences to vectors $\xi:\Gamma \rightarrow \mathbb R^2$
in two forms
\begin{align}\label{not:1sttil}
\widetilde \nabla \coloneqq \begin{pmatrix}
	\widetilde D_x & 0\\
	 \widetilde D_y & 0 \\
	0&\widetilde D_x\\
	0& \widetilde D_y
	\end{pmatrix},
	\quad
	\widetilde \nabla_S \coloneqq \begin{pmatrix}
	\widetilde D_x &0\\
	\frac{1}{2} \widetilde D_y & \frac{1}{2} \widetilde D_x \\
	0& \widetilde D_y
	\end{pmatrix}.
\end{align}
We define central second order differences in $x$-direction 
\begin{align} \label{laplace}
D_{xx} \coloneqq \widetilde D_x D_x,
\end{align}
i.e.,
\begin{align}\label{not:2nd}
(D_{xx} u)_{i} \coloneqq
\begin{cases}
u_{i-(1,0)} - 2 u_{i}+ u_{i+(1,0)}\!\!\!\!\!\!\!\!&  \\& \mathrm{if} \ i \pm (1,0) \in \grid, \\
0 & \mathrm{otherwise},
\end{cases}
\end{align}
and mixed second order differences
\begin{align}\label{dxy}
	D_{xy} \coloneqq \widetilde D_y D_x,
\end{align}
and analogously for the other directions.
Then a $\TV_2$ regularizer 
can be defined by the Frobenuis norm of the Hessian of u,
\begin{align}\label{second_disrc}
&\TV_2(u)
\coloneqq  \lVert \widetilde \nabla \nabla u \rVert_{2,1}\\
&= \sum_{i \in \grid}
\left(
|D_{xx} u|_{i}^2 + |D_{yy} u|_{i}^2  
+ |D_{xy}u|_{i}^2 + |D_{yx}u|_{i}^2 
\right)^\frac{1}{2}.
\end{align}
The infimal convolution (IC) of two functions 
\\
$F_i \colon \R^N \rightarrow \R \cup \{+\infty\}$, $i=1,2,$ is defined by
\begin{equation}\label{def_IC}
(F_1 \square F_2)(u) \coloneqq \inf_{u=v+w} \{F_1(v) + F_2(w)\} .
\end{equation}
If $F_i$, $i=1,2,$ are proper, convex, lower semi-continuous and $F_i(u) = F_i(-u)$,
then  $F_1 \square F_2$ is also proper, convex, lower semi-continuous and
the infimum is attained~\cite{Ro70,SST11}.

We consider two common ways to incorporate first and second order information into the prior, namely in an additive way
and by IC.  The corresponding priors look for $\beta \in (0,1)$ as follows:
\begin{enumerate}
	\item \textbf{Additive Coupling}
	\begin{align} \label{add_coup}
	\TV_{1\wedge2}(u) \coloneqq \beta \TV(u) + (1- \beta) \TV_2(u).
	\end{align}
	\item \textbf{Infimal Convolution}
	\begin{align} \label{IC_coup}
	\IC(u)
	\coloneqq \min_{u=v+w} \left\{ \beta \TV(v) + (1- \beta) \TV_2(w) \right\},
	\end{align}
	\end{enumerate}
The IC model is related to TGV of order two given by:
	\begin{enumerate}
	\item[3.]  \textbf{Total Generalized Variation}
	\begin{align}\label{TGV_discr}
	\TGV(u) \coloneqq \min_\xi \left\{ \beta \lVert\nabla u - \xi\rVert_{2,1} + (1-\beta) \lVert\widetilde \nabla_S \xi\rVert_{2,1}\right\},
	\end{align}
\end{enumerate}
In contrast to the IC prior, the TGV prior does not require the
	computation of second order differences.
The relation to the IC model, which by \eqref{intro:discrTV} and \eqref{second_disrc} 
	can be rewritten as 
	\begin{align} 
	\IC(u)
	= \min_{w} \left\{\beta \lVert\nabla (u-w) \rVert_{2,1} + (1-\beta) \lVert \widetilde \nabla \nabla w\rVert_{2,1}\right\}
	\end{align}
appears if $\xi \colon\grid\to\R^2$ in \eqref{TGV_discr} has the form
$\xi = \nabla w$ for some \(w\colon\grid\to\R\).
Then  both models differ only in the use of the nonsymmetric or symmetric backward difference operator.	
	
For the IC and TGV models we are interested in the corresponding decompositions:
\\[1ex]
 \textbf{IC Decomposition} ($u = v+w$)
\begin{align}\label{eq:sec_order_ic_ext_vw}
	E_{\IC} (v,w)
	&\coloneqq
	\frac{1}{2}\lVert f-v-w\rVert_2^2 
	\\
	&+ \alpha \bigl(
	\beta \TV(v) + (1- \beta) \TV_2(w)
	\bigr).
	\end{align}
\textbf{TGV Decomposition} ($\nabla u = \tilde \xi + \xi$)
\begin{align}\label{eq:tv_ext_ua}
E_{\TGV} (u,\xi)
	&\coloneqq
	\frac{1}{2}\lVert f-u\rVert_2^2  \\
	&+\alpha\big( \beta \lVert \underbrace{\nabla u - \xi}_{\tilde \xi} \rVert_{2,1} + (1-\beta) \lVert\widetilde \nabla \xi \rVert_{2,1} \big).
	\end{align}
For details on the discrete models we refer to \cite{SST11}.
In various applications, the individual IC components $v$ and $w$ are 
of interest, e.g., in motion separation~\cite{HK14} or early detection of cracks
in materials during tensile tests~\cite{BBEFSS18,BEFSS15}.
In the Euclidean setting, tools from convex analysis can be applied for finding minimizers of the functionals
including algorithms based on duality theory as the ADMM.

\section{Extrinsic Models for Manifold-Valued Images}\label{sec:ext_mani}

\begin{figure*}[ht]
\centering
  \begin{subfigure}{.49\textwidth}\centering
    \includegraphics{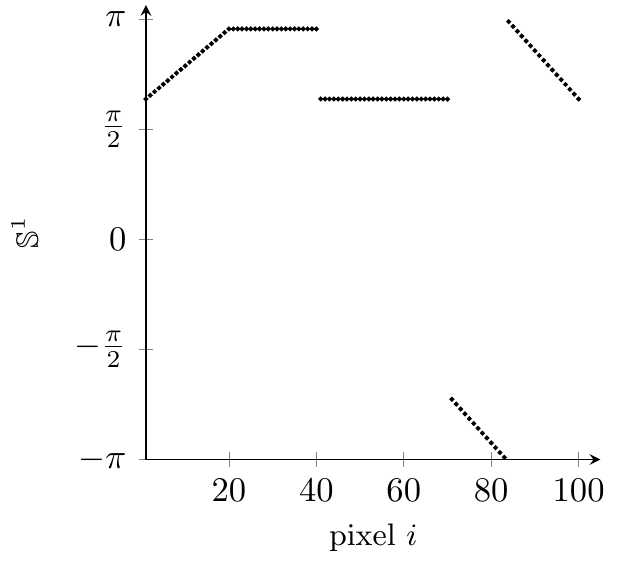}
		\caption{Intrinsic data representation.}\label{ictgv:fig:ex:ic:para}
  \end{subfigure}
  \begin{subfigure}{.49\textwidth}\centering
     \includegraphics{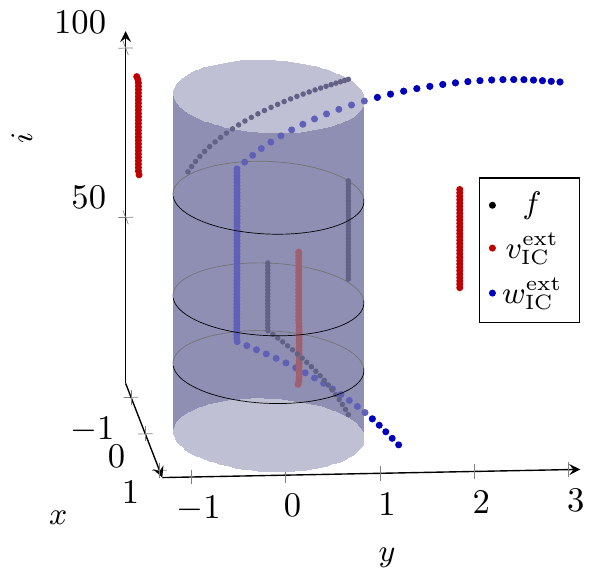}
		\caption{Embedded signals.}\label{ictgv:fig:ex:ic:emb}    
  \end{subfigure}
\caption{ 
\textbf{Extrinsic IC model}: Decomposition of a piecewise geodesic signal~$f$ with values in $\mathbb S^1$.
(a) Original signal~$f$ determined by its angles in $[-\pi,\pi)$. 
(b) Decomposition of $f$ into its piecewise constant part $v_{\mathrm{IC}}^{\mathrm{ext}}$ (red) and piecewise linear 
part $w_{\mathrm{IC}}^{\mathrm{ext}}$ (blue) which can only be depicted in the embedding space $\R^2$ of $\mathbb S^1$.
}\label{fig:ex:ic}
\end{figure*}

The simplest idea to generalize the gray-value models to images
$u: \Gamma \rightarrow {\mathcal M}$ having values in a manifold ${\mathcal M}$
is to embed the manifold into an Euclidean space. 
Recall that by Whitney's theorem \cite{Whi36} every smooth $d$-dimensional
manifold can be smoothly embedded into an Euclidean space of
dimension~$n=2d$.
Moreover, by  Nash's theorem \cite{Nas56}  every Riemannian manifold
can be isometrically embedded into an Euclidean space of suitable
dimension.
Assuming that $\mathcal M$ is embedded
in $\mathbb R^n$, we can just 
apply the three Euclidean models from the previous section, which we denote by
\begin{equation}
{\mathcal E}_*, \quad * \in \{{\mathrm{ADD}},\IC,\TGV\},
\end{equation}
	for $u \in \mathbb R^{nN}$
with the constraint that the image values have to lie in the manifold.
In other words, we are interested in
\begin{equation}
{\mathcal E}_*^{\mathrm{ext}} (u) \coloneqq {\mathcal E}_* (u) + \iota_{\mathcal{M}^N}(u), 
\quad * \in \{\mathrm{ADD},\IC,\TGV\},
\end{equation}
and for the IC and TGV decompositions in
\begin{align}
&E_{\IC}^{\mathrm{ext}} (v,w) \coloneqq E_{\IC} (v,w) + \iota_{\mathcal{M}^N}(v+w),\\
&E_{\TGV}^{\mathrm{ext}} (u,\xi) \coloneqq E_{\TGV} (u,\xi) + \iota_{\mathcal{M}^N}(u),
\end{align}
where 
$\iota_{\mathcal{M}^N}$ denotes the indicator function of the product manifold ${\mathcal{M}^N}$.
Due to the manifold constraints, the models are in general no longer convex.
Exceptions are manifolds which are closed convex sets in the embedding space as, e.g., (the closure of) $\SPD{r}$.
If $\mathcal M$ is closed, we directly
get the existence of a global minimizer by the coercivity and lower
semi-conti\-nuity of the functional.
For the squared $\ell_2$--TV model an extrinsic approach was
given in \cite{RTKB14,RWTKB14} with a sketch how it can be generalized for the additive model.
The extrinsic IC model was discussed in our conference paper \cite{BFPS17}.

To minimize ${\mathcal E}_{\mathrm{ADD}}^{\mathrm{ext}} (u)$, $E_{\IC}^{\mathrm{ext}} (v,w)$ and $E_{\TGV}^{\mathrm{ext}} (u,\xi)$
we apply an alternating direction method of multipliers (ADMM)  \cite{GM76,GM75} in the form given in \cite{BSS16}.
For details we refer to \cite{persch2018}.
As additional step to the Euclidean setting ADMM 
requires the orthogonal projections of elements from the embedding space onto the manifold.
For the manifolds  in our numerical example we notice the following:
\begin{itemize}
\item[-] For $\mathbb S^d \subset \mathbb R^{d+1}$,
the projection is just the normalization of the vector with respect to the Euclidean norm in $\mathbb R^{d+1}$.
\item[-]
For $\mathrm{SO}(3)$, the authors of \cite{RTKB14,RWTKB14} suggested to embed the $\mathrm{SO}(3)$ into~$\mathbb R^9$.
Then the projection requires the singular value decomposition of the matrix in $\mathbb R^{3,3}$ we want to project.
In this paper, we prefer an embedding of  $\mathrm{SO}(3)$  into~$\mathbb R^4$ via the quaternion representation, see Appendix \ref{sec:app}.
This reduces the dimension of the problem and the projection is again just a normalization.
\item[-]
For $\SPD{r} \subset \mathbb R^{n}$, \(n=\frac{r(r+1)}{2}\), the orthogonal projection approach is in general not possible
since the manifold is an \emph{open} cone of $\mathbb R^{n}$.
A numerical remedy would be to project onto the closed cone
and add a small parameter to the eigenvalues of the resulting matrix to make it positive definite.
Often also a convex barrier function as $-\log \det$ is added instead of the indicator function
to stay in the manifold, see, e.g., \cite{Jarre2000}. 
However, in this paper, we apply only intrinsic approaches to images with values in $\SPD{r}$.
\end{itemize}

\begin{remark} (Convergence of ADMM)
If $\mathcal{M}$ is a closed, convex set in $\mathbb R^n$, then the algorithm converges by standard arguments.
For spheres and the $\mathrm{SO}(3)$, convergence is observed numerically, but cannot be guaranteed theoretically.

The convergence of the ADMM for special non-convex functionals was recently
addressed in~\cite{WYZ15}.
Unfortunately, the assumptions of that paper do not fit into our setting:
More precisely, Assumption 2 in~\cite{WYZ15} would require with respect to our setting that the range
of~$(\nabla^\tT, I)^\tT$
is a subset of the range of the identity matrix
which is clearly  not the case.

A possibility to circumvent theoretical convergence problems would be to consider a smoothed version
of the functional such that
it becomes Lipschitz differentiable and a gradient reprojection algorithm can be applied.
For such an algorithm convergence was shown in~\cite{ABS13} for functions 
satisfying a Kurdyka-\L ojasiewicz property.
However, in our experiments, the algorithm shows a bad convergence behavior numerically such
that we prefer ADMM. \hfill $\Box$
\end{remark}

The following example motivates our efforts to find intrinsic IC and TGV decompositions.

\begin{example}\label{ex:S1Signal:ExtIC}
We apply  the extrinsic IC decomposition \eqref{eq:sec_order_ic_ext_vw} with parameters $\alpha=0.03$,
$\beta=\frac13$
to the phase-valued (noise-free) 
signal \(f\) of length \(N_1=100\). In Fig.~\ref{fig:ex:ic} (left) the signal is given by its angles in $[-\pi,\pi)$.
It consists of a linearly increasing line, two constant parts divided by a small jump and a
decreasing part. The second „jump“ in the signal is smaller than it occurs, 
since the shorter arc on the circle is the one “wrapping” around (going over \(\pm\pi\)). 
The “jump” at \(t=82\) is just only due to the representation system.
Embedding each pixel \(f_i\in\mathbb S^1\) into \(\mathbb R^2\) yields the black 
signal in \(\mathbb R^3\) shown in Fig.~\ref{fig:ex:ic} (right).
Since the model decomposes 
\(f \approx v_{\mathrm{IC}}^{\mathrm{ext}} + w_{\mathrm{IC}}^{\mathrm{ext}}\) 
within the embedding space $\mathbb R^2$, 
the components
can only be visualized within $\mathbb R^3$, in particular not in the left plot.
Still, 
\(v_{\IC}^{\mathrm{ext}}\) is piecewise constant and contains the jumps, 
and~\(w_{\IC}^{\mathrm{ext}}\) is continuous. 
\end{example}

\section{Preliminaries on Manifolds}\label{sec:mani}

Let $\mathcal M$ be a
connected, complete $d$-dimensional
Riemannian manifold.
By \(T_x\mathcal M\) we denote the tangent space of $\mathcal M$
at \(x\in\mathcal M\) with the Riemannian
metric~$\langle\cdot,\cdot\rangle_x$ and corresponding norm~$\| \cdot \|_x$.
Further, let $T\MM$ be the tangent bundle of $\MM$. 
By
$\operatorname{dist} \colon \mathcal M \times \mathcal M
\rightarrow \mathbb R_{\ge 0}$
we denote the geodesic distance on \(\mathcal M\).
Let $\mathcal M^N$ be the product or \(N\)-fold power manifold
with product distance
\begin{equation}
\pdist^2(x,y) \coloneqq \Big( \sum_{j=1}^N \operatorname{dist}^2(x_j,y_j) \Big)^\frac{1}{2}.
\end{equation}
Let~\(\gamma_{\overset{\frown}{x,y}}:[0,1] \rightarrow {\mathcal M}\) 
be a (not necessarily shortest)
geodesic connecting~\(x,y\in\mathcal M\).
We will also use the notation
$
\gamma(x,y;t) \coloneqq \gamma_{\overset{\frown}{x,y}} (t) 
$
to address points on the curve.
Further, we apply the notation $\geoPV{x;\xi}$ to characterize the geodesics 
by its starting point $\geoPV{x;\xi}(0) = x$ and direction $\dotgeoPV{x;\xi}(0)=\xi \in T_x{\mathcal M}$.
Note that the geodesic $\gamma_{\overset{\frown}{x,y}}$ is unique
on manifolds with nonpositive
curvature. Simply connected, complete Riemannian manifolds of nonpositive sectional curvature are called Hadamard manifolds. 
Examples are the manifold of positive definite matrices or hyperbolic spaces.
The exponential map $\exp_x\colon T_x\MM \to\MM$ is defined by 
\begin{equation}
\exp_x(\xi) \coloneqq \geoPV{x;\xi}(1).
\end{equation}
Since $\mathcal M$ is connected and complete, we know by the  Hopf-Rinow theorem~\cite{HR31} that the exponential map
is indeed defined on the whole tangent space.
The exponential map realizes a local diffeomorphism 
from a neighborhood $\mathcal{D}_T(0_{x})$ of the origin $0_{x}$ of $T_{x}\MM$ 
into a neighborhood of $x \in \MM$.
More precisely,
extending the geodesic $\geoPV{x;\xi}$ from $t=0$ to infinity
is either minimizing $\dist(x,\geoPV{x;\xi}(t))$ all along or up to a finite time $t_0$ 
and not any longer afterwards.
In the latter case, $\geoPV{x;\xi}(t_0)$ is called cut point 
and the set of all cut points of all geodesics starting from $x$ 
is the cut locus $\mathcal{C}(x)$. 
This allows to define the inverse exponential map, also known as logarithmic map as
\begin{equation} \label{star}
\log_{x} \coloneqq \exp_{x}^{-1}\colon \MM \backslash {\mathcal C} (x) \to T_{x}\MM.
\end{equation}
Then the Riemannian distance between $x,y\in\MM$, for $y\notin \mathcal{C}(x)$, can be written as 
\begin{equation} \label{log_dist}
\dist(x,y) = \langle \log_x(y),\log_x(y)\rangle_{x}^{\frac{1}{2}} = \lVert\log_x(y)\rVert_x.
\end{equation} 
Let $F\colon\MM\to\mathcal{N}$  be a smooth mapping between manifolds and $\xi \in T_x\MM$.
The mapping $DF(x)[\xi]$ from the set of smooth functions on a neighborhood of $x$ to $\RR$
given by 
\begin{equation}
\bigl(DF(x)[\xi]\bigr)f\coloneqq  \xi(f\circ F)
\end{equation}
is a tangent vector in  $T_{F(x)} \mathcal{N}$ 
and the linear mapping
\begin{equation}\label{eq:differential}
DF(x)\colon T_x\MM \to T_{F(x)}\mathcal{N},\quad \xi\mapsto DF(x)[\xi]
\end{equation}
the differential of $F$ at $x \in \MM$. 
Let $F\colon\MM_1\to\MM_2$ and $G\colon\MM_2\to\MM_3$ 
be two smooth mappings. 
Then the differential of their concatenation $G\circ F$ applied to $\xi\in T_x\MM_1$ 
is given by the \emph{chain rule}
\begin{equation}
D(G\circ F)(x)[\xi] = D G\bigl( F(x)\bigr)\bigl[DF(x)[\xi]\bigr].
\end{equation}
For a function $f\colon\MM\to\RR$ the Riemannian gradient $\grad_{\MM}$ is defined by
\begin{equation}
\langle\grad_{\MM}f(x), \xi\rangle_x \coloneqq Df(x)[\xi],\text{ for all }\xi\in T_x\MM.
\end{equation}
A mapping $\mathcal{R}_y\colon\MM \rightarrow \MM$ 
on a Riemannian manifold $\MM$
is called \emph{geodesic reflection} at $x\in \MM$ if
\begin{equation}
\mathcal{R}_x(x) = x \quad \text{and} \quad D(\mathcal{R}_x)(x) = -I.
\end{equation}
A connected Riemannian manifold $\MM$ is (globally)
symmetric if the geodesic reflection at any point $x \in \MM$
is an isometry of $\MM$.
All manifolds considered in this paper are symmetric ones.

Let ${\mathcal X}(\MM)$ be the set of smooth vector fields on $\MM$.
Given a curve $\gamma\colon[0,1]\to\MM$, we denote by ${\mathcal X}(\gamma)$ the set of smooth vector fields along $\gamma$, 
i.e., $X\in{\mathcal X}(\gamma)$ is a smooth mapping $X\colon[0,1]\to T\MM$ with $X(t)\in T_{\gamma(t)}\MM$.
A vector field $X\in{\mathcal X}(\gamma)$ 
is called parallel to $\gamma\colon[0,1]\to\MM$ 
if the covariate derivative along $\gamma$ fulfills
$\frac{D}{\dx t}X = 0$  for all $t\in[0,1]$.
We define the \emph{parallel transport} of a tangent vector $\xi\in T_x\MM$ to $T_y\MM$  by 
\begin{equation}
P_{x\to y} \xi \coloneqq X(1),
\end{equation}
where $X\in\mathcal{X}(\geo{x,y})$ is the vector field parallel to a minimizing geodesic $\geo{x,y}$ with $X(0)=\xi$.
There exist analytical expressions of the parallel transport only for few manifolds as spheres
or positive definite matrices. However, the parallel transport can be locally approximated, e.g., 
by Schild's ladder~\cite{EPS72,KMN2000}
or by the the pole ladder~\cite{LP14Sch}.
In this paper, we focus on the pole ladder since the approximation is  \emph{exact} in symmetric Riemannian manifolds~\cite{Pen18}.
Given $x,y\in\MM$, 
the pole ladder transports $\xi\in T_x\MM$ to $\zeta\in T_y\MM$ in four steps, cf.~Fig.~\ref{fig:pole+schild} left:
\begin{enumerate}
	\item take the mid point between $x$ and $y$, $c \coloneqq \gamma(x,y;\frac{1}{2})$;
	\item map $\xi$ onto the manifold by the exponential map, $e \coloneqq \exp_x(\xi)$;
	\item evaluate the geodesic between $e$ and $c$ at 2, i.e. $p \coloneqq \gamma(e,c;2)$;
	\item lift the end point to the tangent space of $y$ by the logarithmic map and multiply with $-1$ to get 
	$\zeta = P_{x\to y}^{\mathrm{P}}(\xi)\coloneqq-\log_{y}(p)$.	
\end{enumerate}
In summary, the transported vector is given by
\begin{equation}\label{pre:eq:pole}
P_{x\to y}^{\mathrm{P}}(\xi)\coloneqq -\log_{y}\Bigl(\gamma\Bigl(\exp_{x}(\xi),\gamma\bigl(x,y;\tfrac{1}{2}\bigr);2\Bigr)\Bigr)\in T_y\MM.
\end{equation}
For comparison, Schild's ladder transports as follows, cf.~Fig.~\ref{fig:pole+schild} right:
\begin{enumerate}
	\item map $v$ to the manifold by the exponential map, $e \coloneqq \exp_x(\xi)$;
	\item take the mid point between $y$ and $e$, $c \coloneqq \gamma(y,e;\frac{1}{2})$;
	\item evaluate the geodesic between $x$ and $c$ at 2, $p \coloneqq \gamma(x,c;2)$;
	\item lift the point $p$ to the tangent space of $y$ with the logarithmic map, $w\coloneqq \log_{y}(p)$.
\end{enumerate}
The transported vector is given by
\begin{equation}\label{pre:eq:schild}
P_{x\to y}^{\mathrm{S}}(\xi)\coloneqq \log_{y}\Bigl(\gamma\Bigl(x,\gamma\bigl(y,\exp_{x}(\xi);\tfrac{1}{2}\bigr);2\Bigr)\Bigr)\in T_y\MM.
\end{equation}
\begin{figure*}[t]
	\centering
	\includegraphics{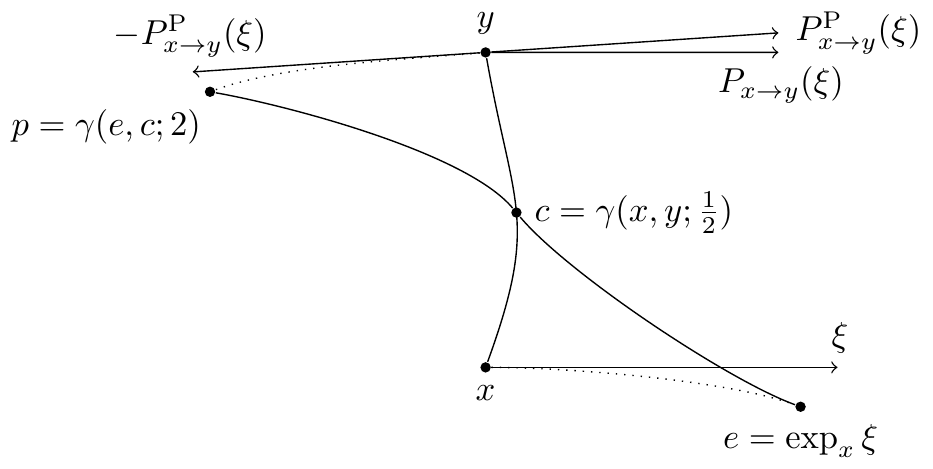}
	\hspace{0.2cm}
	\includegraphics{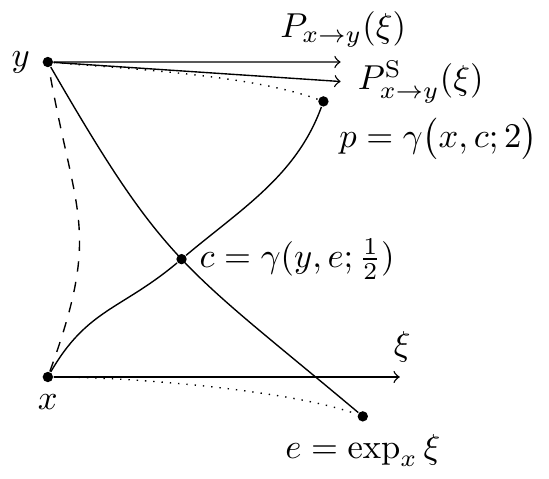}
	\caption{Illustration of pole ladder (left) and Schild's ladder (right) for the approximation of $P_{x\to y} \xi$.}\label{fig:pole+schild}
\end{figure*}%
In our minimization algorithms, we will need the Riemannian gradient of special functions,
in particular of those appearing in the pole ladder \eqref{pre:eq:pole}.
This gradient can be computed for symmetric Riemannian manifolds using the theory of Jacobi fields.
The following lemma collects the final results which can be partially found 
in~\cite{BBSW16,BHSW17,Carmo92}. 
For the complete proof we refer to~\cite{persch2018}.
\begin{lemma}\label{pre:differentials}
Let $\MM$ be a symmetric Riemannian ma\-ni\-fold and 
$F$ one of the functions i) - v) listed below together with 
the coefficient maps $\alpha: \mathbb R \rightarrow \mathbb R$ and parameters $T$.
Then the differential $DF(x)$ at $x \in \MM$ is given for all $\xi \in T_x \MM$ by 
	\begin{equation}\label{alles}
	D F (x) [\xi] = \sum_{k=1}^d \langle \xi, \Xi_k(0)\rangle_{x} \alpha(\kappa_k) \Xi_k(T),
	\end{equation}
where $\{\Xi_k\}_{k=1}^d$ denotes a parallel transported orthogonal frame along the geodesic $\gamma$
with $\gamma(0) = x$ and $\gamma (1) = y$ depending on $F$.
Further, the frame diagonalizes the Riemannian curvature tensor 
$R(\cdot,\dot\gamma)\dot{\gamma}$ with respective eigenvalues $\kappa_k$, $k = 1,\dots,d$.
The functions $F$ and $\alpha$ are given as follows:
	\begin{enumerate}
		\item[i)] For $F: = \exp_{\cdot} (u)$, we have $T=1$,  $y \coloneqq \exp_{x} u$ and 
		\begin{equation}
		\alpha(\kappa) \coloneqq \begin{cases}
		\cosh(\sqrt{-\kappa})\quad&\kappa <0,\\		
		1\quad&\kappa =0,\\
		\cos(\sqrt{\kappa})\quad&\kappa >0.
		\end{cases}
		\end{equation}
		\item[ii)] \ For $F  \coloneqq \log_\cdot (y)$, we have $T=0$ and
		\begin{equation}
		\alpha(\kappa) \coloneqq \begin{cases}
		-\sqrt{-\kappa}\frac{\cosh(\sqrt{-\kappa})}{\sinh(\sqrt{-\kappa})}\quad&\kappa <0,\\		
		-1\quad&\kappa =0,\\
		-\sqrt{\kappa}\frac{\cos(\sqrt{\kappa})}{\sin(\sqrt{\kappa})}\quad&\kappa >0.
		\end{cases}
		\end{equation}
		
		\item[iii)]\ For $F \coloneqq \log_y (\cdot)$, we have  $T=1$ and
		\begin{equation}
		\alpha(\kappa) \coloneqq \begin{cases}
		\frac{\sqrt{-\kappa}}{\sinh(\sqrt{-\kappa})}\quad&\kappa <0,\\		
		1\quad&\kappa =0,\\
		\frac{\sqrt{\kappa}}{\sin(\sqrt{\kappa})}\quad&\kappa >0.
		\end{cases}
		\end{equation}
		\item[iv)] For $F\coloneqq \geodesic{\cdot,y}(\tau)$, we have $T=\tau$ and
		\begin{equation}
		\alpha(\kappa) \coloneqq \begin{cases}
		\frac{\sinh\bigl(\sqrt{-\kappa}(1-\tau)\bigr)}{\sinh(\sqrt{-\kappa})}\quad&\kappa <0,\\		
		1-\tau\quad&\kappa =0,\\
		\frac{\sin\bigl(\sqrt{\kappa}(1-\tau)\bigr)}{\sin(\sqrt{\kappa})}\quad&\kappa >0.
		\end{cases}
		\end{equation}
		\item[v)] For $F \coloneqq \geodesic{y,\cdot}(\tau)$, we have $T=1-\tau$ and
		\begin{equation}
		\alpha(\kappa) \coloneqq \begin{cases}
		\frac{\sinh(\sqrt{-\kappa}\tau)}{\sinh(\sqrt{-\kappa})}\quad&\kappa <0,\\		
		\tau\quad&\kappa =0,\\
		\frac{\sin(\sqrt{\kappa}\tau)}{\sin(\sqrt{\kappa})}\quad&\kappa >0.
		\end{cases}
		\end{equation}
\item[vi)]\ Finally, we obtain for $F \coloneqq 	\exp_x (\cdot)$  with
\begin{equation}
		\alpha(\kappa) = \begin{cases}
		\frac{\sinh(\sqrt{-\kappa})}{\sqrt{-\kappa}}\quad&\kappa <0,\\		
		1&\kappa =0,\\
		\frac{\sin(\sqrt{\kappa})}{\sqrt{\kappa}}\quad&\kappa >0,\\
		\end{cases}
		\end{equation}
		and $T=1$
		that the differential $D F (u)$ of $F$ at $u \in T_x \MM$ is given by \eqref{alles},
		where we have to replace $x \in \MM$ by $u \in T_x \MM$ and to set $y \coloneqq \exp_x u$.
\end{enumerate}
\end{lemma}

The adjoint operator $(D F)^* (x)\colon T_{F(x)}\MM\to T_x\MM$ of \eqref{alles} is given by
\begin{equation}\label{pre:eq:adjoint}
(D F)^* (x) [w] = \sum_{k=1}^d \langle w,\Xi_k\rangle_{F(x)} \alpha_k\xi_k, \quad w\in T_{F(x)}\MM.
\end{equation}
%
\section{Intrinsic Models for Manifold-Valued Images}\label{sec:int_mani} 
%
In this section, we develop intrinsic variational models to process
manifold-valued images~\(u\colon\mathcal G\to\mathcal M\).
Instead of the data term \eqref{data_disr} we use
\begin{align}\label{data_mani}
\mathcal{E}_{\mathrm{data}}^{\mathrm{int}}(u;f) = \frac{1}{2} \pdist^2 (f,u).
\end{align}

\subsection{First Order Differences}
%
We define forward differences in $x$-direction by
\begin{equation} \label{eq:gradUManifoldLog}
	(D_x^{\mathrm{int}} u)_{i}
	\coloneqq
	\begin{cases}
	\log_{u_{i}}u_{i+(1,0)}\quad&\mathrm{if} \; i+(1,0)\in\grid,\\	0\quad&\text{otherwise},
	\end{cases}
\end{equation}
and analogously in $y$-direction.
Then we define
\begin{align}\label{not:1stg}
	D^{\mathrm{int}} \coloneqq \begin{pmatrix} D_x^{\mathrm{int}} \\ \nabla_y^{\mathrm{int}} \end{pmatrix}
\end{align}
as discrete gradient.
As counterpart of \eqref{mixed_norm} we introduce for  
$\xi = (\xi_i)_{i \in \Gamma}$ 
with \(\xi_i \in (T_{u_{i}}\mathcal{M})^s\)
the expression
\begin{equation} \label{mixed_mani}
\lVert \xi \rVert_{2,1,u} \coloneqq \sum_{i\in\grid}\left( \lVert \xi_{i,1} \rVert_{u_i}^2 + \ldots + \lVert \xi_{i,s} \rVert_{u_i}^2 \right)^\frac{1}{2}.
\end{equation}
Having \eqref{log_dist} in mind, the TV regularizer for manifold-valued images becomes
\begin{align} \label{TV_mani}
\TV^{\mathrm{int}}(u)
&\coloneqq
\lVert \nabla^{\mathrm{int}} u \rVert_{2,1,u}\\
&= \sum_{i \in\grid}
\Big(
\sum_{j \in\neighbor(i)}
\dist^2(u_{i},u_{j})
\Big)^\frac{1}{2}. 
\end{align}
The model 
$
\mathcal{E}_{\mathrm{ data}}^{\mathrm{int}}(u;f) + \alpha \operatorname{TV}^{\mathrm{int}}(u)
$
was already considered in \cite{LSKC13,WDS14}.
Recently, several attempts have been made to translate concepts from convex analysis to the  manifold-valued setting
and it turns out that a rich theory of convex functions can be built in Hadamard manifolds,
for an overview see, e.g., ~\cite{Bac14}.
Then the functional is convex so that various algorithms as, 
e.g., the cyclic proximal point algorithm can be proved to converge, see \cite{WDS14}.

\subsection{Second Order Differences via Midpoints of Geodesics}
%
To incorporate second order differences into the functional is not straightforward 
since there is no general definition of second order differences for manifold-valued data.
We emphasize that we do not speak about Hessians of real-valued functions living on a manifold. 
In our case, the differences
are taken with respect to ${\Gamma}$.
If the manifold is in particular a Lie group, additions can be replaced by group operations
which we will consider in the next section. 
In this section, we adopt the definition of the absolute value of second order differences from~\cite{BBSW16}.
Obser\-ving that in the Euclidean case the absolute second order difference
  of $x_1,x_1,x_3 \in \mathbb R^d$ can be rewritten as
  \( |x_1-2x_2+x_3| = 2 |\frac{1}{2}(x_1+x_3)-x_2|\),
  we define a counterpart for
  $x_1,x_2,x_3 \in \mathcal M$  as 
	\begin{align}
    \mathrm{d}_2(x_1,x_2,x_3)
    \coloneqq
    \min_{c \in\mathcal C_{x_1,x_3}}
     \operatorname{dist}(c,x_2),
   \end{align}
	where \(\mathcal C_{x_1,x_3}\) 
	is the set of mid
  points~\(\gamma_{\overset{\frown}{x_1,x_3}}(\frac{1}{2})\)
  of all geodesics joining $x_1$ and $x_3$.
Similarly,  second order mixed
differences were defined for \(x_i\in\mathcal M, i=1,\ldots,4,\) in~\cite{BW16}:
   \[
    \mathrm{d}_{1,1}(x_1,x_2,x_3,x_4)
    \coloneqq
    \min_{c \in\mathcal C_{x_1,x_3}, \tilde c\in\mathcal C_{x_2,x_4}}
     \operatorname{dist}(c,\tilde c).
   \]
We emphasize  that	in contrast to the TV functional $\TV^{\mathrm{int}}$
the second order absolute difference \(\mathrm d_2\) is not convex in \(x_i\), $i=1,3$
on Hadamard manifolds. 
However, using this definition we can introduce 
the absolute value of the second order difference in $x$-direction
     \begin{align}
       (\mathrm{d}_{xx}^{\mathrm{int}}u)_{i}
        \coloneqq\begin{cases}
          \mathrm{d}_2(u_{{i} + (1,0)},u_{i},u_{{i} - (1,0)})&\mathrm{if} \; i \pm (1,0) \in \grid, \\
          0&\text{otherwise},
        \end{cases}
     \end{align}
     and similarly in $y$-direction.
Note that $\mathrm{d}_{xx}^{\mathrm{int}}(u)_{i}$ is the counterpart of the \emph{absolute value} of the Euclidean difference $\frac{1}{2}|(D_{xx} u)_{i}|$.
The absolute value of the mixed derivative $\frac{1}{2}|(D_{xy} u)_{i}|$ is replaced in the manifold-valued setting by
		\begin{align}
		(\mathrm{d}_{xy}^{\mathrm{int}} u)_{i}
        \coloneqq
    \begin{cases}
      	\mathrm d_{1,1}
          \bigl(
            u_{i},
            \!\!&\!\!
            u_{i+(0,-1)},
            u_{i+(1,0)},
            u_{i+(1,-1)}
          \bigr)
          \\& \text{if}
          \ i \pm (0,1) \wedge
      	i + (1,0)\in \grid, \\
      	0 & \mathrm{otherwise},
      	\end{cases}
		\end{align}
    and similarly for $\mathrm{d}_{yx}^{\mathrm{int}}$. 
		Then we define the following counter part of $\TV_2$:
    \begin{align} \label{TV2_mani}
      \operatorname{TV}_2^{\mathrm{int}}(u)
        \coloneqq
          \sum_{i \in\grid}&
            \Bigl(
						  (\mathrm{d}_{xx}^{\mathrm{int}}u)_{i}^2 +(\mathrm{d}_{yy}^{\mathrm{int}}u)_{i}^2\\
        &+		(\mathrm{d}_{xy}^{\mathrm{int}} u)_{i}^2  + (\mathrm{d}_{xy}^{\mathrm{int}} u)_{i}^2
						\Bigr)^\frac{1}{2}.
      \end{align}
	In~\cite{BBSW16}, anisotropic versions of 
	$\operatorname{TV}^{\mathrm{int}}$ and $\operatorname{TV}_2^{\mathrm{int}}$ were used to set up an additive prior
	within a denoising model.
	Here we focus on the isotropic prior given by\\[1ex]
  \textbf{Additive Coupling}
	\begin{equation} \label{model_add_int}	
	\TV_{1 \wedge 2}^{\mathrm{int}}(u)
	\coloneqq
	\beta \TV^{\mathrm{int}}(u) + (1-\beta) \TV_2^{\mathrm{int}}(u).
	\end{equation}

\begin{figure*}\centering
  \begin{subfigure}{.49\textwidth}\centering
    \includegraphics{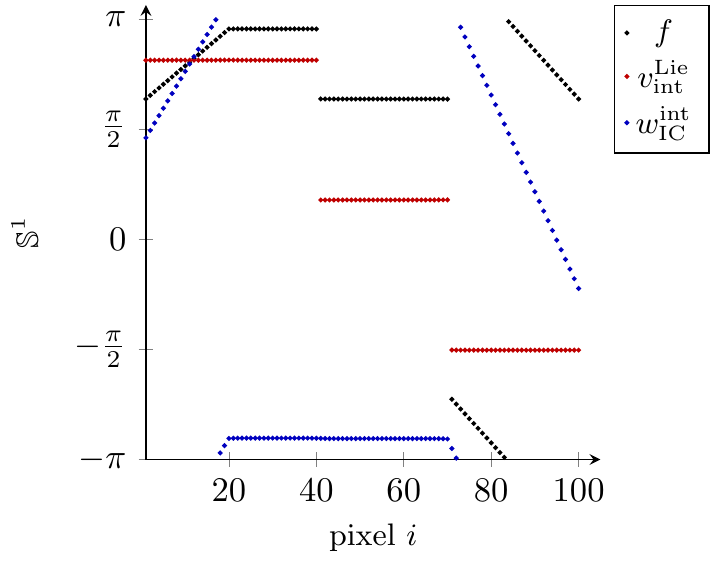}
		\caption{Intrinsic data representation.}\label{ictgv:fig:int:mid:para}
  \end{subfigure}
  \begin{subfigure}{.49\textwidth}\centering
    \includegraphics{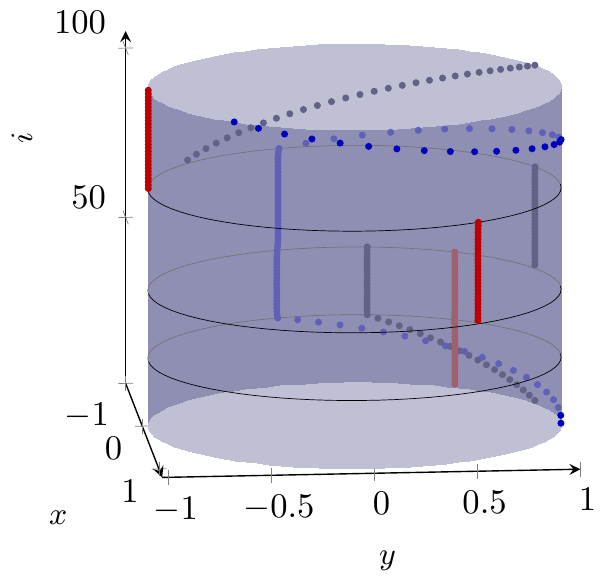}
		\caption{Embedded signals.}\label{ictgv:fig:int:mid:emb}
  \end{subfigure}
	\caption[]{\textbf{Midpoint IC model}: Decomposition of~\(f\) from Fig.~\ref{fig:ex:ic}.
		(a) Original piecewise geodesic
		signal $f$ (black),
		its the piecewise constant part $v^{\mathrm{int}}_{\IC}$ (red) and piecewise geodesic part $w^{\mathrm{int}}_{\IC}$ (blue) 
		parameterized by their angles in $[-\pi,\pi)$.
		(b) Same signals in the embedding space $\R^2$.}
	\label{fig:int:mid}
\end{figure*}

	Concerning our IC model we realize that in the Euclidean setting the IC of two
one-homogeneous functions $F_1,F_2$ as for example  $\TV$ and $\TV_2$,
  can be rewritten as
  \begin{equation}\label{scaling}
    F_1 \square F_2(u) = \frac{1}{2} \inf_{u = \frac{1}{2} (v+w)} \{ F_1(v) + F_2(w)\}.
  \end{equation}
  Now we may consider the ``midpoint infimal convolution''
  of $F_i\colon \mathcal M \rightarrow \mathbb R$, $i=1,2$,
   \[
    F_1 \square_m F_2(u)
    \coloneqq  \inf_{u \in \mathcal{C}_{v,w} } \{ F_1(v) + F_2(w)\},
  \]
	in the following ``Midpoint'' IC prior:
	\\[1ex]
	\textbf{Infimal Convolution} (Midpoint Approach)
	\begin{align}\label{ficm}
	{\mathrm{IC}}^{\mathrm{int}} (u)
	\coloneqq 
	\inf_{u \in \mathcal{C}_{v,w} } \{ \beta \TV^{\mathrm{int}}(v) \square_m (1-\beta) \TV^{\mathrm{int}}_2 (w) \}.
	\end{align}
	We are interested in the 
	\\[1ex]
	\textbf{IC Decomposition} (Midpoint Approach: $u = \geodesic{v,w}(\frac{1}{2})$)
	\begin{equation}	\label{eq:midpointorig}	
	\begin{split}	
	E_{\IC}^{\mathrm{int}}(v,w)
	&\coloneqq
	\frac{1}{2} \sum_{i \in\grid}
	\dist^2(\geodesic{v_i,w_i}
	(\tfrac{1}{2}),f_i)
	\\&\qquad+
	\alpha\bigl(
	\beta \TV^{\mathrm{int}}(v) + (1-\beta) \TV_2^{\mathrm{int}}(w)
	\bigr)
	\end{split}
	\end{equation}
	Here, $\geodesic{v_i,w_i}(\frac{1}{2})$
	addresses the midpoint of the geodesic having smallest distance from $f_i$, for all $i\in\grid$,
	and we finally set $u \coloneqq \geodesic{v,w}(\frac{1}{2})$.

\begin{example}\label{ex:S1Signal:MPIC}
	We consider the signal \(f\)
	from Fig.~\ref{fig:ex:ic}.
	Its Midpoint IC decomposition with parameters $\alpha=0.005$,
	$\beta=\frac25$ into a piecewise constant part~\(v^{\mathrm{int}}_{\IC}\) and a piecewise
	geodesic part~\(w^{\mathrm{int}}_{\IC}\) is shown in Fig.~\ref{fig:int:mid}.
	In contrast to the extrinsic model, both parts have values in the manifold now and can be also visualized in the left plot.	
\end{example}

\subsection{Intrinsic TGV Model}\label{sec:int_mani:tgv} 

TGV does not require the definition of second order differences.
The first summand in the Euclidean  TGV model \eqref{TGV_discr}
can be replaced for $\xi = (\xi_i)_{i \in \Gamma}$ 
with \(\xi_i \in (T_{u_{i}}\mathcal{M})^2\)
by
\begin{equation}
\lVert \nabla^{\mathrm{int}} u - \xi \rVert_{2,1,u}.
\end{equation}
The treatment of the backward differences $\widetilde \nabla_S \xi$ in the second TGV summand requires 
to ``substract'' tangent vector from different tangent spaces. 
For this purpose, we apply the parallel transport between the tangent spaces.
We realize the parallel transport by the pole ladder \eqref{pre:eq:pole} which is exact in symmetric Riemannian manifolds. 
Then the backward difference of a vector field $\xi \in T_{u}\mathcal{M}^N$, i.e. $\xi_i \in T_{u_i}\mathcal{M}^N$, 
in $x$-direction reads as
\begin{equation}
(\widetilde{D}_x^{\mathrm{int}} \xi)_i \coloneqq 
\left\{
\begin{array}{ll}
\xi_i - P_{u_{i-(1,0)}\to u_i}^\mathrm{P}(\xi_{i-(1,0)}) & \mathrm{if} \; i\pm(1,0)\in\grid,\\
0\quad&\mathrm{otherwise,}
\end{array}
\right.
\end{equation}
similarly in $y$-direction. 
Application of backward differences to a vector field $\xi\in (T_u\mathcal{M}^N)^s$ 
is meant componentwise.
In our minimization algorithms we will need the differential of the backward differences.
Note that the pole ladder consists only of the concatenation of geodesics, exponential and logarithmic maps
whose differentials are given in Lemma \ref{pre:differentials}.
For the differentials of the direct parallel transport in $\mathbb{S}^d$ and $\SPD{r}$ we refer to~\cite{BHSW17}.
We set 
\begin{equation}
\widetilde{\nabla}^{\mathrm{int}} 
\coloneqq 
\begin{pmatrix}
\widetilde{D}_x^{\mathrm{int}}&0\\
\widetilde{D}_y^{\mathrm{int}}&0\\
0& \widetilde{D}_x^{\mathrm{int}}\\
0&\widetilde{D}_y^{\mathrm{int}}
\end{pmatrix}.
\end{equation}
For simplicity of computations, we use $\widetilde{\nabla}^{\mathrm{int}}$ instead of the counterpart of $\widetilde{\nabla}_S$
to define a (pole ladder) TGV model by
\\[1ex]
	\textbf{Total Generalized Variation}
\begin{align}\label{tgvmaniinf_pole}
\TGV^{\mathrm{int}}(u)
\coloneqq &
\inf_{\xi} \left\{ \beta\lVert\nabla^{\mathrm{int}} u-\xi\rVert_{2,1,u} \right.\\
& + \left. (1-\beta)\lVert\widetilde\nabla^{\mathrm{int}} \xi \rVert_{2,1,u}\right\}.
\end{align}
Again we are interested in the 
\\[1ex]
	\textbf{TGV Decomposition} ($\nabla^\mathrm{int} u= \tilde \xi + \xi$)
\begin{equation}\label{mod_mani_tgv_pole}
\begin{split}
&E_{\TGV}^{\mathrm{int}}(u,\xi)
\coloneqq 
\frac{1}{2}\pdist^2(u,f)\\
&+\alpha \Big( \beta\lVert \nabla^\mathrm{int} u-\xi \rVert_{2,1,u}+ (1-\beta)\lVert\widetilde\nabla^{\mathrm{int}} \xi\rVert_{2,1,u} \Big).
\end{split}
\end{equation}

\begin{figure*}
\centering	
    \begin{subfigure}{.49\textwidth}\centering
      \includegraphics{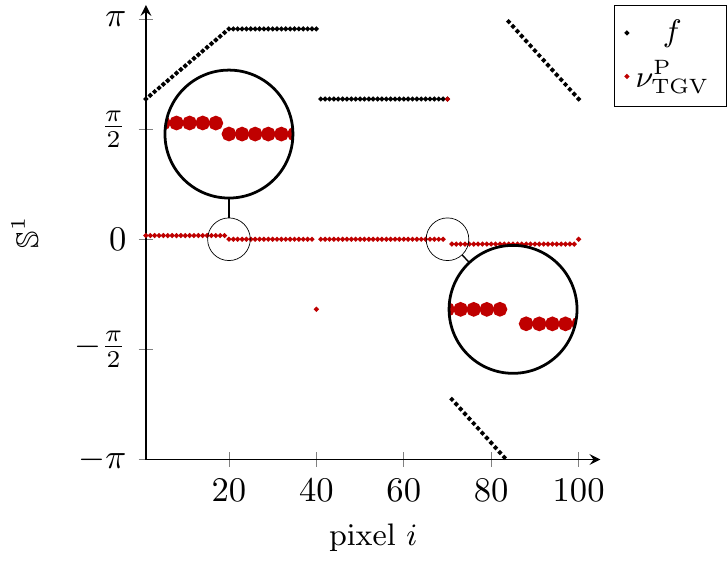}      
		\caption{Intrinsic representation of pole ladder TGV.}\label{ictgv:fig:int:poletgv:para}
    \end{subfigure}
    \begin{subfigure}{.49\textwidth}\centering
      \includegraphics{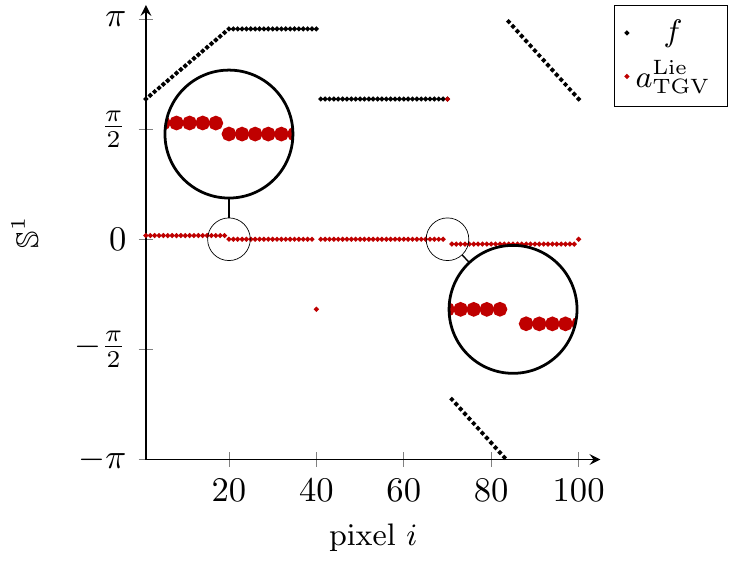}
		\caption{Intrinsic representation of Lie group TGV.}\label{ictgv:fig:int:poletgv:lie}
    \end{subfigure}
	\caption[]{\textbf{Pole ladder TGV model} (left) versus \textbf{Lie group TGV model} (right) 
	  applied to $f$ from Fig.~\ref{fig:ex:ic}.
	  Using $\xi^{\mathrm{P}}_{\TGV} = \nu^{\mathrm{P}}_{\TGV} \zb u^{\perp}$, 
		see Example \ref{ex:S1Signal:TGV}, 
		the signal $\nu^{\mathrm{P}}_{\TGV}$  (red) can be visualized in $[-\pi,\pi)$. 
		Interestingly, the components of the signal~\(\nu^{\textrm{P}}_{\textrm{TGV}}\) and \(a^{\textrm{Lie}}_{\textrm{TGV}}\)
		from Example \ref{ex:S1Signal:LieIC}
differ by $3.8 \times 10^{-4}$.
	 }\label{fig:both_TGV}
\end{figure*}

\begin{example}	\label{ex:S1Signal:TGV}
	We apply the pole ladder TGV model~\eqref{mod_mani_tgv_pole} 
	with parameters $\alpha=0.0101$, $\beta=\frac{100}{101}$
	to the signal $f$ from Fig.~\ref{fig:ex:ic}.
	Since we have for $x \in \mathbb S^1$ that 
	$T_x \mathbb S^1 \coloneqq \{\nu {\zb x}^\perp: \\nu \in \mathbb R\}$,
	with the embedding $\zb x \in \mathbb R^2$ of $x$,
	the tangent vectors vectors
	\(\xi_i \in T_{u_i}\), $i=1,\ldots,100$,
	can be represented by $\nu_i \in[-\pi,\pi)$,
	where $	\xi_i = \nu_i {\zb u}_i^\perp$.
	The result is shown in Fig. \ref{fig:both_TGV} left.
	The signal $\nu^{\mathrm{P}}_{\mathrm{TGV}} \coloneqq (\nu_i)_{i=1}^{100}$ 
	approximates the finite differences of \(f\) 
	taking its phase-valued structure into account.	
\end{example}

\begin{remark}[Comparison to~\cite{BHSW17}]\label{differences}
In the parallel work~\cite{BHSW17}, the authors introduced an axiomatic 
discrete TGV approach for manifold-valued images.
In the one-dimensional setting, they proposed the prior 
\begin{align}
	&\TGV^{\mathrm{BHSW}} (u,v) \coloneqq \min_v \Big\{
  \sum_{i \in \Gamma} \beta\dist(u_{i+1},v_i)  \\
&+ (1-\beta) 
\dist\left( v_i, \gamma \left(u_{i-1}, \gamma (u_i,v_{i-1}; \tfrac12);2 \right) \right)  \Big\}.
\end{align}
In contrast to our prior, $\TGV^{\mathrm{BHSW}}$ is directly defined on the manifold $\MM$ and not on $T\MM$.
Setting 
$
v_i \coloneqq \exp_{u_i} \xi_i
$
we can relate the distances in $\TGV^{\mathrm{BHSW}}$ to those in our $\TGV^{\mathrm{int}}$ prior by
\begin{align}
&\dist(u_{i+1},v_i) = \dist(u_{i+1},\exp_{u_i} \xi_i) \\
&\approx \|\log_{u_i} u_{i+1} - \xi_i\|_{u_i} = \| (\nabla^{\mathrm{int}} u)_i - \xi_i\|_{u_i},
\end{align}
and
\begin{align}
&\dist\left( v_i, \gamma \left(u_{i-1}, \gamma (u_i,v_{i-1}; \tfrac12);2 \right) \right) \\
&= \dist \left(\exp_{u_i} \xi_i, \exp_{u_i} P_{u_{i-1}\to u_i}^\mathrm{S} \xi_{i-1} \right)\\
&\approx \|\xi_i - P_{u_{i-1}\to u_i}^\mathrm{S} \xi_{i-1}\|_{u_i}
\approx
\|\xi_i - P_{u_{i-1}\to u_i}^\mathrm{P} \xi_{i-1}\|_{u_i}\\
&= \|\widetilde \nabla_x \xi_i \|_{u_i}.
\end{align}

In the two-dimensional setting, we prefer an isotropic models instead of an anisotropic one  in~\cite{BHSW17}.
For minimizing the TGV decomposition model we apply a
gradient descent algorithm to a slightly smoothed version while the authors in~\cite{BHSW17} use
a cyclic proximal point algorithm  without any convergence guarantee.
\end{remark}

%
\section{Intrinsic Models for Lie Groups}\label{sec:int_lie} 
Now we assume that the manifold $\mathcal{M}$ is in addition a Lie group
with group action
$\circ \colon\mathcal M\times\mathcal M \to\mathcal M$ and unit element \(e\in\mathcal M\).
This means that the group action as well as the mapping $x \mapsto x^{-1}$, $x \in \mathcal M$
are smooth. 
For more information on Lie groups we refer to \cite{GQ17,Rossmann03}. 
In our numerical examples, $\mathbb S^1$ and $\operatorname{SO}(3)$ are Lie groups.

The idea is to set up the different priors by replacing additions and substractions in the Euclidean models
by the group operation. All three priors are defined on the manifold now.

The \textit{left} and \textit{right translation} ${\mathcal L}_x, {\mathcal R}_x:  M\times\mathcal M \to\mathcal M$
with respect
to~\(x \in\mathcal M\) are given by
\[
{\mathcal L}_x(y) \coloneqq x \circ y, \quad
{\mathcal R}_x(y) \coloneqq y \circ x,
\]
respectively and 
\begin{equation}
D{\mathcal L}_x(y)[\xi] = x \circ \xi, \quad
D{\mathcal R}_x(y)[\xi] =  \xi \circ x.
\end{equation}
A metric on a Lie group is called right-invariant
if for all $x,y \in \mathcal{M}$ and all \(\xi,\zeta\in T_y\mathcal M\) it holds
    \begin{align}
      \langle \xi,\zeta\rangle_y &=
      \langle D{\mathcal R}_x(y)[\xi],D{\mathcal R}_x(y)[\zeta]\rangle_{y \circ x}
\end{align}
and similarly for the left-invariant metric.
Therefore a right (left) invariant metric is induced 
by a metric on the tangent space $T_e {\mathcal M}$
which is actually the Lie algebra of $\mathcal{M}$. 
For matrix groups we will use the Frobenius inner product on $T_e {\mathcal M}$.
Every compact Lie group, in particular $\mathbb S^1$ and $\mathrm{SO}(3)$, admit  
a metric which is both left- and right-invariant, i.e. they have a bi-invariant metric. 
This is in general not the case as the example of
Euclidean transformation group $\mathrm{SE}(n)$, $n \ge 2$ shows, see~\cite{APA06}.
In this section, we restrict our attention to manifolds $\MM$ having a right-invariant metric.
Then we have for the distance function on $\mathcal{M}$,
\begin{align}\label{left}
\distM(x,y) &= \distM(x \circ y^{-1},e) = \distM(e,y \circ x^{-1}).
\end{align}
This distance function is used in the
the data term in \eqref{data_mani}.
Replacing substractions by appropriate group operations, we can define forward and backward ``differences'' in $x$-direction
in the Lie group as
\begin{align}\label{not:1st}
(D^{\mathrm {Lie}}_x u)_{i} \coloneqq
\begin{cases}
u_{i+(1,0)} \circ u_i^{-1}& \mathrm{if} \;  i+(1,0) \in \grid, \\
e & \mathrm{otherwise},
\end{cases}
\end{align}
and
\begin{align}
(\widetilde D^{\mathrm {Lie}}_x u)_i \coloneqq
\begin{cases}
u_i \circ u_{i-(1,0)}^{-1}& \mathrm{if} \  i \pm (1,0) \in \grid, \\
e & \mathrm{otherwise}.
\end{cases}
\end{align}
and similarly in $y$-direction.
Then we see for the manifold-valued TV term \eqref{TV_mani} by \eqref{left} that
\begin{align} \label{TV_mani_a}
&\TV^{\mathrm {Lie}} (u)
\coloneqq
\TV^{\mathrm {int}} (u)\\
&= 
\sum_{i \in\grid}
\big(
\dist^2 \left( (D^{\mathrm {Lie}}_x u)_{i},e \right)
+ \dist^2 \left( (D^{\mathrm {Lie}}_y u)_{i},e \right)
\big)^\frac{1}{2}.
\end{align}
Furthermore, second order differences on Lie groups resemble 
the concatenation of forward and backward operations the Euclidean case, e.g.,
\\
$(D_{xx}u)_i = \left( u_{i+(1,0)} - u_i \right) - \left( u_{i} - u_{i+(1,0)} \right)$
by 
\begin{align}\label{not:2ndmani}
&(D^{\mathrm {Lie}}_{xx} u)_i \\
&\coloneqq
\begin{cases}
u_{i+(1,0)} \circ u_i^{-1} \circ u_{i-(1,0)} \circ u_i^{-1} & \mathrm{if} \; i \pm (1,0) \in \grid, \\
e & \mathrm{otherwise,}
\end{cases}
\end{align}
and in mixed directions
\begin{equation}\label{dxymani}
\begin{split}
&(D^{\mathrm {Lie}}_{xy} u)_i  \\
&\coloneqq
\begin{cases}
u_{i+(1,0)} \circ u_i^{-1} \circ u_{i-(0,1)} \circ u_{i + (1,-1)}^{-1} 
& \mathrm{if} \ i \pm (0,1) \wedge\\&\phantom{\mathrm{if }} 
i  + (1,0)\in \grid, \\
e & \mathrm{otherwise},
\end{cases}
\end{split}
\end{equation}
and similarly for $D^{\mathrm {Lie}}_{yy}$ and $D^{\mathrm {Lie}}_{yx}$.
Then, with
\begin{equation}
(\mathrm{d}_{*}^{\mathrm{Lie}} u)^2_{i}
\coloneqq \dist \left( (D^{\mathrm{Lie}}_{*} u)_i,e \right), \quad * \in \{xx,yy,xy,yx\},
\end{equation}
we define
\begin{align}
\mathrm{TV}_2^{\mathrm{Lie}}(u)
\coloneqq 
\sum_{i \in\grid}
\bigl(&
(\mathrm{d}_{xx}^{\mathrm{Lie}} u)^2_{i} + (\mathrm{d}_{yy}^{\mathrm{Lie}} u)^2_{i} \\
&+ 
(\mathrm{d}_{xy}^{\mathrm{Lie}} u)^2_{i}  + (\mathrm{d}_{yx}^{\mathrm{Lie}} u)^2_{i}
\bigr)^\frac{1}{2}.
\end{align}

Now the additive and IC prior on  Lie groups can be introduced as follows:
\\[1ex]
\textbf{Additive Coupling}
	\begin{equation*}
	\TV_{1 \wedge 2}^{\mathrm{Lie}}(u)
	\coloneqq
		\beta \TV^{\mathrm{Lie}}(u) + (1-\beta) \TV_2^{\mathrm{Lie}}(u).
	\end{equation*}
\textbf{Infimal Convolution}
	\begin{align}
	\IC^{\mathrm{Lie}}(u)
	&\coloneqq
	\inf_{u=v\circ w} \{
	\beta\TV^{\mathrm{Lie}}(v) + (1-\beta)\TV_2^{\mathrm{Lie}}(w)\}.
	\end{align}
Again, we are interested in the splitting model:	
\\[1ex]
\textbf{IC Decomposition} ($u= v\circ w$)
		\begin{align} \label{mod_lie_ic}
		E_{\IC}^{\mathrm{Lie}}(v,w)
	&\coloneqq
	\frac{1}{2}\pdist^2(f,v\circ w)
	\\&\quad+ \alpha\bigl(\beta
	\TV^{\mathrm{Lie}}(v) + (1-\beta)\TV_2^{\mathrm{Lie}}(w)\bigr).
	\end{align}
	
Since we just apply group operations to define ,,differences''
the TGV prior is also defined on the Lie group by
\\[1ex]
\textbf{Total Generalized Variation}
	\begin{align}\label{mod_lie_tgv}
	&\TGV^{\mathrm{Lie}}(u)\\
	\coloneqq&
	\inf_{a=(a_1,a_2)} 
	\Bigl\{ \beta 
	\big(
	\pdist (D_x^{\mathrm {Lie}} u, a_1 )^2
	+
	\pdist (D_y^{\mathrm {Lie}} u, a_2)^2
	\big)^\frac{1}{2} \\
	&+ (1-\beta) 
	\big(
	\pdist (\widetilde D_x^{\mathrm {Lie}} a_1,e)^2
	+ 
	\pdist (\widetilde D_y^{\mathrm {Lie}} a_2,e)^2
	\\&\qquad\qquad+ 
	\pdist (\widetilde D_y^{\mathrm {Lie}} a_1,e)^2
	+ 
	\pdist (\widetilde D_x^{\mathrm {Lie}} a_2,e)^2
	\big)^\frac{1}{2}
	\Bigr\}.	
	\end{align}
	Actually, we are interested in the following decomposition:
	\\[1ex]
	\textbf{TGV Decomposition} ($(D^{\mathrm {Lie}}_x u, D^{\mathrm {Lie}}_y u)^\tT = (\tilde a_k \circ a_k)_{k=1}^2$)
	\begin{align}\label{prior_tgv_lie}
	&E_{\TGV}^{\mathrm{Lie}}(u,a)
	\coloneqq
	\frac{1}{2}\pdist^2(f,u)
	      \\
	      &
	+ \alpha 
	\Big(
	\beta 
	\big(
	\pdist (D_x^{\mathrm {Lie}} u, a_1 )^2
	+
	\pdist (D_y^{\mathrm {Lie}} u, a_2)^2
	\big)^\frac{1}{2} \\
	&+ (1-\beta)
	\big(
	\pdist (\widetilde D_x^{\mathrm {Lie}} a_1,e)^2
	+ 
	\pdist (\widetilde D_y^{\mathrm {Lie}} a_2,e)^2\\
	&\qquad\qquad+\pdist (\widetilde D_y^{\mathrm {Lie}} a_1,e)^2
	+ 
	\pdist (\widetilde D_x^{\mathrm {Lie}} a_2,e)^2
	\big)^\frac{1}{2}
	\Big)
	\end{align}

\begin{figure*}\centering
\begin{subfigure}{.49\textwidth}\centering
    \includegraphics{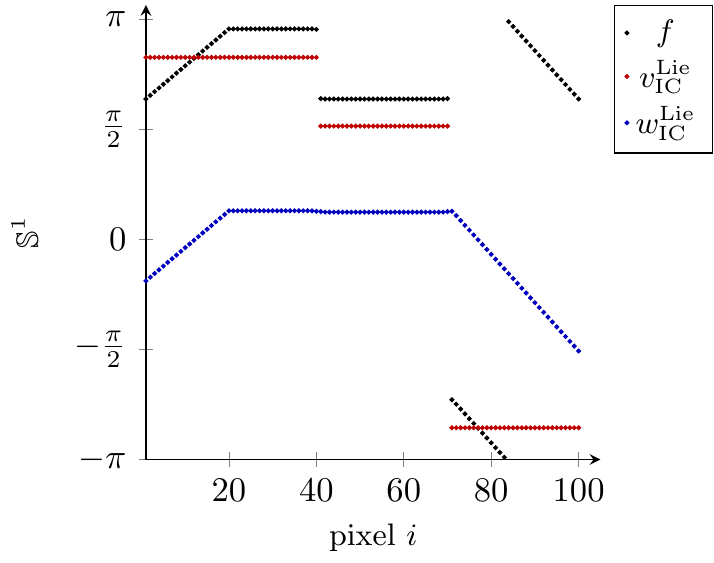}
		\caption{Intrinsic data representation.}\label{ictgv:fig:int:ic:para}
\end{subfigure}
\begin{subfigure}{.49\textwidth}\centering
  \includegraphics{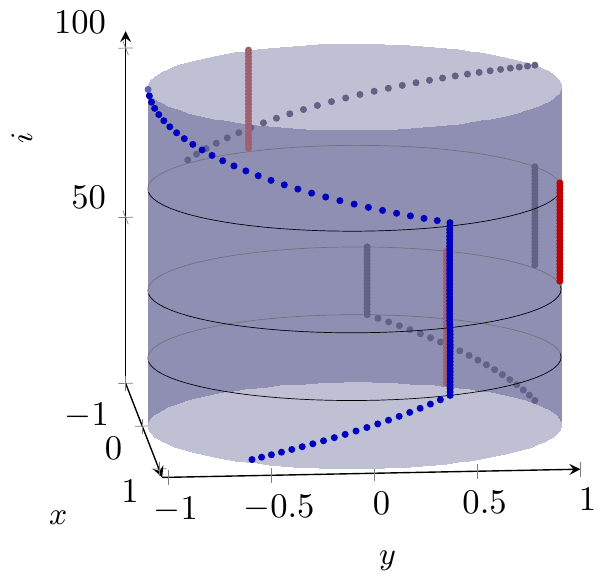}
	\caption{Embedded signals.}\label{ictgv:fig:int:ic:emb}  
\end{subfigure}
\caption{\textbf{Lie group IC model}: Decomposition of \(f\) from Fig.~\ref{fig:ex:ic}. 
(a) Original piecewise geodesic signal $f$ (black) 
  determined by its angles in $[-\pi,\pi)$, the piecewise constant component $v^{\textrm{Lie}}_{\textrm{IC}}$ (red), 
  and piecewise geodesic component $w^{\textrm{Lie}}_{\textrm{IC}}$ (blue).
	the piecewise constant component $v^{\textrm{Lie}}_{\textrm{IC}}$ (red), 
	and piecewise geodesic component $w^{\textrm{Lie}}_{\textrm{IC}}$ (blue).
	In contrast to the Midpoint IC model, the decomposed parts have other slopes and jump heights.
	(b) Same signals in the embedding space $\R^2$.}\label{fig:int:ic}
\end{figure*}

\begin{example}\label{ex:S1Signal:LieIC}
First, we apply the Lie group IC model~\eqref{mod_lie_ic}
with parameters $\alpha = 0.001$, $\beta = \frac23$ to the signal $f$ in Fig.~\ref{fig:ex:ic}.
The result is shown in Fig. \ref{fig:int:ic}.
As expected, the piecewise constant part~\(v^{\textrm{Lie}}_{\textrm{IC}}\)
contains the jumps and the second order 
component~\(w^{\textrm{Lie}}_{\textrm{IC}}\)
is piecewise geodesic.
Due to the construction, the main difference to the result of the Midpoint IC model 
is the slope of the piecewise geodesic parts and 
the jump height  of the piecewise constant part.

Next, we apply the Lie group TGV model~\eqref{mod_lie_tgv} with parameters $\alpha = 0.001$, $\beta = \frac23$.
The result is depicted in Fig. \ref{fig:both_TGV} right.
Interestingly, comparing~\(\nu^{\textrm{P}}_{\textrm{TGV}}\) with \(a^{\textrm{Lie}}_{\textrm{TGV}}\),
we do not see a difference in the plot.
\end{example}

\section{Gradient Descent for the Intrinsic Models}\label{sec:ini_alg}
%
To compute critical points of the intrinsic models we apply
gradient descent algorithms. 
To make the priors differentiable,
we have to add a small positive value $\varepsilon^2 \ll 1$ within the square roots
appearing in $\TV^*$, $\TV_2^*$ and $\TGV^*$, $* \in \{\mathrm{int}, \mathrm{Lie}\}$.
For the intrinsic IC models $E_{\IC}^{\mathrm{int}}$, $E_{\IC}^{\mathrm{Lie}}$ defined on $(\MM^N)^2$
and $E_{\TGV}^{\mathrm{Lie}}$ defined on $(\MM^N)^3$
we apply the gradient descent Algorithm ~\ref{alg:grad_desc_ic}. We use the notation
\begin{equation}
\mathbf{M} \coloneqq (\mathcal{M}^{N})^s, \quad s \in \{2,3\}.
\end{equation}
For the $E_{\TGV}^{\mathrm{int}}$ model which is defined on the manifold and the tangent bundle, 
we propose Algorithm~\ref{ictgv:alg:grad_desc_tgv}.

\begin{algorithm}[t]
	\caption[]{Gradient Descent for \\
	$E \colon \mathbf{M} \to\R$
	from $\{E_{\IC}^{\mathrm{int}}, E_{\IC}^{\mathrm{Lie}}, E_{\TGV}^{\mathrm{Lie}} \}$
	}
	\label{alg:grad_desc_ic}
	\begin{algorithmic}
		\State \textbf{Input:}
		 $p^{(0)}\in \mathbf{M}$; $c,  \sigma> 0;\ \rho\in(0,1)$
		\State \textbf{Output:} $\hat p \in \mathbf{M}$
		\State $r=0$
		\Repeat
		\State
		Choose the smallest \( l\in\N \) fulfilling the  Armijo condition	
		\begin{equation}\label{eq:armijo}
		 \qquad   E(p^{(r)})
		    - c \sigma \rho^l\lVert
		    \grad_{\mathbf{M}} E (p^{(r)})
		    \rVert^2
		    \ge E
		    (p_l^{(r)})
		\end{equation}	
		\State with 
		\begin{equation}
		 \qquad   p_l^{(r)} \coloneqq
		    \exp_{p^{(r)}} \bigl(-\rho^l \sigma \grad_{\mathbf{M}}
		    E (p^{(r)})\bigr)
		\end{equation}
		\State Set
		\begin{equation}\label{eq:graddes}
		\qquad p^{(r+1)} \coloneqq \exp_{p^{(r)}}(
		-\rho^l\sigma \grad_{\mathbf{M}} E(p^{(r)}))
		\end{equation}
		\State $r\gets r+1$;
		\Until a stopping criterion is reached;
		\State $\hat{p} = p^{(r)}$;
	\end{algorithmic}
\end{algorithm}

\begin{algorithm}[t]
	\caption[]{Gradient Descent for $E \coloneqq E_{\TGV}^{\mathrm{int}}$}
	\label{ictgv:alg:grad_desc_tgv}
	\begin{algorithmic}
		\State \textbf{Input:}
		$u^{(0)}\in\MM^{N};\ \xi^{(0)}\in (T_{u^{(0)}}\MM^N)^2$;
		\State 
		\phantom{\textbf{Input:}} $\ \sigma,c> 0;\ \rho\in(0,1)$
		\State \textbf{Output:} 
		$\hat{u}\in\MM^{N}, \ \hat \xi \in (T_{\hat u} \MM^N)^2$
		\State $r=0$;
		\Repeat
		\State Compute
		\begin{align}
		\qquad b^{(r)} &\coloneqq -\grad_{(T\MM^N)^2,\xi} \bigl(E (u^{(r)},\cdot) \bigr)(\xi^{(r)})\\
		v^{(r)} &\coloneqq -\grad_{\MM^N,u} \bigl(E(\cdot,\xi^{(r)}) \bigr)(u^{(r)})
		\end{align}		
		\State Choose the smallest \( l\in\mathbb N\) fulfilling the Armijo condition
		\begin{align}
		\qquad & E \bigl(u^{(r,l)},\xi^{(r,l)})\bigr) \le E\bigl(u^{(r)},v^{(r)}\bigr)\\
		& - c\sigma\rho^l\bigl(\lVert v^{(r)}\rVert_{u^{(r)}}^{2}+\lVert b_1^{(r)}\rVert^2_{u^{(r)}}+b_2^{(r)}\rVert^2_{u^{(r)}}\bigr)
		\end{align}
		\State with
		\begin{align}
		\qquad u^{(r,l)} &\coloneqq \exp_{u^{(r)}}(\sigma\rho^l v^{(r)})\\
		\xi_1^{(r,l)} &\coloneqq P_{u^{(r)}\to u^{(r,l)}}\bigl(\xi_1^{(r)} +\sigma\rho^l b_1^{(r)}\bigr)\\
		\xi_2^{(r,l)} &\coloneqq P_{u^{(r)}\to u^{(r,l)}}\bigl(\xi_2^{(r)} +\sigma\rho^l b_2^{(r)}\bigr)
		\end{align}
		\State Set 
		\begin{align}
		\qquad
		u^{(r+1)} &\coloneqq u^{(r,l)}\\
		\xi^{(r+1)}&\coloneqq \xi^{(r,l)}
		\end{align}	
		\State $r\gets r+1$
		\Until a stopping criterion is reached;
		\State $(\hat{u}, \hat \xi) \coloneqq (u^{(r)}, \xi^{(r)})$
	\end{algorithmic}
\end{algorithm}

\begin{proposition}
\begin{enumerate}[label = \normalfont\roman*)]
 \item
	For any of the functionals $E_{\IC}^{\mathrm{int}}$, $E_{\IC}^{\mathrm{Lie}}$, and $E_{\TGV}^{\mathrm{Lie}}$, every accumulation point of the sequence $(p^{(r)})_{r\in\N}$ generated
	by Algorithm \ref{alg:grad_desc_ic}  is a critical point.
 \item
    For the functional $E^{\mathrm{int}}_{\TGV}$,
	every accumulation point of the sequence $(u^{(r)},\xi^{(r)})_{r\in\mathbb N}$ generated
	by Algorithm \ref{ictgv:alg:grad_desc_tgv}  is a critical point.
	\end{enumerate}
\end{proposition}

\begin{proof}
Part i) follows by \cite[Theorem~4.3.1]{AMS08}.

Concerning ii) we recognize that Algorithm~\ref{ictgv:alg:grad_desc_tgv} 
is a descent algorithm on $\MM^N\times (T\MM^N)^2$~\cite{Rent11,Sas1958} 
with an Armijo step size rule. Hence the assumption follows by~\cite[Theorem~4.3.1]{AMS08}. \hfill $\Box$
\end{proof}

To obtain the gradients in Algorithm~\ref{alg:grad_desc_ic} and~\ref{ictgv:alg:grad_desc_tgv}
we have to compute the gradients of all involved summands.
In the rest of this section, we sketch their computation.
We restrict our attention to symmetric Riemannian manifolds and Lie groups with bi-invariant metric.
for a detailed treatment of all involved Riemannian gradients we refer to ~\cite{persch2018}.

The gradients 
of the data term $\mathcal{E}^{\mathrm{int}}_{\mathrm{data}}$ in \eqref{data_mani}
and 
the smoothed terms $\TV^{\mathrm{int}} = \TV^{\mathrm{Lie}}$ in \eqref{TV_mani} can be obtained
by the chain rule and
\begin{equation}
\grad_{\mathcal{M}} \dist^2(\cdot,y) (x) = -2 \log_x y. \label{eq:der:dsq}
\end{equation}
For $\TV_2^{\mathrm{int}}$ in \eqref{TV2_mani} we can apply the
the gradient computations of $\mathrm{d}_*^{\mathrm{int}}$, $* \in \{xx,yy,xy,yx\}$, 
detailed in \cite{BBSW16} together with the chain rule.
The gradient of \\
$\dist^2(\geodesic{v_i,w_i}(\tfrac{1}{2}),f_i)$
in the data term of $E_{\IC}^{\mathrm{int}}$ in \eqref{eq:midpointorig}	
follows by the chain rule and Lemma~\ref{pre:differentials} iv) and v).
The gradient of 
$\dist^2(f_i , \cdot \circ w_i)$
in the data term 
of $E_{\IC}^{\mathrm{Lie}}$ in \eqref{mod_lie_ic}
can be obtained by
\begin{align}
	\grad&_{\mathcal{M},v_{i}}\Bigl( \dist^2(f,\cdot\circ w)\Bigr)(v_i) \\
	&= \grad_{\mathcal{M},v_{i}}\Bigl(\dist^2(f\circ w^{-1},\cdot)\Bigr)(v_i)
	\\&=-2 \log_{v_i} (f_i\circ w_i^{-1})
\end{align}
using the right invariance of the geodesic distance.
Similarly, we  compute the gradient with respect to $w$ using the left invariance of the metric.
To get the gradient of $\TV_2^{\mathrm{Lie}}$ we apply the chain rule with the following lemma.

\begin{lemma} \label{lem:grad_lie}
 The gradients of $(\mathrm{d}_{xx}^{\mathrm{Lie}})^2$ and   $(\mathrm{d}_{xy}^{\mathrm{Lie}})^2$
 are given by
 \begin{align}
		&\grad_{\mathcal{M},w_{i}} \bigl( \mathrm{d}_{xx}^{\mathrm{Lie}} \cdot \bigr)^2_{i}(w) = - 2 \big( \\		
		& D {\mathcal L}_ {w_{i-(1,0)}\circ w_{i}^{-1}} [\log_{w_{i}\circ w^{-1}_{i-(1,0)}\circ w_{i}}w_{i+(1,0)}]\\
		&+D {\mathcal R}_{w^{-1}_{i}\circ w_{i-(1,0)}} [\log_{w_{i}\circ w^{-1}_{i-(1,0)}\circ w_{i}}w_{i+(1,0)}] \big)
 \end{align}
	and
 \begin{align}
	\grad_{\mathcal{M},w_{i}} 
	&(\mathrm{d}_{xy}^{\mathrm{Lie}} \cdot)^2_{i}(w)
	\\
	&=-2\log_{w_{i}} (w_{i-(0,1)} \circ w^{-1}_{i + (1,-1)}\circ w_{i+(1,0)}).
	\end{align}
\end{lemma}

The proof is given in Appendix.

The summands in the TGV prior on Lie groups \eqref{prior_tgv_lie}
have a similar structure as those in the Lie group IC model. 
The computation is even simpler since no argument exists twice in one distance term. 
Therefore the gradients can be calculated by isolating the arguments of interest on one side of the distance function 
and then apply the chain rule with \eqref{eq:der:dsq}.

It remains to consider the gradient of the $\TGV^{\mathrm{int}}$ prior in \eqref{mod_mani_tgv_pole}. 
Due to symmetries we may stick to the one-dimensional case. 
Further, we abbreviate the differentials from Lemma~\ref{pre:differentials}  by
	\begin{align}
	E_{x}(u) 
	&\coloneqq D(\exp_{x}\cdot)(\xi)\colon T_x \mathcal{M}\to T_y \mathcal{M}, \ y \coloneqq \exp_x \xi,\\
	\tilde E_{u}(x) 
	&\coloneqq D\left(\exp_{\cdot}(\xi) \right)(x) \colon T_x \mathcal{M}\to T_y \mathcal{M},\\
	L_{x}(y) 
	&\coloneqq D(\log_{x}\cdot)(y)\colon T_y \mathcal{M}\to T_x \mathcal{M},\\
	\tilde L_{y}(x) 
	&\coloneqq D\left(\log_{\cdot}(y) \right) (x)\colon T_x \mathcal{M}\to T_x \mathcal{M},\\
	G_{\cdot,y,\tau}(x) 
	&\coloneqq D\gamma(\cdot,y;\tau)(x) \colon T_x \mathcal{M}\to T_{\geodesic{x,y}(\tau)} \mathcal{M},\\
	G_{x,\cdot, \tau}(y) 
	&\coloneqq D\gamma(x,\cdot;\tau)(y) \colon T_y \mathcal{M}\to T_{\geodesic{x,y}(\tau)} \mathcal{M}.
	\end{align}
Then the gradients can be derived by the chain rule and the following lemma.

\begin{lemma} \label{lem:grad_tgv}
The functions
 \begin{align}
	&F_1  (u_i,u_{i+1},\xi_i)
	\coloneqq \lVert\log_{u_i} u_{i+1}-\xi_i \rVert^2_{u_i},\\
	&F_2(u_i,u_{i-1},\xi_i,\xi_{i-1})
	\coloneqq
	\lVert \xi_i - P_{u_{i-1}\to u_i}^\mathrm{P}(\xi_{i-1})\rVert^2_{u_i}	
 \end{align}
have the Riemannian gradients 
\begin{align} \label{grad_simple_1}
&\grad_{\MM, \xi_i} F_1 (u_i,u_{i+1}, \cdot) (\xi_i) 
= 
2( \xi_i - \log_{u_i} u_{i+1}) \eqqcolon -T,
 \\
&\grad_{\mathcal{M},u_i}F_1(\cdot,u_{i +1},\xi_{i}\bigr)(u_{i}) = \tilde L^*_{u_{i+1}}(u_{i})[T]\\
&\grad_{\mathcal{M},u_{i+1}}F_1(u_{i},\cdot,\xi_{i}\bigr)(u_{i+1}) = L^*_{u_{i}}(u_{i+1})[T],
\end{align}
and 
\begin{align}
\grad&_{\mathcal{M},\xi_{i}} F_2(u_i,u_{i-1},\cdot,\xi_{i-1})(\xi_{i}) 
\\&=
 2(\xi_{i}- P_{u_{i-1}\to u_i}^{\mathrm{P}}(\xi_{i-1}))
\eqqcolon -S ,
\\
\grad&_{\mathcal{M},\xi_{i-1}}\bigl(F_2(u_{i},u_{i -1},\xi_{i},\cdot)\bigr)(\xi_{i-1}) \\
&=
E^*_{u_{i-1}}(\xi_{i-1})
	\Bigl[G^*_{\cdot,c_{i},2}(e_i)
	\bigl[L^*_{u_{i}}(p_{i})[S]\bigr]\Bigr],
	\\
\grad&_{\mathcal{M},u_i} \bigl(F_2(\cdot,u_{i -1},\xi_{i},\xi_{i-1})\bigr)(u_{i}) \\
&=	
	\tilde{L}^*_{p_{i}}(u_{i})[S]	
	    \\
	    &\qquad+G^*_{\cdot,u_{i-1},\frac{1}{2}}(u_{i})
	    \Bigl[G^*_{e_i,\cdot,2}(c_{i})
	    \bigl[L^*_{u_{i}}(p_{i})
	    [S]\bigr]\Bigr]  ,
	    \\
\grad&_{\mathcal{M},u_{i-1}} \bigl( F_2( u_{i},\cdot,\xi_{i},\xi_{i-1}) \bigr)(u_{i -1})\\
&=
\tilde E_{\xi_{i-1}}^*(u_{i-1})
\Bigl[
	    G^*_{\cdot,c_{i},2} (e_{i})
	    \bigl[
	    L^*_{u_{i}}(p_{i})[S]
	    \bigr]
\Bigr]
	    \\
	    &\qquad+ G^*_{u_{i},\cdot,\frac{1}{2}}(u_{i-1})
	    \Bigl[G^*_{e_i,\cdot,2}(c_{i})
	    \bigr[L^*_{u_{i}}(p_{i})[S]\bigr]
	    \Bigr],
\end{align}
where we set
$e_{i} \coloneqq \exp_{u_{i-1}} (\xi_{i-1})$,
$c_{i}  \coloneqq \gamma\bigl(u_{i},u_{i-1};\tfrac{1}{2}\bigr)$ and	
$p_{i} \coloneqq \gamma(e_{i},c_{i};2)$.
\end{lemma}
The proof is given in the appendix.

\section{Numerical examples} \label{sec:Num}

The gradient descent algorithm and ADMM
are implemented in \textsc{Matlab}. The basic manifold functions, like logarithmic and exponential maps, 
as well as the distance function are implemented as C\texttt{++} functions within the \href{http://www.mathematik.uni-kl.de/imagepro/members/bergmann/mvirt/}{\glqq Manifold-valued 
	Image Restoration Toolbox\grqq (MVIRT)}\footnote{http://www.mathematik.uni-kl.de/imagepro/members/bergmann/mvirt/}
and imported into \textsc{Matlab} using \lstinline!mex!-interfaces with the GCC 4.8.4 compiler.
As a quality measure we use the
mean squared error (MSE) defined by
\begin{equation}
\epsilon\coloneqq
\tfrac{1}{\lvert\grid\rvert}
\sum_{{i} \in\grid}\dist^2(u_{i},u_{0,{i}}),
\end{equation}
where $u_0$ denotes the original image. 
The parameters  in the artificial examples are adapted 
via grid search to minimize this measure~\(\epsilon\). The relaxation parameter
$\varepsilon$ is chosen for each experiment based on the data.\\
The algorithm stops if one of the following criteria is fulfilled
\begin{itemize}
	\item the maximal change is small enough, i.e., 
	\begin{equation}
	\max_{i\in\grid}\{\dist(p_i^{(r)},p_i^{(r-1)})\} < \delta,
	\end{equation}
	with $\delta = 10^{-10}$ for signals and $\delta = 10^{-8}$ for images;
	\item the number of iterations exceeds $c\in\mathbb{N}$, i.e., $r> c$, with $c = 10^{6}$ for signals and $c = 10^{5}$ for images.
\end{itemize}

\subsection{$\mathbb{S}^1$-valued data}
\begin{figure*}\centering
  \begin{subfigure}[t]{.315\textwidth}
	  \includegraphics[width = 0.98\textwidth]{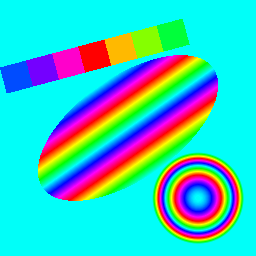}
    \caption{Original image.}\label{subfig:arts1:orig}
  \end{subfigure}
  \begin{subfigure}[t]{.315\textwidth}
	  \includegraphics[width = 0.98\textwidth]{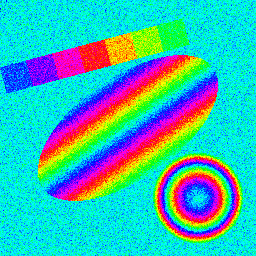}
    \caption{Noisy image, \(\epsilon = 88.5\times10^{-3}\)}
    \label{subfig:arts1:noisy}
  \end{subfigure}
  \begin{subfigure}[t]{.315\textwidth}
	  \includegraphics[width = 0.98\textwidth]{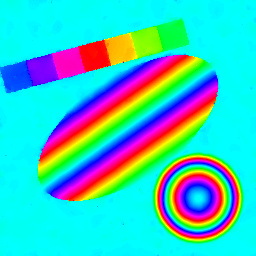}
    \caption{Pole ladder TGV, $\epsilon=2.6\times10^{-3}$.}
    \label{subfig:arts1:pole}
  \end{subfigure}
  \begin{subfigure}[t]{.315\textwidth}
	  \includegraphics[width = 0.98\textwidth]{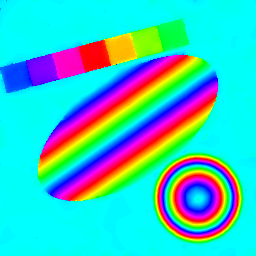}
    \caption{Lie group TGV, $\epsilon=2.5\times10^{-3}$.}
    \label{subfig:arts1:Lie}
  \end{subfigure}
  \begin{subfigure}[t]{.315\textwidth}
	  \includegraphics[width = 0.98\textwidth]{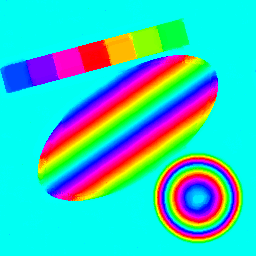}
    \caption{Graph Laplacian~\cite{BT17} $\epsilon=2.6\times10^{-3}$.}
    \label{subfig:arts1:laplace}
  \end{subfigure}
  \begin{subfigure}[t]{.315\textwidth}
    \includegraphics[width = 0.98\textwidth]{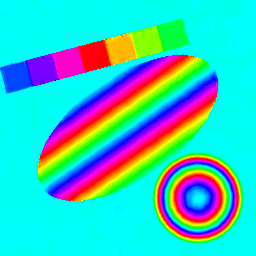}
    \caption{NL-MMSE~\cite{LNPS17}, $\epsilon=2.5\times10^{-3}$.}
    \label{subfig:arts1:NLMSSE}
  \end{subfigure}
	\caption{Denoising of an artificial $\mathbb{S}^1$ image with different methods.}\label{fig:arts1}
\end{figure*}

We start with the $\mathbb S^1$-valued image in Fig.~\ref{fig:arts1}\,\subref{subfig:arts1:orig} from \cite{BLSW14}.
Adding wrapped Gaussian noise results in the corrupted image~\subref{subfig:arts1:noisy}.
In~\cite{BLSW14} the additive $\TV_{1\wedge2}$ yields to an error of $\epsilon = 5.4\times10^{-3}$. Comparing this result to the pole ladder TGV ($\alpha = 1$, $\beta = 0.3$, $\varepsilon = 10^{-3}$), see~\subref{subfig:arts1:pole}
and Lie group TGV ($\alpha = 1$, $\beta = 0.3$, $\varepsilon = 10^{-6}$),
cf.~\subref{subfig:arts1:Lie}, we see that the TGV models yield a smaller error.
These models are able to nicely reconstruct the linear parts in the ellipsoid 
and the edges of the boxes. 
Compared to the nonlocal methods~\cite{LNPS17} and~\cite{BT17}
shown in~\subref{subfig:arts1:laplace} and~\subref{subfig:arts1:NLMSSE}, respectively, the TGV models have nearly the same error.
However, looking at the paraboloid in the bottom right corner,
they outperform the nonlocal methods visually.

\subsection{$\mathbb S^2$-valued data}
\begin{figure*}\centering
	\begin{subfigure}{0.48\textwidth}
		\centering
		\includegraphics[width=.8\textwidth]{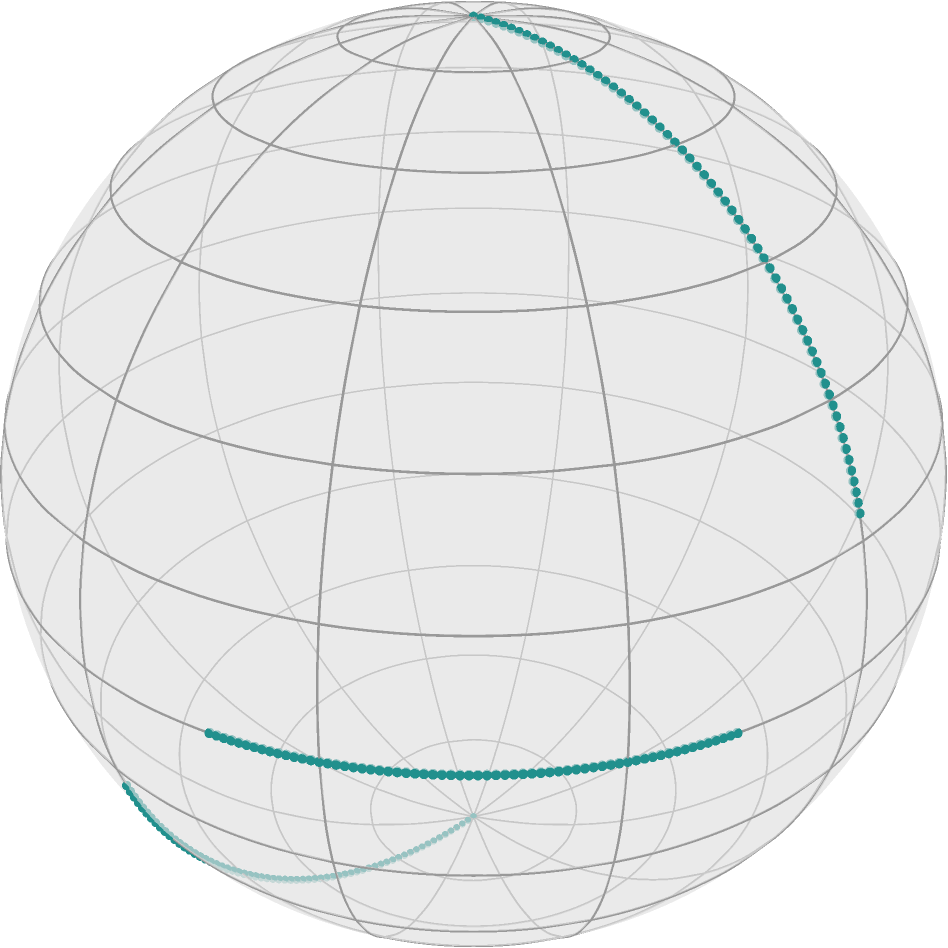}
		\caption{Piecewise geodesic signal.}
		\label{ictgv:fig:S2Decomp:orig}
	\end{subfigure}
	\begin{subfigure}{0.48\textwidth}
	\centering
	\includegraphics[width=.8\textwidth]{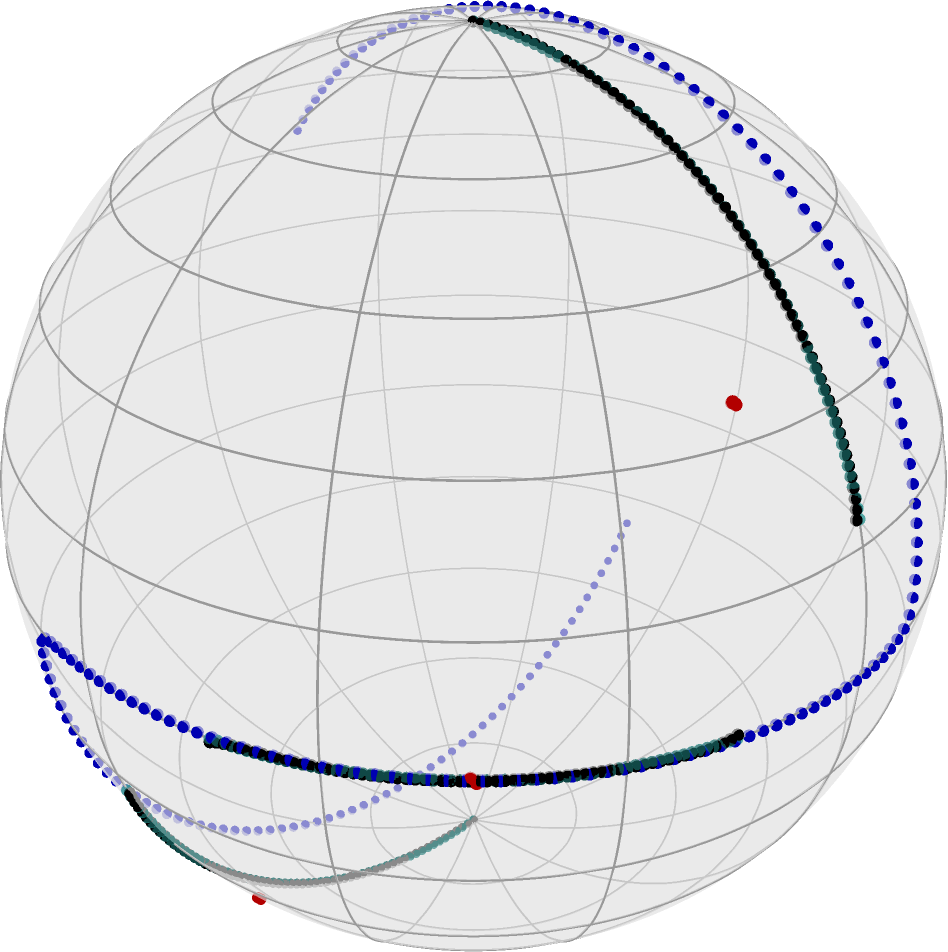}
	\caption{Decomposed signal.}	
	\label{ictgv:fig:S2Decomp:decomp}
	\end{subfigure}
	\caption[]{Midpoint IC decomposition of a piecewise geodesic signal \(f\) (green), 
	into a piecewise constant part \(v\) (red) and a continuous piecewise geodesic curve \(w\) (blue). 
	The mid point signal \(u = \gamma_{\overset{\frown}{v,w}} (\tfrac12)\) (black) nearly reconstructs \(f\) (green).}
	\label{fig:S2Decomp}
\end{figure*}

Now we are interested in $\mathbb S^2$ valued signals, where the Lie group approach cannot be applied.
First, we are interested in the decomposition of a signal.
The ground truth signal in {{ Fig. \ref{fig:S2Decomp}} is obtained as follows:
we take three great arcs from the north pole to the equator, a quarter great arc along the equator, and from thereon further to the south pole. We scale the segments by~\(\frac{1}{5}\), \(\frac{3}{20}\), and~\(\frac{1}{5}\), respectively such that we obtain jumps between the three geodesic segments. 
This yields a signal of length \(192\) shown in Fig.~\ref{fig:S2Decomp}\,\subref{ictgv:fig:S2Decomp:decomp}.
We apply the Midpoint IC model 
with~\(\alpha = \frac{11}{100}, \beta = \frac{1}{11}\).
The result $u$ approximates \(f\) and its decomposition into~\(v\)
and~\(w\) yields signals that are nearly piecewise constant and piecewise
geodesic, respectively.

\begin{figure*}
	\begin{subfigure}{0.24\textwidth}
		\centering
		\includegraphics[width=.98\textwidth]{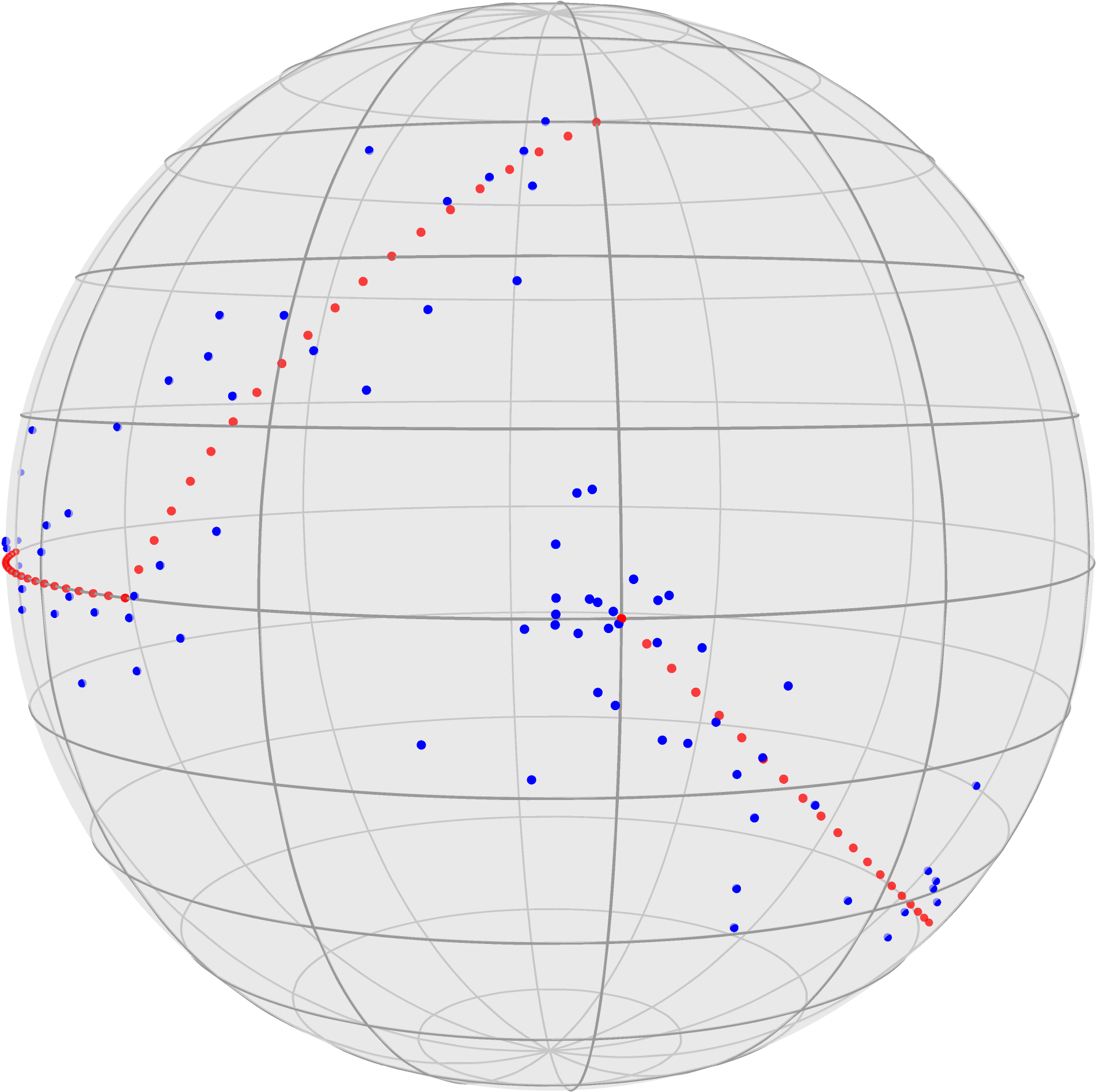}
		\caption{Noisy signal, $\epsilon = 0.0168$.}
		\label{fig:s2tgv:orig}
	\end{subfigure}
	\begin{subfigure}{0.24\textwidth}
		\centering
		\includegraphics[width=.98\textwidth]{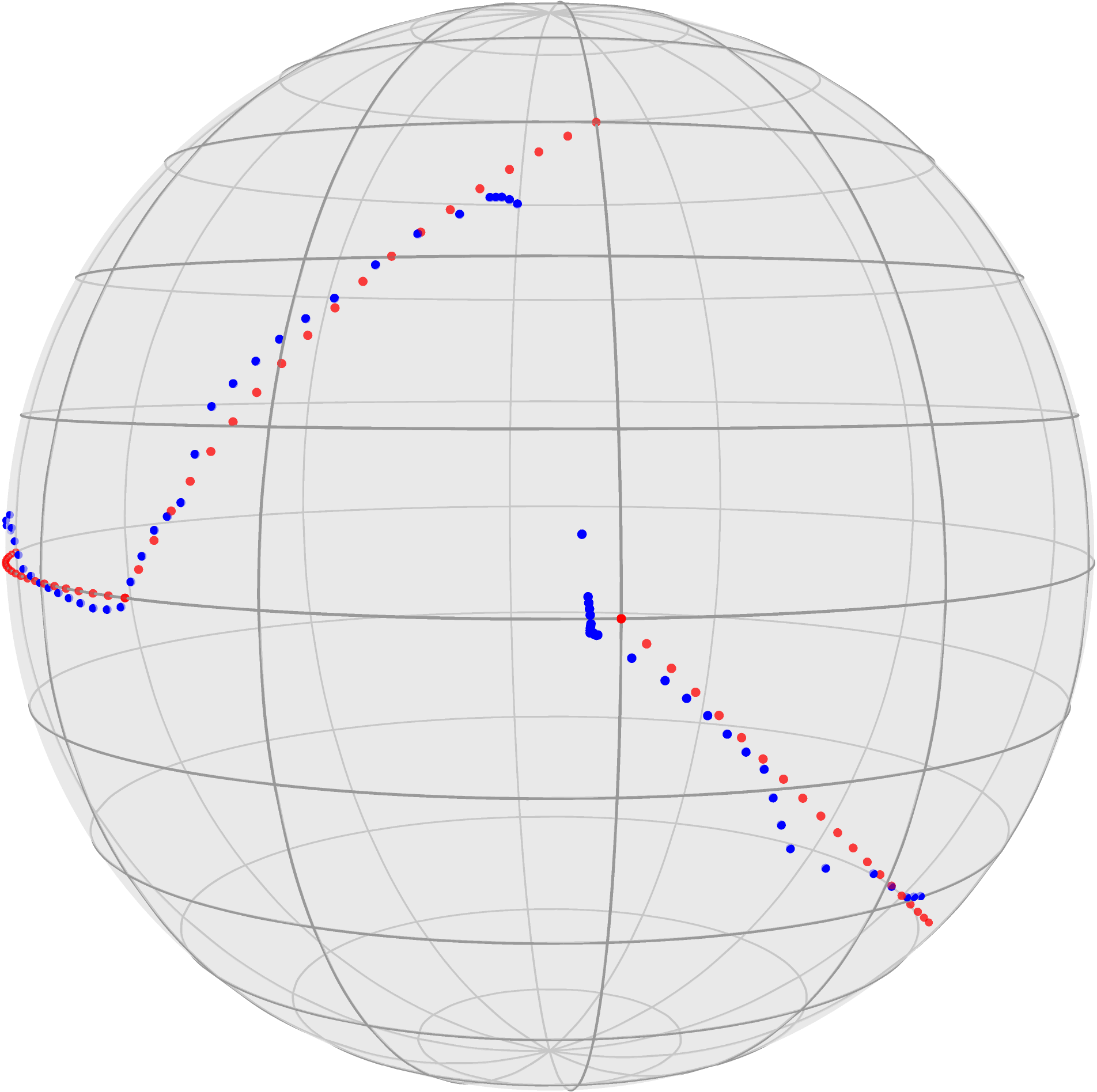}
		\caption{Denoising using ${\mathcal E}^{\textrm{int}}_{\textrm{ADD}}$, $\epsilon = 0.0034 $.}	
		\label{fig:s2tgv:add}
	\end{subfigure}
	\begin{subfigure}{0.24\textwidth}
		\centering
		\includegraphics[width=.98\textwidth]{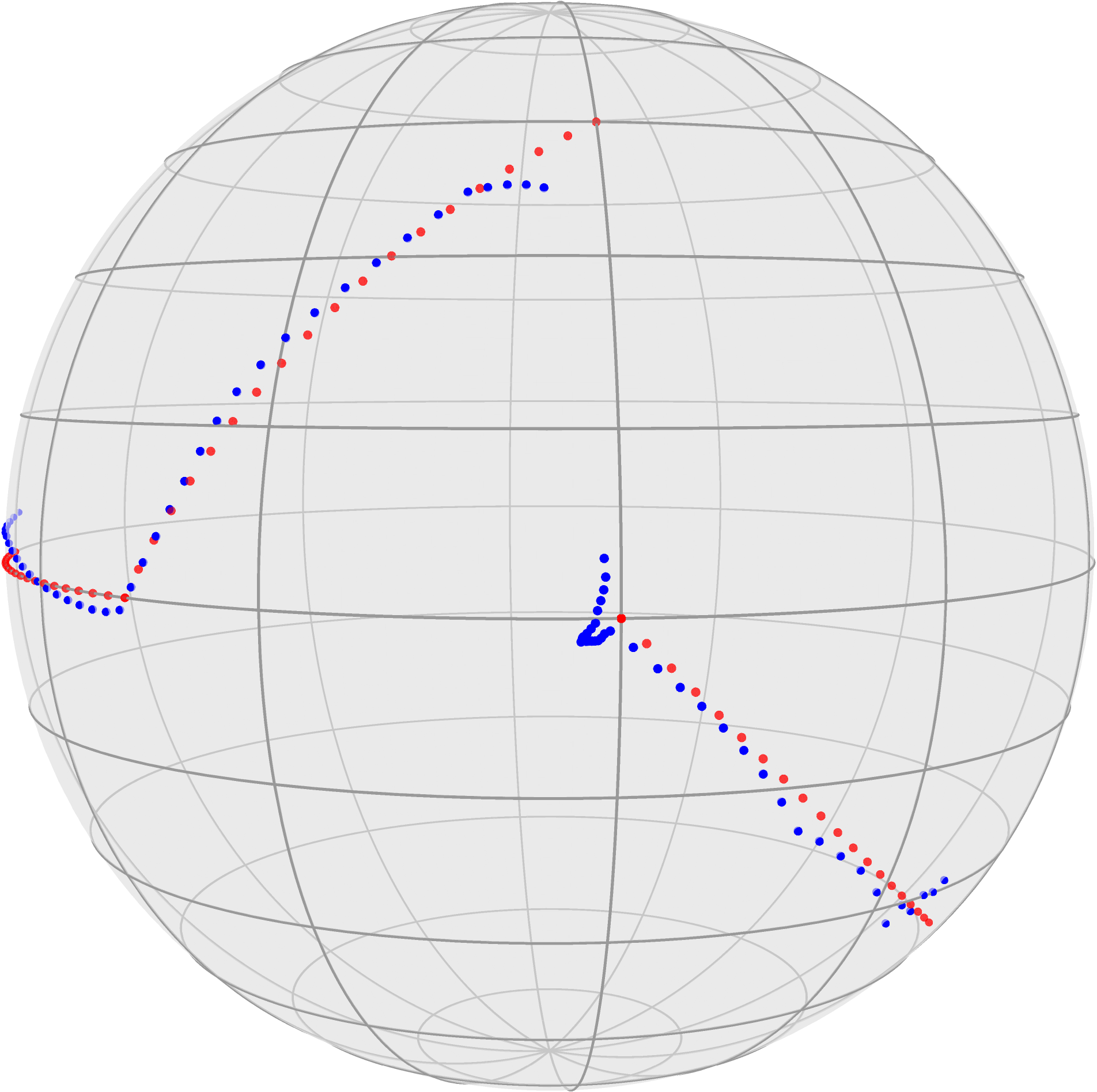}
		\caption{Denoising using $E^{\textrm{int}}_{\IC}$,\\ $\epsilon = 0.0025$.}	
		\label{fig:s2tgv:midic}
	\end{subfigure}
	\begin{subfigure}{0.24\textwidth}
		\centering
		\includegraphics[width=.98\textwidth]{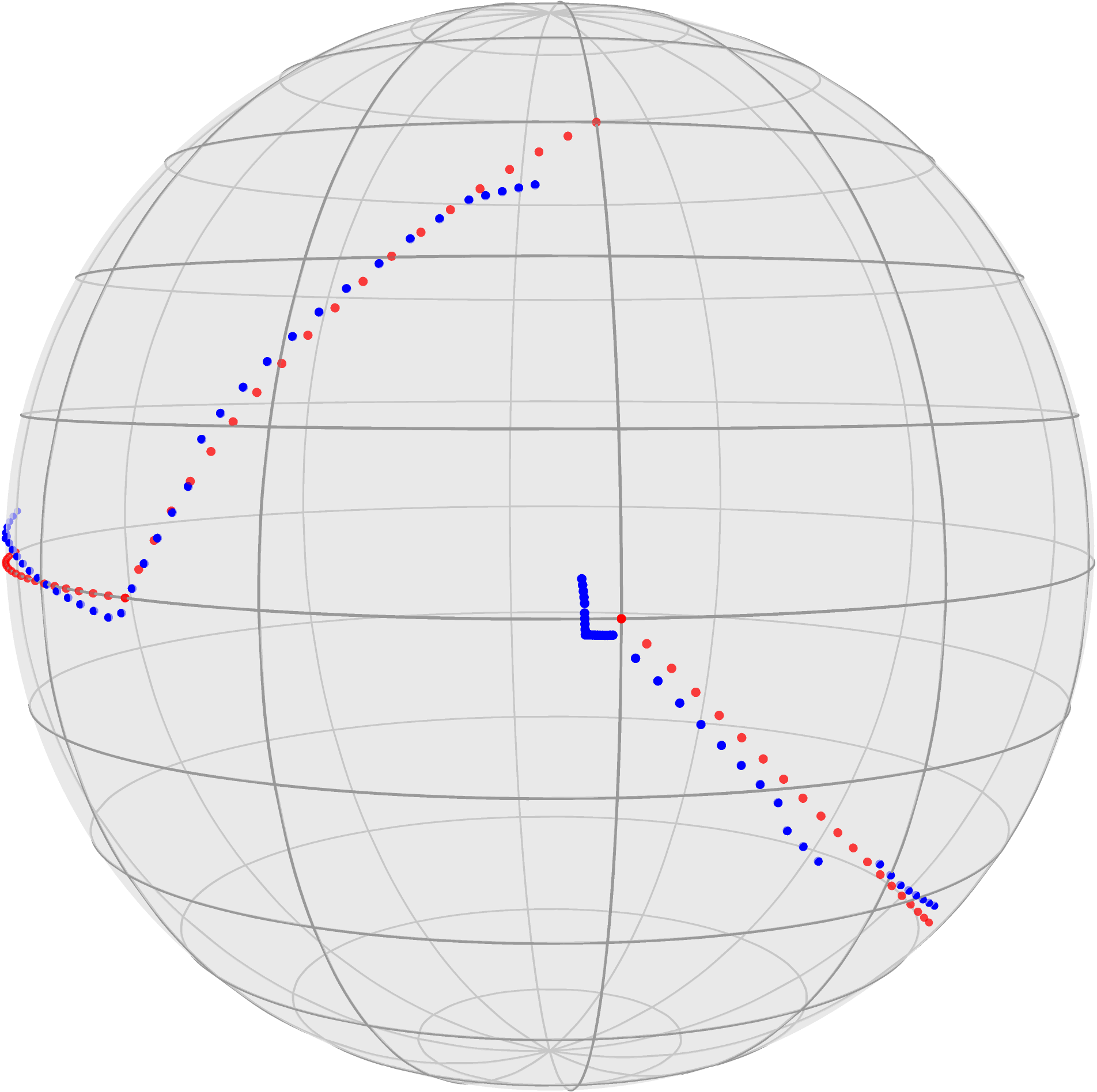}
		\caption{Denoising using $E^{\textrm{P}}_{\TGV}$,\\ $\epsilon = 0.0022$.}	
		\label{fig:s2tgv:pole}
	\end{subfigure}
	\caption[]{Denoising of a $\mathbb{S}^2$-valued signal with additive and TGV priors. 
	The original signal is plotted in red and the noisy/restored signals in blue.}
	\label{fig:s2tgv}
\end{figure*}

Next we present a denoising result.
In Fig.~\ref{fig:s2tgv} we compare intrinsic additive model with the Midpoint $\IC$ and
pole ladder $\TGV$ approach.
The original signal consists of four segments of length 20; 
the first two are geodesic segments,
then there is a jump to a constant segment, and the last segment is again
geodesic. 
Fig.~\ref{fig:s2tgv:orig} shows the original and the corrupted signal with
Gaussian noise ($\sigma = 0.1$). The parameters for the additive model are 
$\alpha\beta = \frac{1}{10}$ and $\alpha(1-\beta)= 4.6$, for the Midpoint IC $\alpha = \frac{1}{2},\ \beta=0.3$, $\varepsilon = 10^{-3}$
and for the pole ladder TGV model $\alpha = \frac{3}{5},\ \beta = 0.3$ and we relax by~\(\varepsilon = 10^{-5}\).
The results are shown in ~\ref{fig:s2tgv} (b) to (d).   
In comparison to the additive model the IC and TGV models preserve the jump better and yield a lower error. Comparing the IC and TGV model we see that the TGV model is more suited to reconstruct the geodesic parts. Hence, it yields the lowest error $\epsilon$.

\subsection{SPD-valued data}

\begin{figure}[tb]
	\begin{align*}
	u_0\colon& \vcenter{\hbox{
			\includegraphics[width =0.4\textwidth]{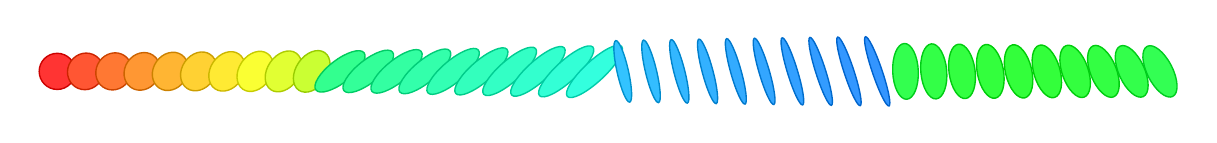}}}
	\\[-.25\baselineskip]
	f\colon& \vcenter{\hbox{
			\includegraphics[width = 0.4\textwidth]{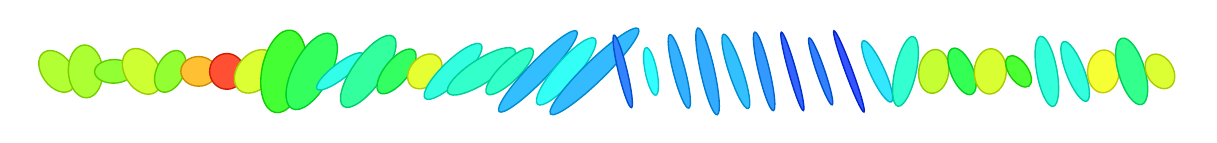}}}
	\\[-.25\baselineskip]
	u_{\mathrm{ADD}}^{\mathrm{int}}\colon& \vcenter{\hbox{
			\includegraphics[width = 0.4\textwidth]{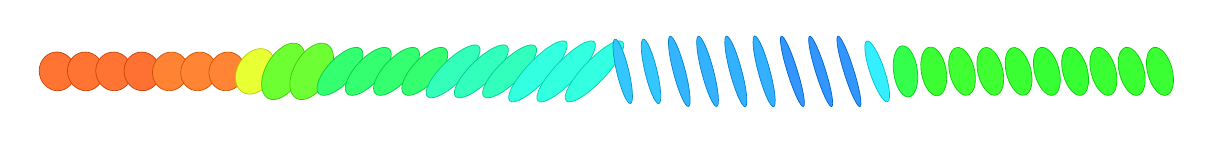}}}
	\\[-.25\baselineskip]
	u_{\mathrm{TGV}}^{\mathrm{int}}\colon& \vcenter{\hbox{
			\includegraphics[width = 0.4\textwidth]{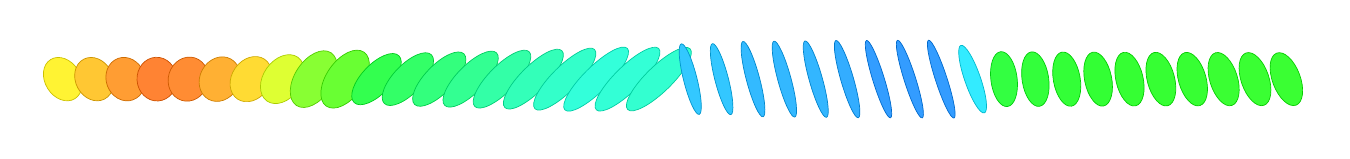}}}
	\\[-.25\baselineskip]
	u_{\IC}^{\mathrm{int}}\colon& \vcenter{\hbox{
			\includegraphics[width = 0.4\textwidth]{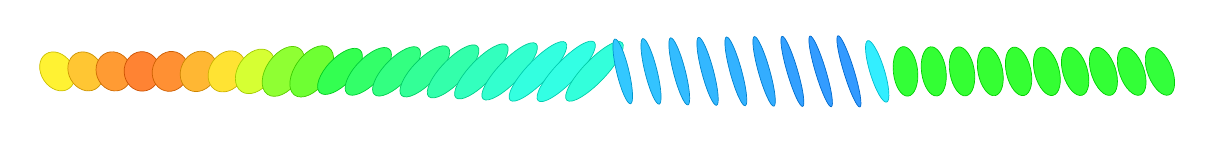}}}
	\\[-.25\baselineskip]
	v_{\IC}^{\mathrm{int}}\colon& \vcenter{\hbox{
			\includegraphics[width = 0.4\textwidth]{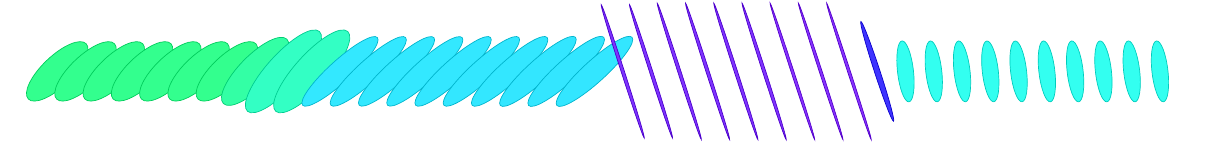}}}
	\\[-.25\baselineskip]
	w_{\IC}^{\mathrm{int}}\colon& \vcenter{\hbox{
			\includegraphics[width = 0.4\textwidth]{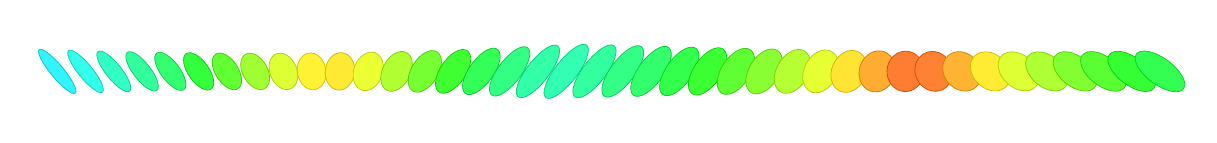}}}
	\end{align*}
	\caption{Denoising and decomposition of a~\(\SPD{2}\) valued signal $f$. 
	Denoising result $u_{\mathrm{ADD}}^{\mathrm{int}}$ 
	by additive model with prior~$\operatorname{TV}_{1 \wedge 2}^{\mathrm{int}}$, $\epsilon = 0.0316$, with the TGV model, $\epsilon = 0.0259$, 
	and $u_{\IC}^{\mathrm{int}}$ by Midpoint IC model,  $\epsilon = 0.0269$.
	Decomposition by Midpoint IC model gives~$v_{\IC}^{\mathrm{int}}$ and geodesic part~$w_{\IC}^{\mathrm{int}}$
	with geodesic midpoint $u_{\IC}^{\mathrm{int}}$.
		}\label{fig:spd2}
\end{figure}

In this subsection, we consider the decomposition and denoising of
SPD-valued data.
{{Fig.~\ref{fig:spd2}} shows a signal $u_0$ with values in
${\mathcal P} (2)$
which is the midpoint of a signal with four constant
parts and one with two geodesic parts.
The signal $f$ is its noisy version with Gaussian noise.
We apply the intrinsic additive model 
($\alpha = 0.46, \beta = 1$), the intrinsic TGV model ($\alpha = 2.5, \beta =  0.2$),
and the Midpoint IC
model ($\alpha = 4.5,\beta = \frac{1}{9}$) 
to $f$. Here we use $\varepsilon = 10^{-3}$ as relaxation.
With respect to both  the MSE and visually the TGV and Midpoint IC model outperform the additive one. 
In particular, the smooth parts are better reconstructed. The results from the sophisticated intrinsic priors are very similar and visually not distinguishable.
The components $v$ and $w$ from the IC model give a decomposition of $u_0$ into a
piecewise constant and geodesic component.
\begin{figure*}
	\centering
	\begin{subfigure}{0.32\textwidth}
		\centering
		\includegraphics[width=.98\textwidth]{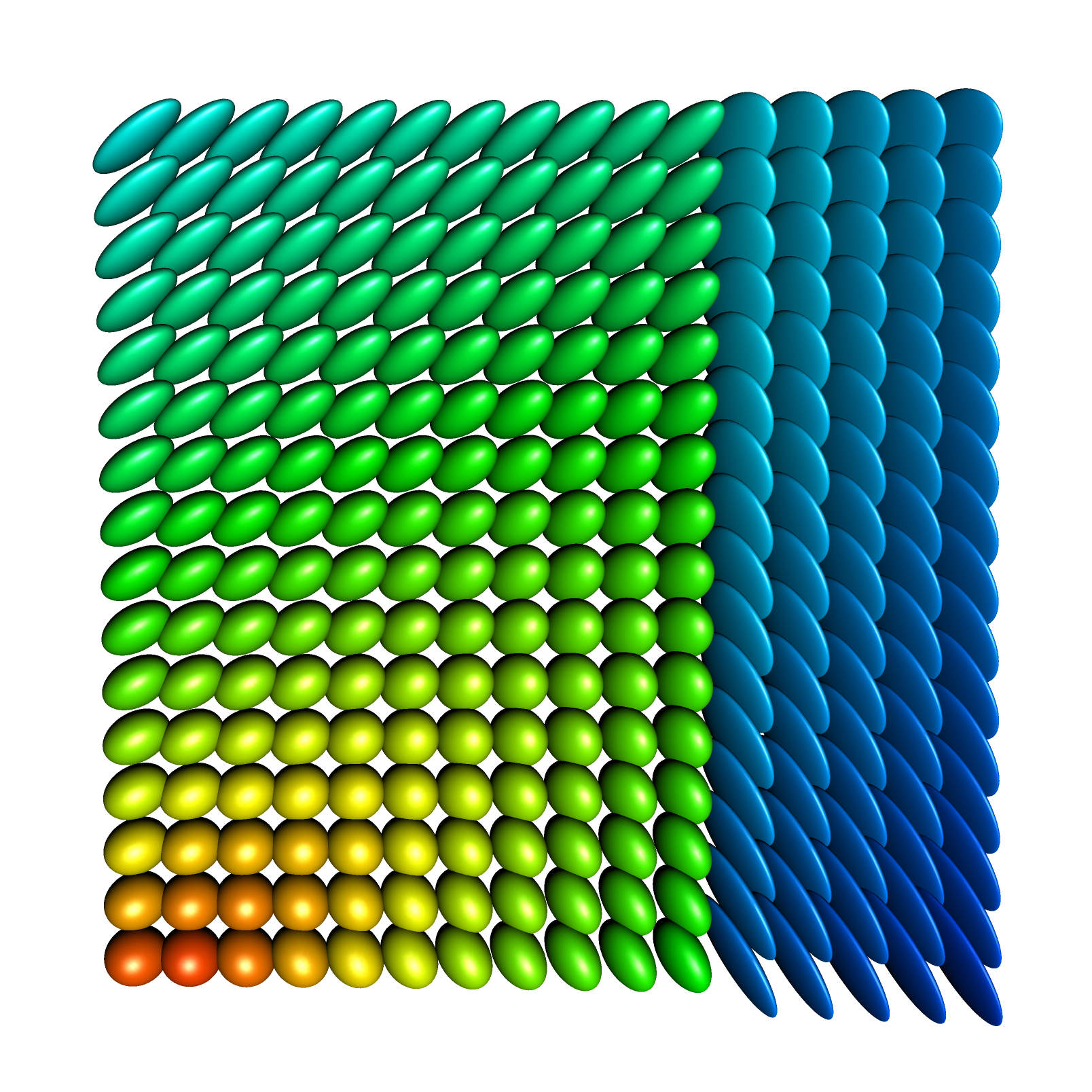}
		\caption{Original image.}\label{ictgv:fig:SPDImgIC:orig}
	\end{subfigure}
	\begin{subfigure}{0.32\textwidth}
		\centering	
		\includegraphics[width=.98\textwidth]{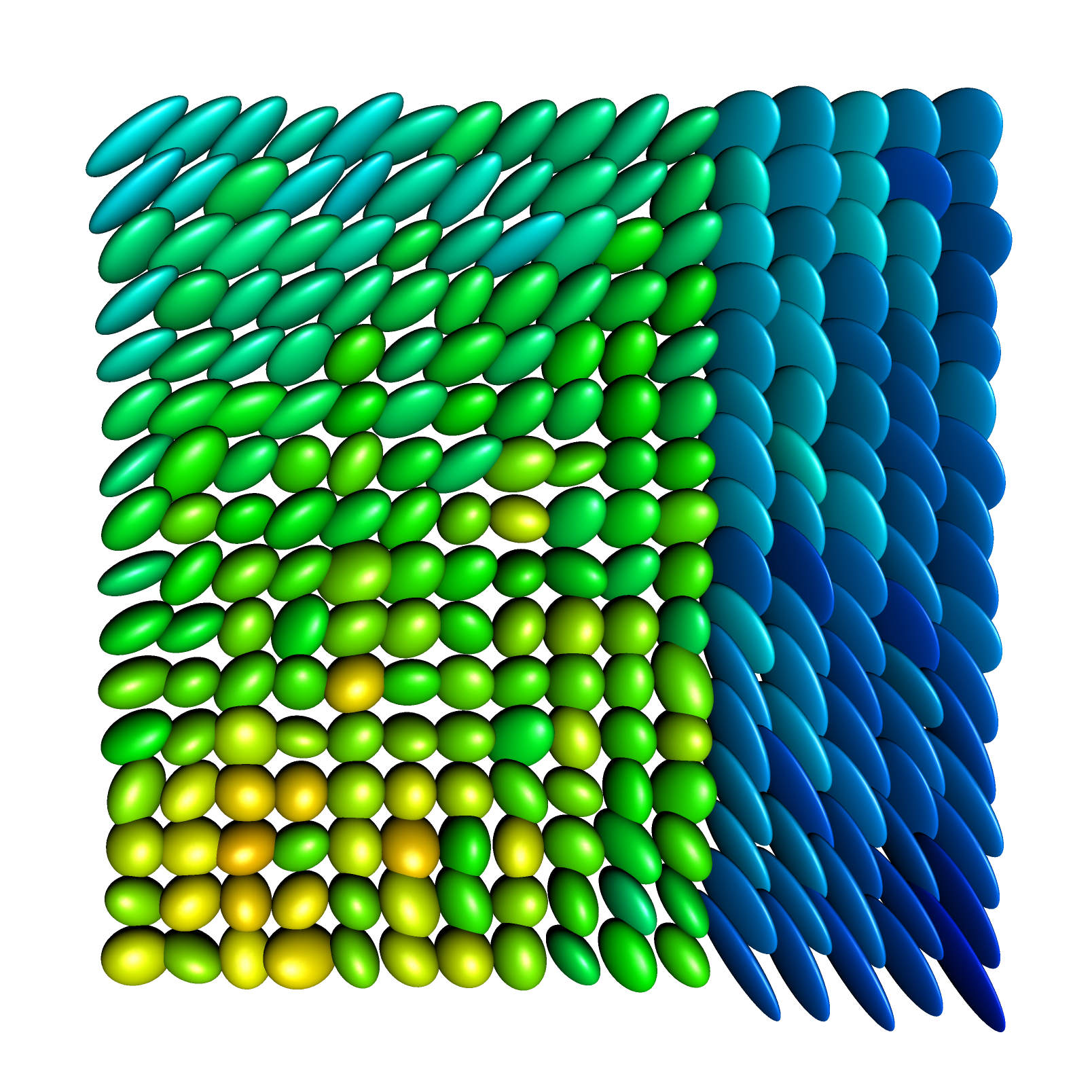}
		\caption{Noisy image, $\epsilon=0.0583$.}\label{ictgv:fig:SPDImgIC:noisy}
	\end{subfigure}\\	
	\begin{subfigure}{0.32\textwidth}
		\centering	
		\includegraphics[width=.98\textwidth]{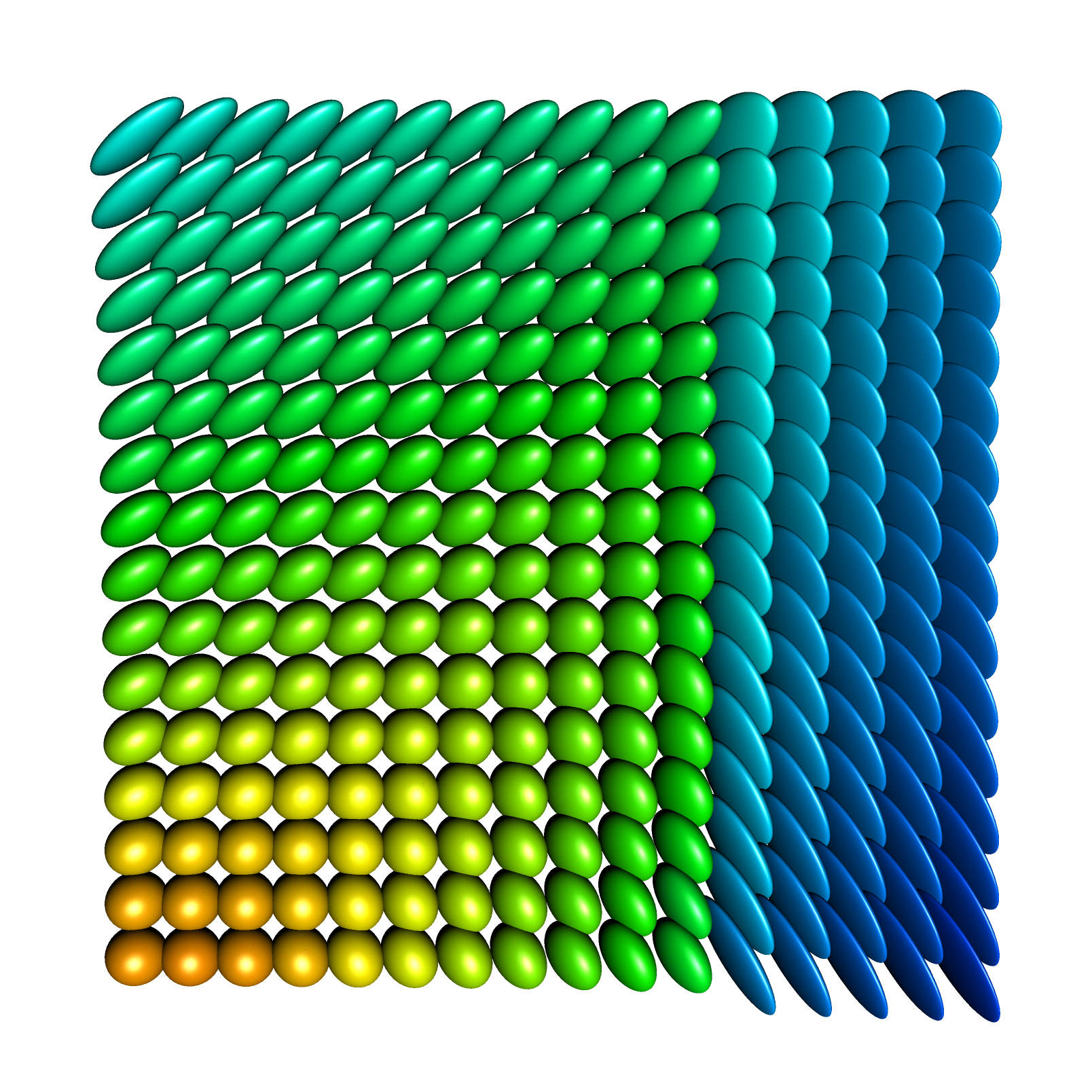}
		\caption{$u^{\mathrm{ext}}_{\IC}$, $\epsilon=0.0066$.}\label{ictgv:fig:SPDImgIC:ext}
	\end{subfigure}
	\begin{subfigure}{0.32\textwidth}
		\centering	
		\includegraphics[width=.98\textwidth]{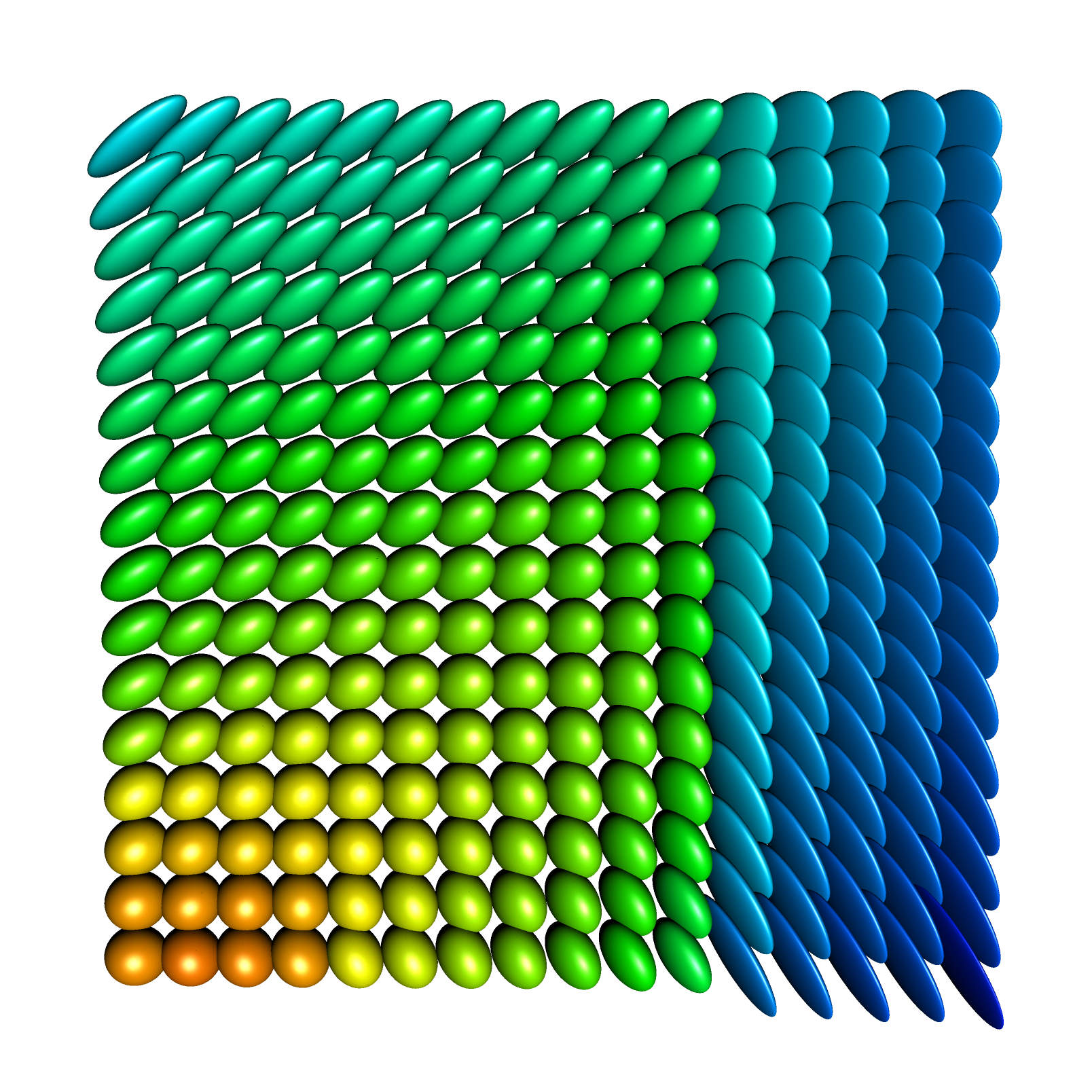}
		\caption{$u^{\mathrm{ext}}_{\TGV}$, $\epsilon=0.0065$.}\label{ictgv:fig:SPDImgTGV:ext}	
	\end{subfigure}
	\begin{subfigure}{0.32\textwidth}
		\centering	
		\includegraphics[width=.98\textwidth]{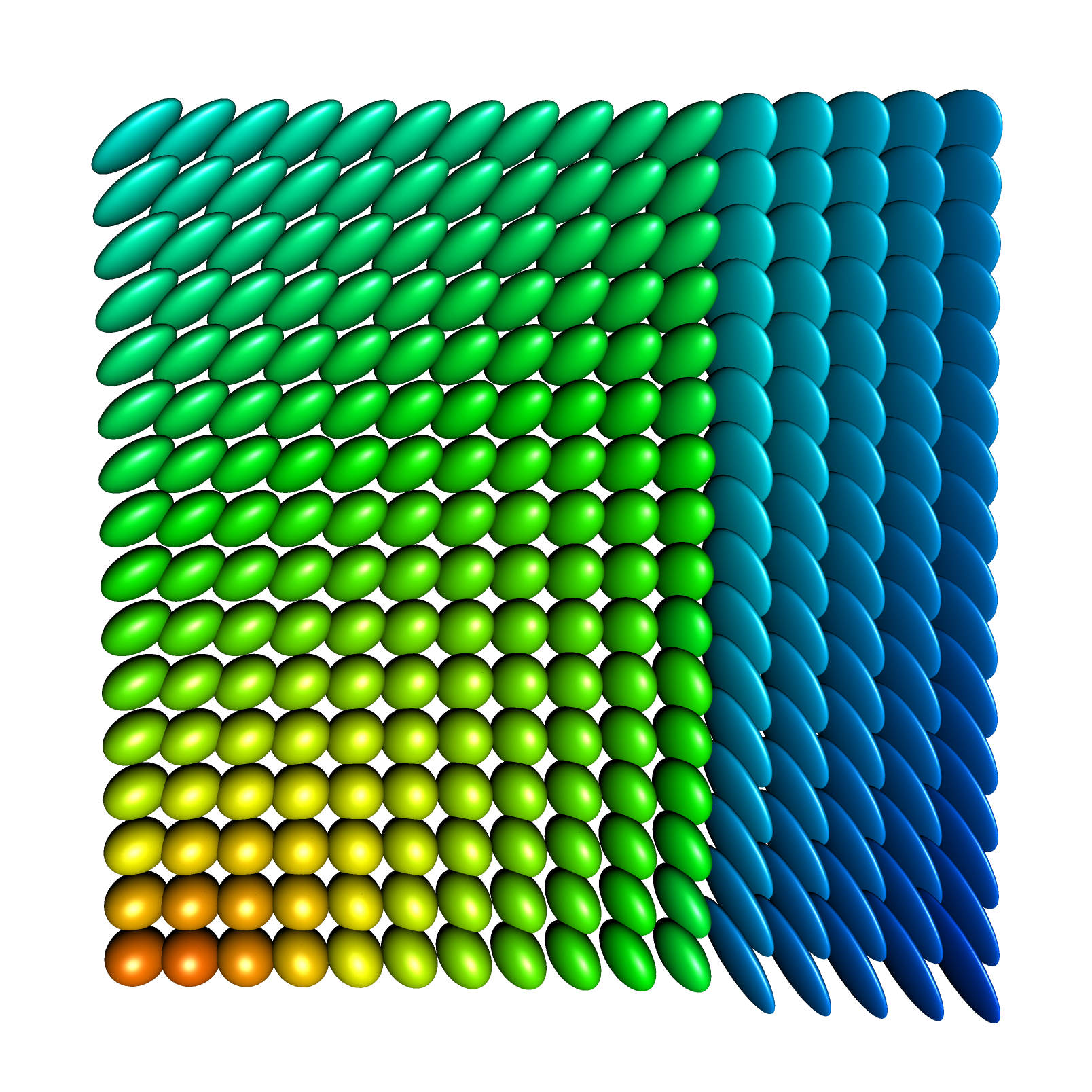}
		\caption{$u^{\mathrm{int}}_{\TGV}$, $\epsilon=0.0034$.}\label{ictgv:fig:SPDImgTGV:int}
	\end{subfigure}
\\
	\begin{subfigure}{0.32\textwidth}
		\centering
		\includegraphics[width=.98\textwidth]{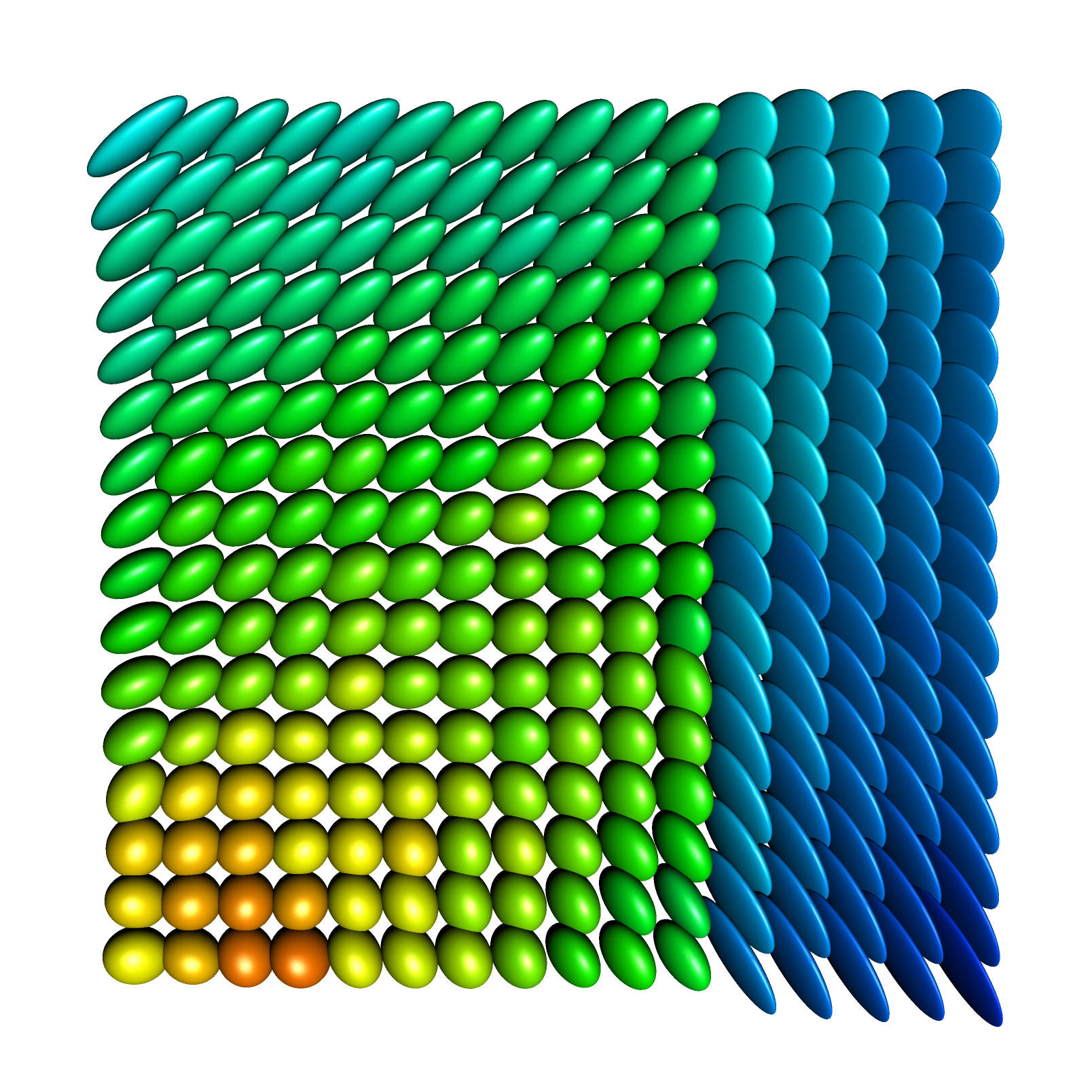}
		\caption{$u^{\mathrm{int}}_{\IC}$, $\epsilon=0.0125$.}\label{ictgv:fig:SPDImgIC:recon}
	\end{subfigure}
	\begin{subfigure}{0.32\textwidth}
		\centering
		\includegraphics[width=.98\textwidth]{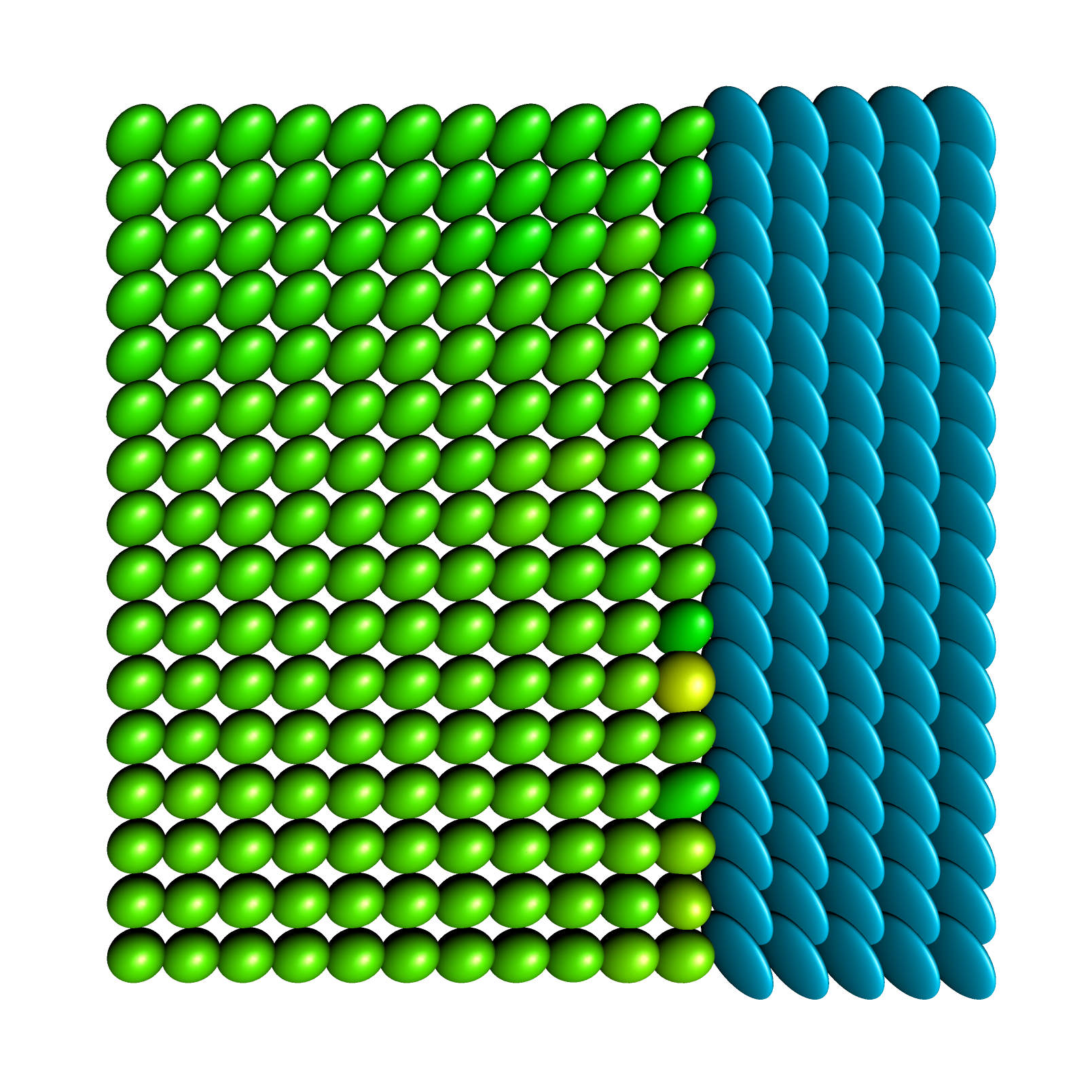}
		\caption{$v^{\mathrm{int}}_{\IC}$.}\label{ictgv:fig:SPDImgIC:v}
	\end{subfigure}
	\begin{subfigure}{0.32\textwidth}
		\centering	
		\includegraphics[width=.98\textwidth]{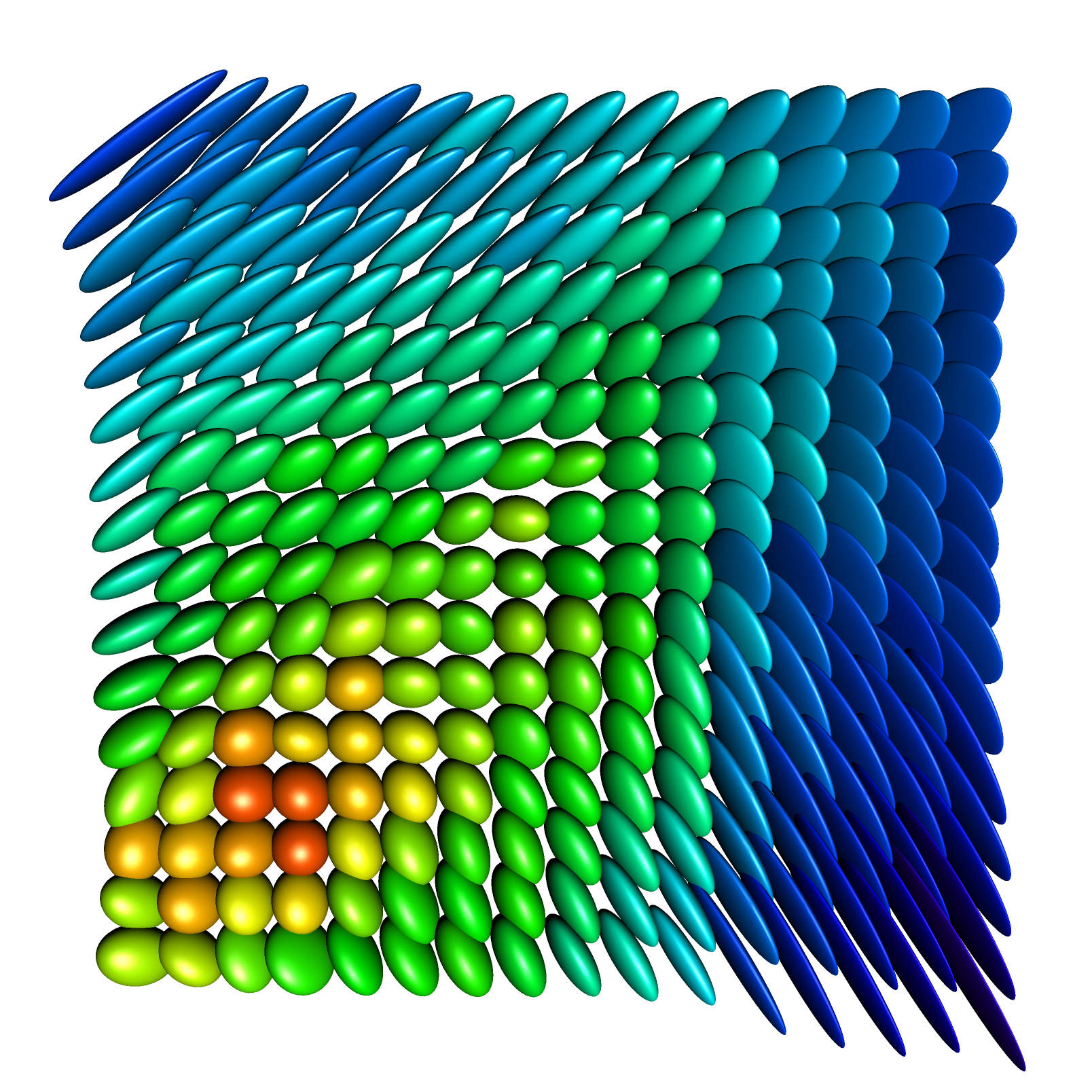}
		\caption{$w^{\mathrm{int}}_{\IC}$.}\label{ictgv:fig:SPDImgIC:w}
	\end{subfigure}
	\caption[]{Denoising of an artificial $\SPD{3}$-valued image with extrinsic and Midpoint $\IC$ model.
	}
	\label{ictgv:fig:SPDImgIC}
\end{figure*}

Fig.~\ref{ictgv:fig:SPDImgIC:orig}
shows an artificial \(\SPD{3}\)-valued image which is 
corrupted by Gaussian noise (\(\sigma=0.1\)) resulting in \subref{ictgv:fig:SPDImgIC:noisy}. 
The result of the extrinsic TGV model $u^{\mathrm{ext}}_{\TGV}$ ($\alpha = 12$, $\beta = 0.9$) is shown in \subref{ictgv:fig:SPDImgTGV:ext} and the extrinsic IC model $u^{\mathrm{ext}}_{\IC}$ ($\alpha = 4$, $\beta = 0.4$) in \subref{ictgv:fig:SPDImgIC:ext}}.
As the geodesics in $\SPD{3}$ are not linear in the embedding it is advantageous to use intrinsic models to denoise the image. 
The pole ladder TGV model ($\alpha = 0.7,\ \beta = 0.3,\ \varepsilon = 10^{-4}$) yields the result in \subref{ictgv:fig:SPDImgTGV:int}, which has a lower error as the extrinsic methods.
Denoising with the Midpoint IC model 
(\(\alpha=\frac{1}{10}\), \(\beta=\frac{1}{2},\ \varepsilon = 10^{-5}\))
leads to the denoised image \(u^{\mathrm{int}}_{\IC}\) \subref{ictgv:fig:SPDImgIC:recon}. 
The corresponding IC components depicted in \subref{ictgv:fig:SPDImgIC:v} and \subref{ictgv:fig:SPDImgIC:w} show nicely
the piecewise constant part~\(v^{\mathrm{int}}_{\IC}\) containing the jump
and the geodesic part~\(w^{\mathrm{int}}_{\IC}\).

\subsection{SO(3)-valued data}
EBSD is often given as images having values $[f_i]$ in the quotient $\operatorname{SO}(3)/S$ of the Lie group $\operatorname{SO}(3)$,
where $S \subset \operatorname{SO}(3)/S$ denotes the symmetry group of the
crystal structure in point $i$.
EBSD images usually consist of regions with similar orientations called
 grains.  {{ Fig.~\ref{fig:ebsd}} displays a typical EBSD image of a magnesium specimen
 from the software package MTEX \cite{MTEX} which is also used for the color visualizing
 the data. For certain macroscopic properties the pattern of orientations within
 single grains is important, see e.g.,~\cite{BHJPSW10,SAK00}. 
 
  %
 \begin{figure*}
 	\centering
  \begin{subfigure}[t]{.49\textwidth}\centering
    \includegraphics[width=0.95\textwidth]{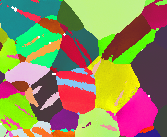}
    \caption{Color coding of magnesium.}
  \end{subfigure}
  \begin{subfigure}[t]{.49\textwidth}\centering
    \includegraphics[width=0.95\textwidth]{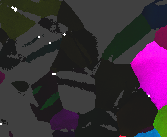} \\[1.5ex]
    \caption{Clinched color coding.}
	\end{subfigure}
   	\caption{Left: The raw EBSD data of a Magnesium sample 
		with color visualization from \cite{MTEX}.
 		Right: 	Clinched colorization to highlight details in single grains.\label{fig:ebsd}
 	}
 \end{figure*}
 
 Fig.~\ref{ictgv:fig:ebsd_IC} displays the single grain at the lower right corner of Fig. \ref{fig:ebsd}.
 Since the rotations vary little within a single grain, we treat the data as
 \(\operatorname{SO}(3)\)-valued.
 Within this single grain there occurs a so-called subgrain  boundary which should be preserved
during denoising.
 We compare results of the different IC models.
Note that we apply a different colorization for the component $w$ by placing the center of the 
colormap at the Karcher mean of the samples and using the same stretching factor for all $w$.
The denoising result of the extrinsic IC model ($\alpha = 0.06,\ \beta=\frac{1}{3}$)
are similar to those if the Midpoint IC model  $\alpha = \frac{1}{20},\ \beta = \frac{1}{3}$ and the
Lie group IC model ($\alpha = \frac{1}{20},\ \beta = \frac{1}{3}$).
The component $v$ penalized with the $\TV$ term has a piecewise constant structure, cf. Fig.~\ref{ictgv:fig:ebsd_ic:vmid} and~\subref{ictgv:fig:ebsd_ic:vgroup}, while the $w$ part is smooth, see Fig.~\ref{ictgv:fig:ebsd_ic:wmid} and~\subref{ictgv:fig:ebsd_ic:wgroup}. 
Even though, we used the same set of parameters for both approaches, we observe some differences. The $v^{\mathrm{int}}_{\IC}$-component of the mid point model has a larger jump, which is expected, as the jump should match twice the jump in the original signal. The $w^{\mathrm{int}}_{\IC}$-component on the other hand has less movement as $w^{\mathrm{Lie}}_{\IC}$ of the Lie group model. The $\TV_2^{\textrm{int}}$ regularizes stronger as the Lie group counterpart. This effect is also visible in the reconstructed images, i.e., $u^{\mathrm{Lie}}_{\IC}$ is more constant, as the other two results.

\begin{figure*}\centering
	\begin{subfigure}{0.32\textwidth}
		\centering
		\includegraphics[width=.98\textwidth]{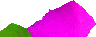}
		\caption{Noisy grain.}\label{ictgv:fig:ebsd_ic:orig}
	\end{subfigure}
	\begin{subfigure}{0.32\textwidth}
		\centering
		\includegraphics[width=.98\textwidth]{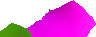}
		\caption{$u^{\mathrm{ext}}_{\IC}$.}\label{ictgv:fig:ebsd_ic:ext}
	\end{subfigure}
	
	\begin{subfigure}{0.32\textwidth}
		\centering
		\includegraphics[width=.98\textwidth]{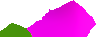}
		\caption{$u^{\mathrm{int}}_{\IC}$.}\label{ictgv:fig:ebsd_ic:umid}
	\end{subfigure}
	\begin{subfigure}{0.32\textwidth}
		\centering
		\includegraphics[width=.98\textwidth]{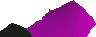}
		\caption{$v^{\mathrm{int}}_{\IC}$.}\label{ictgv:fig:ebsd_ic:vmid}
	\end{subfigure}
	\begin{subfigure}{0.32\textwidth}
		\centering
		\includegraphics[width=.98\textwidth]{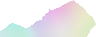}
		\caption{$w^{\mathrm{int}}_{\IC}$.}\label{ictgv:fig:ebsd_ic:wmid}
	\end{subfigure}
	
	\begin{subfigure}{0.32\textwidth}
		\centering
		\includegraphics[width=.98\textwidth]{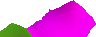}
		\caption{$u^{\mathrm{Lie}}_{\IC}$.}\label{ictgv:fig:ebsd_ic:ugroup}
	\end{subfigure}
	\begin{subfigure}{0.32\textwidth}
		\centering
		\includegraphics[width=.98\textwidth]{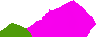}
		\caption{$v^{\mathrm{Lie}}_{\IC}$.}\label{ictgv:fig:ebsd_ic:vgroup}
	\end{subfigure}
	\begin{subfigure}{0.32\textwidth}
		\centering
		\includegraphics[width=.98\textwidth]{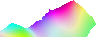}
		\caption{$w^{\mathrm{Lie}}_{\IC}$.}\label{ictgv:fig:ebsd_ic:wgroup}
	\end{subfigure}
	\caption[]{Denoising and decomposition of EBSD data of a grain with subgrain boundary by the different IC models.}
	\label{ictgv:fig:ebsd_IC}
\end{figure*}

\begin{figure*}
	\centering
	\begin{subfigure}{0.32\textwidth}	
		\centering
		\includegraphics[width = 0.98\textwidth]{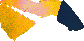}
		\caption{Noisy grain.}
	\end{subfigure}
	\hspace{.15\textwidth}
	\begin{subfigure}{0.32\textwidth}	
		\centering		
		\includegraphics[width = 0.98\textwidth]{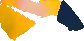}
		\caption{$u^{\mathrm{ext}}_{\TGV}$.}\label{fig:ebsd:tgv:ext}
	\end{subfigure}
	\vspace{0.2cm}
	
	\begin{subfigure}{0.32\textwidth}	
		\centering	
		\includegraphics[width = 0.98\textwidth]{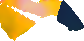}
		\caption{$u^{\mathrm{Lie}}_{\TGV}$.}\label{fig:ebsd:tgv:int}
	\end{subfigure}	
	\begin{subfigure}{0.32\textwidth}	
		\centering	
		\includegraphics[width = 0.98\textwidth]{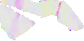}
		\caption{$a^{\mathrm{Lie}}_{1,\TGV}$.}\label{fig:ebsd:tgv:x}
	\end{subfigure}	
	\begin{subfigure}{0.32\textwidth}	
		\centering	
		\includegraphics[width = 0.98\textwidth]{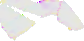}
		\caption{$a^{\mathrm{Lie}}_{2,\TGV}$.}\label{fig:ebsd:tgv:y}
	\end{subfigure}	
	\caption[]{Denoised EBSD data of a grain with subgrain boundary with the extrinsic and Lie group $\TGV$.}\label{fig:ebsd:tgv}
\end{figure*}	
	
In Fig.~\ref{fig:ebsd:tgv}, we apply the extrinsic and Lie group TGV model to the grain in the upper left corner of
Fig.~\ref{fig:ebsd}.
For the extrinsic approach we chose the parameters $(\alpha = 1.4\times10^{-2},\beta = \frac{3}{7})$
and for the Lie group one $(\alpha=0.005,\beta = \frac{2}{3})$.
Both methods lead to similar denoising results. 
However, the intrinsic Lie group TGV allows a meaningful decomposition of the “gradient”.
The components of the vector $a = (a_1,a_2)$  
are shown in Fig.~\ref{fig:ebsd:tgv:x} and~\subref{fig:ebsd:tgv:y}. 
In particular, the “compensator” $a_1$ has a jump at the subgrain boundary.

\section{Conclusions} \label{sec:Concl}
We proposed space discrete intrinsic variational models for the restoration
of manifold-valued images, where we considered three different priors, 
namely additive and IC coupling of absolute first and  second order differences
and a TGV model. 
For Lie groups, another approach was given based on the group operation.
In contrast to our general intrinsic TGV model, where gradients and their additive components are defined in tangent spaces,  
the components of the TGV Lie group approach live on the manifold.
To get a better intuition of the models one should discuss what happens if the grid mesh size goes to zero.
Or the other way around, 
a spatial continuous setting from which the current models follow by discretization is highly interesting,
but clearly out of the focus of this paper.
Note that for ${\mathcal M} = \mathbb S^1$ there exists a continuous TV approach
\cite{GMS93,GM06,GM07}.

The performance of our models 
was demonstrated by numerical examples.
A future topic consists in speeding up the computations.
In \cite{BCHPS16} we proposed for example a half-quadratic method which may be applied.

\appendix

\section{Riemannian Gradients of Differences} \label{app:grad}
\allowdisplaybreaks
\emph{Proof of Lemma} \ref{lem:grad_lie}: Using that the bi-invariant metric is invariant to inversion, i.e., 
$\dist(x,y) = \dist(x^{-1},y^{-1})$ for $x,y\in\mathcal{M}$,
we obtain
\begin{align}
&\grad_{\mathcal{M},w_{i}} (\mathrm{d}_{xy}^{\mathrm{Lie}} w)^2_{i}(w_{i}) \\
&= \grad_{\mathcal{M},w_{i}}\bigl(\dist^2(w_{i+(1,0)} \circ w_{i}^{-1}, w_{i + (1,-1)}\circ w^{-1}_{i-(0,1)})\big)\\
&= \grad_{\mathcal{M},w_{i}}\bigl(\dist^2( w_{i}^{-1},w^{-1}_{i+(1,0)} \circ w_{i + (1,-1)}\circ w^{-1}_{i-(0,1)})\big)\\
&= \grad_{\mathcal{M},w_{i}}\bigl(\dist^2( w_{i},w_{i-(0,1)} \circ w^{-1}_{i + (1,-1)}\circ w_{i+(1,0)})\big)\\
&=-2\log_{w_{i}} (w_{i-(0,1)} \circ w^{-1}_{i + (1,-1)}\circ w_{i+(1,0)}).
\end{align}
For $\xi\in T_{w_{i}}\mathcal{M}$ we obtain
\begin{align}
&\Bigl\langle\grad_{\mathcal{M},w_{i}}\bigl(\mathrm{d}_{xx}^{\mathrm{Lie}} \cdot \bigr)^2_{i}(w_{i}),\xi\Bigr\rangle_{w_{i}}\\
&=\Bigl\langle\grad_{\mathcal{M},w_{i}}\bigl(\dist^2(\cdot\circ w^{-1}_{i-(1,0)}\circ \cdot,w_{i+(1,0)}) \bigr) (w_{i}),\xi \Bigr\rangle_{w_{i}}\\
&= D \bigl(\dist^2(\cdot \circ w^{-1}_{i-(1,0)} \circ  \cdot,w_{i+(1,0)})\bigr) (w_{i})[\xi]\\
&= D \bigl(\dist^2(\cdot,w_{i+(1,0)})\bigr) \big( w_{i}\circ w^{-1}_{i-(1,0)}  \circ w_{i} \big)\\
&\quad \; D(\cdot \circ w^{-1}_{i-(1,0)}\circ\cdot)(w_{i}) [\xi]
\\
&\begin{aligned}= 
-2&\Bigl\langle\log_{w_{i}\circ w^{-1}_{i-(1,0)}\circ w_{i}}w_{i+(1,0)},
\\&\
D(\cdot \circ w^{-1}_{i-(1,0)}\circ \cdot)(w_{i}) [\xi]\Bigr\rangle_{w_{i}\circ w^{-1}_{i-(1,0)}\circ w_{i}}
\end{aligned}\\
&\begin{aligned}
= -2&\Bigl\langle\log_{w_{i}\circ w^{-1}_{i-(1,0)}\circ w_{i}}w_{i+(1,0)},\\&\
D {\mathcal L}_{w_{i}\circ w^{-1}_{i-(1,0)}} [\xi]
+
D{\mathcal R}_{w^{-1}_{i-(1,0)}\circ w_{i}}[\xi]\Bigr\rangle_{w_{i}\circ w^{-1}_{i-(1,0)} \circ w_{i}}
\end{aligned}\\
&\begin{aligned}=-2
&\Bigl\langle D{\mathcal L}_{w_{i-(1,0)}\circ w_{i}^{-1}} [\log_{w_{i}\circ w^{-1}_{i-(1,0)}\circ w_{i}}w_{i+(1,0)}]\\
&+D{\mathcal R}_{w^{-1}_{i}\circ w_{i-(1,0)}}
[\log_{w_{i}\circ w^{-1}_{i-(1,0)}\circ w_{i}}w_{i+(1,0)}],
\xi
\Bigl\rangle_{w_{i}}\!\!.
\end{aligned} \Box
\end{align}
\\[2ex]
\emph{Proof of Lemma} \ref{lem:grad_tgv}:
First, we consider the Riemannian gradient of $F_1$.
As the connection is compatible with the metric, we obtain
$\grad_{\mathcal{M},\xi_{i}} F_1$
and
$\grad_{\mathcal{M},\xi_{i}} F_2$.
For the computation of the gradients of $F_1$ with respect to $u_i,u_{i+1}$ 
we know that the outer function has gradient $T$ and obtain for $\zeta\in T_{u_i}\mathcal{M}$ 
	\begin{align}
	D \bigl(F_1 (\cdot,u_{i +1},\xi_{i})\bigr)(u_{i})[\zeta] 
	&= \bigl\langle T,\tilde L_{u_{i+1}}(u_{i})[\zeta]\bigr\rangle_{u_{i}}\\
	&=\bigl\langle \tilde L^*_{u_{i+1}}(u_{i})[T],\zeta\bigr\rangle_{u_i}\\
	&\eqqcolon \bigl\langle \grad_{\mathcal{M},u_i}\tv(\cdot,u_{i +1}),\xi_{i}\bigr)(u_{i}), \zeta\bigr\rangle_{u_i},
	\end{align}
	Similarly we can treat the derivative with respect to $u_{i+1}$ by replacing $\tilde L$ by $l$.
	Next we handle  
	\begin{align}
	&F_2(u_i,u_{i-1},\xi_i,\xi_{i-1})\\
	&=
        \lVert \xi_{i}+ \log_{u_i} \Bigl(\gamma\Bigl(\exp_{u_{i-1}} \xi_{i-1},\gamma\bigl(u_i,u_{i-1};\tfrac{1}{2}\bigr);2\Bigr)\Bigr)\Bigr\rVert_{u_{i}}^2.
	\end{align}
	To compute the differential with respect to $\xi_{i-1}$ we need to apply the differentials in the same order as they appear in the pole ladder. 
	This leads for a $\zeta \in T_{u_{i-1}}\mathcal{M}$ to 
	\begin{align}
	& D\bigl(F_2(u_i,u_{i-(1,0)},\xi_i,\cdot)\bigr)(\xi_{i-1})[\zeta] \\
	&= \Bigr\langle S,L_{u_{i}}(p_{i})
	\Bigl[G_{\cdot,c_{i},2}(e_i)
	\bigl[E_{u_{i-1}}(\xi_{i-1})[\zeta]\bigr]\Bigr]\Bigr\rangle_{u_{i}}\\
	&=\Bigr\langle E^*_{u_{i-1}}(\xi_{i-1})
	\Bigl[G^*_{\cdot,c_{i},2}(e_i)
	\bigl[L^*_{u_{i}}(p_{i})[S]\bigr]\Bigr],\zeta\Bigr\rangle_{u_{i-1}}.
	\end{align}
    As $u_i$ or $u_{i+1}$ appear twice in the pole ladder we get a sum of two differentials. 
    For $u_i$ appearing in the logarithm and the mid point evaluation we obtain for $\zeta\in T_{u_i}\mathcal{M}$, 
    \begin{align}    
	    & D \bigl(F_2(\cdot,u_{i -1},\xi_{i},\xi_{i-1})\bigr)(u_{i})[\zeta]\\
	    &= \biggr\langle S,
	    \tilde{L}_{p_{i}}(u_{i})[\zeta]
	    \\&\qquad+L_{u_{i}}(p_{i})
	    \Bigl[G_{e_i,\cdot,2}(c_{i})
	    \bigl[G_{\cdot,u_{i-1},\frac{1}{2}}(u_{i})[\zeta]\bigr]\Bigr]\biggr\rangle_{u_{i}}\\
	    &=\biggr\langle \tilde{L}^*_{p_{i}}(u_{i})[S]	
	    \\&\qquad+G^*_{\cdot,u_{i-1},\frac{1}{2}}(u_{i})
	    \Bigl[G^*_{e_i,\cdot,2}(c_{i})
	    \bigl[L^*_{u_{i}}(p_{i})
	    [S]\bigr]\Bigr],\zeta\biggr\rangle_{u_{i}}.
    \end{align}
    Similarly, we conclude for $u_{i-1}$ and $\zeta\in T_{u_{i-1}}\mathcal{M}$,
    \begin{align}  
	    & D\bigl(F_2(u_{i},\cdot,\xi_{i},\xi_{i-1})\bigr)(u_{i -1})[\zeta] \\
	    &= \Bigr\langle S,L_{u_{i}}(p_{i})\bigl[G_{\cdot,c_{i},2}(e_{i})\bigl[\tilde{E}_{\xi_{i-1}}(u_{i-1})[\zeta]\bigl]\Bigr]	
	    \\&\qquad+
	    L_{u_{i}}(p_{i})
	    \Bigl[G_{e_i,\cdot,2}(c_{i})
	    \bigl[G_{u_{i},\cdot,\frac{1}{2}}(u_{i-1})[\zeta]\bigr]\Bigr]\Bigr\rangle_{u_{i}}\\
	    &=\Bigr\langle \tilde E_{\xi_{i-1}}^*(u_{i-1})\Bigl[
	    G^*_{\cdot,c_{i},2}(e_{i})\bigl[
	    L^*_{u_{i}}(p_{i})[S]\bigr]\Bigr]
	    \\&
	    \qquad
	    + G^*_{u_i,\cdot,\frac{1}{2}}(u_{i-1})
	    \Bigl[G^*_{e_i,\cdot,2}(c_{i})
	    \bigr[L^*_{u_{i}}(p_{i})[S]\bigr]\Bigr]
	    ,v\Bigr\rangle_{u_{i}}. \qquad \Box
    \end{align}

\section{Special Manifolds} \label{sec:app}
%
\subsection{The \texorpdfstring{$d$}{d}-dimensional Sphere}
Let~$\mathbb S^{d} = \bigl \{  x\in \R^{d+1}\colon \lVert x\rVert_{2}=1\bigr\}$
denote the \(d\)-dimensional unit sphere embedded in \(\mathbb R^{d+1}\).
The tangential space at $ x\in\mathbb S^{d}$ is given by
\[
T_{ x} \mathbb S^{d}=\bigl\{\xi \in\R^{d+1}: \langle {x} ,\xi \rangle = 0\bigr\}.
\]
A Riemannian metric is the metric from the embedding space, i.e., the Euclidean
inner product.
The geodesic distance related to this metric is given by
\begin{equation*}
\dist( x, y) = \arccos\langle x, y\rangle,
\end{equation*}
where $\langle\cdot,\cdot\rangle$ is the standard scalar product in $\R^{d+1}$.
The geodesic
$\gamma_{ x, \xi}(t)$ with
$\gamma_{ x, \xi}(0) =  x$ and $\dot \gamma _{x, \xi}(0) = \xi$
is given by
\begin{equation}
\gamma_{ x, \xi}(t) = \cos(t\| \xi\|_2)  x + \sin(t\| \xi\|_2) \frac{\xi}{\| \xi\|_2}.
\end{equation}
The exponential and logarithmic map read as
\begin{align*}
\exp_{x} (\xi) &= x \cos\bigl(\lVert \xi \rVert\bigr)+\frac{\xi}{\lVert \xi \rVert}\sin\bigl(\lVert \xi\rVert\bigr),\\
\log_{x} (y) &=  \dist_{\mathbb S^{d}}( x, y) \, \frac{{ y}-\langle x,y
\rangle x}{\lVert  y-\langle  x, y \rangle x\rVert},
\quad  x \not = - y.
\end{align*}
The orthogonal projection of $x \in \mathbb R^{d+1}$ onto $\mathbb S^d$ is given by
$\Pi(x) = x/\|x\|_2$.
The parallel transport \[P_{x\to y}\colon T_x{\mathcal S}^d(r)\to T_y{\mathcal S}^d\] 
along the geodesic from $x$ to $y$ is given by, see e.g.~\cite{HU17},
	\begin{equation}
		P_{x\to y}(\xi) = \xi-\frac{\bigl\langle\log_x(y),\xi\bigr\rangle}{\dist^2_{\SS^d}(x,y)}\bigl(\log_x(y)+\log_y(x)\bigr).
	\end{equation}	 

\subsection{The special orthogonal group}
Let $\operatorname{SO}(3)=\{x\in\mathbb{R}^{3,3}:x^\tT x = I_3,\det(x)=1\}$, 
be the space of rotations in $\mathbb R^3$. The tangent space at $x \in \operatorname{SO}(3)$ is
$
T_x\operatorname{SO}(3) = x \operatorname{Skew}(3),
$
with $\operatorname{Skew}(3) = \{x\in\mathbb{R}^{3,3}:x^\tT +x = 0\}$.
It is a Lie group with bi-invariant metric and geodesic distance
\begin{equation}
\dist_{\operatorname{SO}(3)}(x,y) = \sqrt{2}\arccos\Bigl(\frac{\operatorname{tr}(x^\tT y)-1}{2}\Bigr).
\end{equation}
An isometric representation of the rotations in $\mathbb{R}^3$ 
is given by the unit quaternions, see \cite{Graef12}: 
for $p_1,p_2\in\mathbb R ^4$, $p_1= (s_1,v_1)^\tT,p_2= (s_2,v_2)^\tT,\ v_1,v_2\in\mathbb R^3$, the multiplication is defined by
\begin{equation}
p_1\circ p_2 = \begin{pmatrix}
s_1s_2-v_1^\tT v_2\\
s_1v_2+s_2v_1+v_1\times v_2
\end{pmatrix},
\end{equation}
the unit element is $e = (1,0,0,0)^\tT$ and the inverse is given by \[p^{-1} = (p_1,-p_2,-p_3,-p_4).\]
A rotation of a vector $x\in\mathbb R^3$ around the angle $\alpha\in(0,\pi]$ and axis $r\in\mathbb S^2$ can be realized with
\begin{equation}
p(\alpha,r) \coloneqq \begin{pmatrix}
\cos(\frac{\alpha}{2})\\
\sin(\frac{\alpha}{2})r
\end{pmatrix}, p(\alpha,r)\circ\begin{pmatrix}
0\\x
\end{pmatrix}\circ p(\alpha,r)^{-1} = \operatorname{rot}(\alpha,r).
\end{equation}
Note that $p(\alpha,r)\in\mathbb{S}^3$, further $p(\alpha_1,r_1)\circ p(\alpha_2,r_2)\in\mathbb{S}^3$, 
so the rotations can be identified with elements on the sphere $\mathbb{S}^3$. As $p$ and $-p$ yield the same rotation, 
we have a bijection between $\operatorname{SO}(3)$ and $\mathbb S^3\slash\{-1,1\}$. 
Furthermore $(\operatorname{SO}(3),\dist_{\operatorname{SO}(3)})$ 
is isometric to $(\mathbb S^3\slash\{-1,1\},\sqrt{2}\dist_{\mathbb S^3\slash\{-1,1\}})$, with
\begin{equation*}
\dist_{\mathbb S^3\slash\{-1,1\}}(p,q) = \arccos\lvert\langle p,q\rangle\rvert.
\end{equation*}
The exponential map, logarithmic map, and the projection on $\mathbb S^3$ can be used, with a few adjustments. 
The result of the exponential map and the projection is chosen, such that the first entry is positive. 
For the computation of the logarithmic map $\log_p q$, we chose the representation of $q$ having the smallest distance to $p$.

\subsection{Symmetric positive definite matrices}
The dimension of the
manifold~$\SPD{r}$ of symmetric positive definite matrices
is $d = \frac{r(r+1)}{2}$.
Then the affine invariant geodesic distance is given by
\[
\dist_{\SPD{r}}( x,y) = \bigl\lVert \mathrm{Log}( x^{-\frac{1}{2}} { y}  x^{-\frac{1}{2}} )\bigr\rVert_{\mathrm{F}},
\]
where $\lVert \cdot \rVert_{\mathrm{F}}$ denotes the Frobenius norm of matrices and $\mathrm{Exp}$ and $\mathrm{Log}$ denote the matrix exponential
and logarithm, respectively.
The tangential space at $ x\in \SPD{r}$ is given by
\[
  T_{x}\SPD{r} =
  \{ x^{\frac{1}{2}}\xi x^{\frac{1}{2}} : \xi\in\operatorname{Sym}(r) \}
  = \operatorname{Sym}(r),
\]
where  $\operatorname{Sym}(r)$ denotes the space of
symmetric $r \times r$ matrices. The Riemannian metric reads
\[
\langle \xi_1,\xi_2 \rangle_{x} = \mathrm{tr} (\xi_1  x^{-1} \xi_2  x^{-1}),
\quad \xi_1,\xi_2\in T_x\SPD{r}.
\]
The exponential and the logarithmic map are
\begin{align} \label{exp_spd}
\exp_{ p} (\xi) &=  p^{\frac{1}{2}} \mathrm{Exp}\bigl(  p^{-\frac{1}{2}} \xi  p^{-\frac{1}{2}}\bigr)  p^{\frac{1}{2}},\\
\log_{ p}({ q}) &=  p^{\frac{1}{2}}\mathrm{Log}\bigl( p^{-\frac{1}{2}} \, { q} \,  p^{-\frac{1}{2}}\bigr)
 p^{\frac{1}{2}}.
\end{align}
We embed the manifold of symmetric positive definite matrices \(\SPD{r}\)
into~\(\mathbb R^n\), \(n = \frac{r(r+1)}{2}\), using the canonical embedding
of the upper triangular matrix. Then the projection onto the closure of the
manifold \(\SPD{r}\) is given as follows:
let \(x=u\Lambda u^\tT\) denote the eigenvalue decomposition of an real-valued
symmetric matrix \(x\in\mathbb R^{r,r}\) represented as before by its upper
triangular entries as a vector in \(\mathbb R^n\). Hence \(u\) is an orthogonal
matrix, and \(\Lambda = \operatorname{diag}(\lambda_1,\ldots,\lambda_r)\) is
the diagonal matrix of the eigenvalues of \(x\).
The projection is then given by
\begin{align*}
  \Pi(x) = u\tilde\Lambda u^\tT,
  \quad
  \tilde\Lambda \coloneqq \operatorname{diag}(\tilde\lambda_1,\ldots,\tilde\lambda_r),
 \quad
  \tilde\lambda_i \coloneqq \max\{0,\lambda_i\}.
\end{align*}
The parallel transport \[P_{x\to y}\colon T_x\SPD{r}\to T_y\SPD{r}\] along the geodesic from $x$ to $y$ is given by
	\begin{align}
	&P_{x\to y}(\xi) = \geo{x,y}(\tfrac{1}{2}) x^{-1}\xi x^{-1}\geo{x,y}(\tfrac{1}{2}).
	\end{align}

\begin{acknowledgement}
        R. Bergmann wants to thank B. Wirth (University of M\"unster) for fruitful discussions on Schild's ladder TGV.
	Funding by the German Research Foundation (DFG) within the project STE 571/13-1 \& BE 5888/2-1
	and with\-in the Research Training Group 1932,
	project area P3, is gratefully acknowledged.	
\end{acknowledgement}
%
\bibliographystyle{abbrv}
\bibliography{InfConvReferences}

\begin{thebibliography}{10}

\bibitem{AMS08}
P.-A. Absil, R.~Mahony, and R.~Sepulchre.
\newblock {\em Optimization Algorithms on Matrix Manifolds}.
\newblock Princeton University Press, Princeton and Oxford, 2008.

\bibitem{Alouges97}
F.~Alouges.
\newblock A new algorithm for computing liquid crystal stable configurations:
  The harmonic mapping case.
\newblock {\em SIAM Journal on Numerical Analysis}, 34(5):1708--1726, 1997.

\bibitem{APA06}
V.~Arsigny, X.~Pennec, and N.~Ayache.
\newblock Bi-invariant means in {L}ie groups. application to left-invariant
  polyaffine transformations.
\newblock {\em HAL Preprint}, 00071383, 2006.

\bibitem{ABS13}
H.~Attouch, J.~Bolte, and B.~F. Svaiter.
\newblock Convergence of descent methods for semi-algebraic and tame problems:
  proximal algorithms, forward--backward splitting, and regularized
  {G}auss--{S}eidel methods.
\newblock {\em Mathematical Programming}, 137(1):91--129, 2013.

\bibitem{Bac14}
M.~Ba{\v{c}}{\'a}k.
\newblock {\em Convex analysis and optimization in {H}adamard spaces},
  volume~22 of {\em De Gruyter Series in Nonlinear Analysis and Applications}.
\newblock De Gruyter, Berlin, 2014.

\bibitem{BBSW16}
M.~Ba{\v c}{\'a}k, R.~Bergmann, G.~Steidl, and A.~Weinmann.
\newblock A second order non-smooth variational model for restoring
  manifold-valued images.
\newblock {\em SIAM Journal on Scientific Computing}, 38(1):A567--A597, 2016.

\bibitem{MTEX}
F.~Bachmann and R.~Hielscher.
\newblock {MTEX} -- {MATLAB} toolbox for quantitative texture analysis.
\newblock \url{http://mtex-toolbox.github.io/}, 2005--2016.

\bibitem{BHJPSW10}
F.~Bachmann, R.~Hielscher, P.~E. Jupp, W.~Pantleon, H.~Schaeben, and E.~Wegert.
\newblock Inferential statistics of electron backscatter diffraction data from
  within individual crystalline grains.
\newblock {\em Journal of Applied Crystallography}, 43:1338--1355, 2010.

\bibitem{BBEFSS18}
F.~Balle, T.~Beck, D.~Eifler, J.~H. Fitschen, S.~Schuff, and G.~Steidl.
\newblock Strain analysis by a total generalized variation regularized optical
  flow model.
\newblock {\em Inverse Problems in Science \& Engineering}, 2017.
\newblock accepted with minor revision.

\bibitem{BEFSS15}
F.~Balle, D.~Eifler, J.~H. Fitschen, S.~Schuff, and G.~Steidl.
\newblock Computation and visualization of local deformation for multiphase
  metallic materials by infimal convolution of {TV}-type functionals.
\newblock In {\em SSVM 2015}, Lecture Notes in Computer Science, pages
  385--396. Springer, 2015.

\bibitem{BH98}
R.~Bamler and P.~Hartl.
\newblock Synthetic aperture radar interferometry.
\newblock {\em Inverse Problems}, 14(4):R1--R54, 1998.

\bibitem{BCHPS16}
R.~Bergmann, R.~H. Chan, R.~Hielscher, J.~Persch, and G.~Steidl.
\newblock Restoration of manifold-valued images by half-quadratic minimization.
\newblock {\em Inverse Problems and Imaging}, 10(2):281--304, 2016.

\bibitem{BFPS17}
R.~Bergmann, J.~H. Fitschen, J.~Persch, and G.~Steidl.
\newblock Infimal convolution coupling of first and second order differences on
  manifold-valued images.
\newblock In F.~Lauze, Y.~Dong, and A.~B. Dahl, editors, {\em Scale Space and
  Variational Methods in Computer Vision: 6th International Conference, SSVM
  2017, Kolding, Denmark, June 4-8, 2017, Proceedings}, pages 447--459.
  Springer International Publishing, Cham, 2017.

\bibitem{BLSW14}
R.~Bergmann, F.~Laus, G.~Steidl, and A.~Weinmann.
\newblock Second order differences of cyclic data and applications in
  variational denoising.
\newblock {\em SIAM Journal on Imaging Sciences}, 7(4):2916–2953, 2014.

\bibitem{BT17}
R.~Bergmann and D.~Tenbrinck.
\newblock A graph framework for manifold-valued data.
\newblock {\em arXiv Preprint 1702.05293}, 2017.

\bibitem{BW15}
R.~Bergmann and A.~Weinmann.
\newblock Inpainting of cyclic data using first and second order differences.
\newblock In {\em Energy Minimization Methods in Computer Vision and Pattern
  Recognition}, pages 155--168. Springer, 2015.

\bibitem{BW16}
R.~Bergmann and A.~Weinmann.
\newblock A second order {TV}-type approach for inpainting and denoising higher
  dimensional combined cyclic and vector space data.
\newblock {\em Journal of Mathematical Imaging and Vision}, 55(3):401--427,
  2016.

\bibitem{Bre14}
K.~Bredies.
\newblock Recovering piecewise smooth multichannel images by minimization of
  convex functionals with total generalized variation penalty.
\newblock In A.~Bruhn, T.~Pock, and X.-C. Tai, editors, {\em Efficient
  Algorithms for Global Optimization Methods in Computer Vision}, pages 44--77.
  Springer, 2014.

\bibitem{BH14}
K.~Bredies and M.~Holler.
\newblock Regularization of linear inverse problems with total generalized
  variation.
\newblock {\em Journal of Inverse and Ill-posed Problems}, 22(6):871--913,
  2014.

\bibitem{BHSW17}
K.~Bredies, M.~Holler, M.~Storath, and A.~Weinmann.
\newblock Total generalized variation for manifold-valued data.
\newblock {\em Preprint arXiv:1709.01616}, 2017.

\bibitem{BKP10}
K.~Bredies, K.~Kunisch, and T.~Pock.
\newblock Total generalized variation.
\newblock {\em SIAM Journal on Imaging Sciences}, 3(3):492--526, 2010.

\bibitem{BH15}
K.~Bredies and H.~P. Sun.
\newblock Preconditioned {D}ouglas--{R}achford algorithms for {TV}- and
  {TGV}-regularized variational imaging problems.
\newblock {\em Journal of Mathematical Imaging and Vision}, 52(3):317--344, Jul
  2015.

\bibitem{BV11}
K.~Bredies and T.~Valkonen.
\newblock Inverse problems with second-order total generalized variation
  constraints.
\newblock In {\em International Conference on Sampling Theory and
  Applications}, 2011.

\bibitem{BSS16}
M.~Burger, A.~Sawatzky, and G.~Steidl.
\newblock First order algorithms in variational image processing.
\newblock In R.~Glowinski, S.~Osher, and W.~Yin, editors, {\em Operator
  Splittings and Alternating Direction Methods}. Springer, 2016.

\bibitem{BRF00}
R.~B{\"u}rgmann, P.~A. Rosen, and E.~J. Fielding.
\newblock Synthetic aperture radar interferometry to measure
  earth{\textquoteright}s surface topography and its deformation.
\newblock {\em Annu. Rev. Earth Planet. Sci.}, 28(1):169--209, 2000.

\bibitem{CL97}
A.~Chambolle and P.-L. Lions.
\newblock Image recovery via total variation minimization and related problems.
\newblock {\em Numerische Mathematik}, 76(2):167--188, 1997.

\bibitem{CS13}
D.~Cremers and E.~Strekalovskiy.
\newblock Total cyclic variation and generalizations.
\newblock {\em Journal of Mathematical Imaging and Vision}, 47(3):258--277,
  2013.

\bibitem{Carmo92}
M.~P. do~Carmo.
\newblock {\em Riemannian Geometry}, volume 115.
\newblock Birkh\"auser, Basel, 1992.
\newblock Tranlated by F.~Flatherty.

\bibitem{EPS72}
J.~Ehlers, F.~A.~E. Pirani, and A.~Schild.
\newblock The geometry of free fall and light propagation.
\newblock In L.~O’Reifeartaigh, editor, {\em General Relativitiy}, pages
  63--84. Oxford University Press, 1972.

\bibitem{Fi17}
J.~H. Fitschen.
\newblock {\em Variational Models in Image Processing with Applications in the
  Materials Sciences}.
\newblock Dissertation, University of Kaiserslautern, 2017.
\newblock Similarily: Verlag Dr.~Hut, ISBN 978-3843932455, 2017.

\bibitem{FJ07}
P.~Fletcher and S.~Joshi.
\newblock {R}iemannian geometry for the statistical analysis of diffusion
  tensor data.
\newblock {\em Signal Processing}, 87:250--262, 2007.

\bibitem{GM76}
D.~Gabay and B.~Mercier.
\newblock A dual algorithm for the solution of nonlinear variational problems
  via finite element approximations.
\newblock {\em Computer and Mathematics with Applications}, 2:17--40, 1976.

\bibitem{GQ17}
J.~Gallier and J.~Quaintance.
\newblock Notes on {D}ifferential {G}eometry and {L}ie {G}roups, 2017.

\bibitem{GMS93}
M.~Giaquinta, G.~Modica, and J.~Sou{\v c}ek.
\newblock Variational problems for maps of bounded variation with values in
  {$S^1$}.
\newblock {\em Calculus of Variation}, 1(1):87--121, 1993.

\bibitem{GM06}
M.~Giaquinta and D.~Mucci.
\newblock The {BV}-energy of maps into a manifold: relaxation and density
  results.
\newblock {\em Ann. Sc. Norm. Super. Pisa Cl. Sci.}, 5(4):483--548, 2006.

\bibitem{GM07}
M.~Giaquinta and D.~Mucci.
\newblock Maps of bounded variation with values into a manifold: total
  variation and relaxed energy.
\newblock {\em Pure and Applied Mathematics Quarterly}, 3(2):513--538, 2007.

\bibitem{GM75}
R.~Glowinski and A.~Marroco.
\newblock Sur l'approximation, par \'el\'ements finis d'ordre un, et la
  r\'esolution, par p\'enalisation-dualit\'e d'une classe de probl\`emes de
  {D}irichlet non lin\'eaires.
\newblock {\em Revue fran\c{c}aise d'automatique, informatique, recherche
  op\'erationnelle. Analyse num\'erique}, 9(2):41--76, 1975.

\bibitem{Graef12}
M.~Gr{\"a}f.
\newblock A unified approach to scattered data approximation on
  {$\mathbb{S}^{3}$} and {SO}(3).
\newblock {\em Advances in Computational Mathematics}, 37(3):379--392, 2012.

\bibitem{GA10}
V.~K. Gupta and S.~R. Agnew.
\newblock A simple algorithm to eliminate ambiguities in ebsd orientation map
  visualization and analyses: Application to fatigue crack-tips/wakes in
  aluminum alloys.
\newblock {\em Microscopy and Microanalysis}, 16:831, 2010.

\bibitem{HK14}
M.~Holler and K.~Kunisch.
\newblock On infimal convolution of {TV}-type functionals and applications to
  video and image reconstruction.
\newblock {\em SIAM Journal on Imaging Sciences}, 7(4):2258--2300, 2014.

\bibitem{HR31}
H.~Hopf and W.~Rinow.
\newblock Ueber den begriff der vollst{\"a}ndigen differentialgeometrischen
  fl{\"a}che.
\newblock {\em Commentarii Mathematici Helvetici}, 3(1):209--225, Dec 1931.

\bibitem{HU17}
S.~Hosseini and A.~Uschmajew.
\newblock A {R}iemannian gradient sampling algorithm for nonsmooth optimization
  on manifolds.
\newblock {\em SIAM Journal on Optimization}, 27(1):173--189, 2017.

\bibitem{Jarre2000}
F.~Jarre.
\newblock Convex analysis on symmetric matrices.
\newblock In H.~Wolkowicz, R.~Saigal, and L.~Vandenberghe, editors, {\em
  Handbook of Semidefinite Programming}. Kluwer Academic Publishers, 2000.

\bibitem{KMN2000}
A.~Kheyfets, W.~A. Miller, and G.~A. Newton.
\newblock Schild's ladder parallel transport procedure for an arbitrary
  connection.
\newblock {\em International Journal of Theoretical Physics},
  39(12):2891--2898, Dec 2000.

\bibitem{LNPS17}
F.~Laus, M.~Nikolova, J.~Persch, and G.~Steidl.
\newblock A nonlocal denoising algorithm for manifold-valued images using
  second order statistics.
\newblock {\em SIAM Journal on Imaging Sciences}, 10(1):416–448, March 2017.

\bibitem{LSKC13}
J.~Lellmann, E.~Strekalovskiy, S.~Koetter, and D.~Cremers.
\newblock Total variation regularization for functions with values in a
  manifold.
\newblock In {\em IEEE ICCV 2013}, pages 2944--2951, 2013.

\bibitem{LP14}
G.~Li and T.~K. Pong.
\newblock Global convergence of splitting methods for nonconvex composite
  optimization.
\newblock {\em Preprint arXiv: 1407.0753}, 753, 2014.

\bibitem{LP14Sch}
M.~Lorenzi and X.~Pennec.
\newblock Efficient parallel transport of deformations in time series of
  images: From {S}child's to pole ladder.
\newblock {\em Journal of Mathematical Imaging and Vision}, 50(1):5--17, Sep
  2014.

\bibitem{Nas56}
J.~Nash.
\newblock The imbedding problem for {R}iemannian manifolds.
\newblock {\em Annals of Mathematics}, 63(1):20--63, 1956.

\bibitem{PS13}
K.~Papafitsoros and C.~B. Sch{\"o}nlieb.
\newblock A combined first and second order variational approach for image
  reconstruction.
\newblock {\em Journal of Mathematical Imaging and Vision}, 2(48):308--338,
  2014.

\bibitem{Pen18}
X.~Pennec.
\newblock Pole ladder: an exact scheme for parallel transport in symmetric
  spaces.
\newblock {\em In preparation}, 2018.

\bibitem{persch2018}
J.~Persch.
\newblock Optimization methods in manifold-valued image processing.
\newblock {\em PhD Thesis, TU Kaiserslautern}, 2018.

\bibitem{Rent11}
Q.~Rentmeesters.
\newblock A gradient method for geodesic data fitting on some symmetric
  {R}iemannian manifolds.
\newblock In {\em 50th IEEE Conference on Decision and Control and European
  Control Conference 2011}, pages 7141--7146, 2011.

\bibitem{Ro70}
R.~T. Rockafellar.
\newblock {\em Convex Analysis}.
\newblock Princeton University Press, 1970.

\bibitem{RTKB14}
G.~Rosman, X.-C. Tai, R.~Kimmel, and A.~M. Bruckstein.
\newblock Augmented-{L}agrangian regularization of matrix-valued maps.
\newblock {\em Methods and Applications of Analysis}, 21(1):121--138, 2014.

\bibitem{RWTKB14}
G.~Rosman, Y.~Wang, X.-C. Tai, R.~Kimmel, and A.~M. Bruckstein.
\newblock Fast regularization of matrix-valued images.
\newblock In {\em Efficient Algorithms for Global Optimization Methods in
  Computer Vision}, pages 19--43. Springer, 2014.

\bibitem{Rossmann03}
W.~Rossmann.
\newblock {\em Lie Groups}.
\newblock Oxford Science Publications, Oxford, 2003.

\bibitem{ROF92}
L.~I. Rudin, S.~Osher, and E.~Fatemi.
\newblock Nonlinear total variation based noise removal algorithms.
\newblock {\em Physica D: Nonlinear Phenomena}, 60(1):259--268, 1992.

\bibitem{Sas1958}
S.~Sasaki.
\newblock On the differential geometry of tangent bundles of {R}iemannian
  manifolds.
\newblock {\em Tohoku Mathematical Journal, Second Series}, 10(3):338--354,
  1958.

\bibitem{SS08}
S.~Setzer and G.~Steidl.
\newblock Variational methods with higher order derivatives in image
  processing.
\newblock In {\em Approximation XII: San Antonio 2007}, pages 360--385, 2008.

\bibitem{SST11}
S.~Setzer, G.~Steidl, and T.~Teuber.
\newblock Infimal convolution regularizations with discrete \(\ell_1\)-type
  functionals.
\newblock {\em Communications in Mathematical Sciences}, 9(3):797--827, 2011.

\bibitem{SSPB07}
G.~Steidl, S.~Setzer, B.~Popilka, and B.~Burgeth.
\newblock Restoration of matrix fields by second order cone programming.
\newblock {\em Computing}, 81:161--178, 2007.

\bibitem{SC11}
E.~Strekalovskiy and D.~Cremers.
\newblock Total variation for cyclic structures: convex relaxation and
  efficient minimization.
\newblock In {\em 2011 IEEE Conference on Computer Vision and Pattern
  Recognition (CVPR)}, pages 1905--1911, 2011.

\bibitem{SAK00}
S.~Sun, B.~Adams, and W.~King.
\newblock Observation of lattice curvature near the interface of a deformed
  aluminium bicrystal.
\newblock {\em Phil. Mag. A}, 80:9--25, 2000.

\bibitem{VBK13}
T.~Valkonen, K.~Bredies, and F.~Knoll.
\newblock Total generalized variation in diffusion tensor imaging.
\newblock {\em SIAM Journal on Imaging Sciences}, 6(1):487--525, 2013.

\bibitem{WYZ15}
Y.~Wang, W.~Yin, and J.~Zeng.
\newblock Global convergence of {ADMM} in nonconvex nonsmooth optimization.
\newblock {\em ArXiv preprint 1511.06324}, 2015.

\bibitem{WDS14}
A.~Weinmann, L.~Demaret, and M.~Storath.
\newblock Total variation regularization for manifold-valued data.
\newblock {\em SIAM Journal on Imaging Sciences}, 7(4):2226--2257, 2014.

\bibitem{Whi36}
H.~Whitney.
\newblock Differentiable manifolds.
\newblock {\em Annals of Mathematics}, 37(3):645--680, 1936.

\end{thebibliography}
%
%

\end{document}